\documentclass[reqno,oneside,10pt]{amsart}
\usepackage{fancyhdr}
\usepackage{xspace}
\usepackage{geometry}
\usepackage{tikz}

\tikzset{node distance=1.5cm, auto}
\usepackage{color}
\definecolor{darkgreen}{rgb}{0,0.45,0}
\usepackage[pagebackref,colorlinks,citecolor=darkgreen,linkcolor=darkgreen]{hyperref}

\usepackage{graphicx}
\usepackage{enumerate}
\usepackage[all]{xy}
\usepackage{amsmath,amsfonts,amssymb}
\usepackage[latin9]{inputenc}

\def\ul{\underline}
\newcommand{\Hom}{{\sf Hom}}
\renewcommand{\P}{{\sf P}}
\newcommand{\DP}{{\sf DP}}
\newcommand{\LDP}{{\sf LDP}}
\newcommand{\RDP}{{\sf RDP}}

\newcommand{\End}{{\sf End}}

\newcommand{\Ker}{{\sf Ker}\,}

\newcommand{\im}{{\sf Im}\,}

\newcommand{\Mod}{{\sf Mod}}

\newcommand{\Vect}{{\sf Vect}}

\newcommand{\ev}{{\sf ev}}

\def\ot{\otimes}
\def\bul{\bullet}

\newcommand{\bk}[1]{\left<#1\right>}
\newcommand{\bkk}[1]{\left<\!\left< #1\right>\!\right>}
\newcommand{\bkkk}[1]{\left<\!\left<\!\left< #1\right>\!\right>\!\right>}
\newcommand{\ol}[1]{\overline{#1}}

\def\NN{{\mathbb N}}

\def\VV{{\mathbb V}}

\def\ZZ{{\mathbb Z}}

\newcommand{\Cc}{\mathcal{C}}

\newcommand{\Gg}{\mathcal{G}}
\newcommand{\Hh}{\mathcal{H}}

\newcommand{\Mm}{\mathcal{M}}

\newcommand{\Oo}{\mathcal{O}}

\newcommand{\thlabel}[1]{\label{th:#1}}
\newcommand{\thref}[1]{Theorem~\ref{th:#1}}
\newcommand{\selabel}[1]{\label{se:#1}}
\newcommand{\seref}[1]{Section~\ref{se:#1}}
\newcommand{\lelabel}[1]{\label{le:#1}}
\newcommand{\leref}[1]{Lemma~\ref{le:#1}}

\newcommand{\relabel}[1]{\label{re:#1}}

\newtheorem{prop}{Proposition}[section]
\newtheorem{proposition}[prop]{Proposition}
\newtheorem{lemma}[prop]{Lemma} 
 
\newtheorem{corollary}[prop]{Corollary} 

\newtheorem{thm}[prop]{Theorem}
\newtheorem{theorem}[prop]{Theorem}

\theoremstyle{definition}

\newtheorem{defi}[prop]{Definition}
\newtheorem{definition}[prop]{Definition}
\newtheorem{exmp}[prop]{Example}
\newtheorem{example}[prop]{Example}
\newtheorem{examples}[prop]{Examples}
\newtheorem{rmk}[prop]{Remark}
\newtheorem{remark}[prop]{Remark}
\newtheorem{remarks}[prop]{Remarks}

\newcommand{\cosmash}{>\!\!\blacktriangleleft}
\newcommand{\rcosmash}{\blacktriangleright\!\! <}

\newcommand{\um }{1_A}

\newcommand{\ovomega}{\overline{\omega}_0}
\newcommand{\ove}{\overline{E}}
\newcommand{\ovesigma}{\overline{E}\circ\sigma}

\newcommand{\benu}{\begin{enumerate}}
\newcommand{\enu}{\end{enumerate}}

\newcommand{\beqna}{\begin{eqnarray}}
\newcommand{\eqna}{\end{eqnarray}}
\newcommand{\beqnast}{\begin{eqnarray*}}
\newcommand{\eqnast}{\end{eqnarray*}}
\newcommand{\beqn}{\begin{equation}}
\newcommand{\eqn}{\end{equation}}
\newcommand{\beqnst}{\begin{equation*}}
\newcommand{\eqnst}{\end{equation*}}

\usepackage{latexsym}

\usepackage{xspace}
\usepackage{amscd}
\usepackage{amssymb}
\usepackage{amsfonts}

\makeatother

\newcommand{\bema}{\left ( \begin{array}}
\newcommand{\ema}{\end{array} \right )}

\usepackage[shortlabels]{enumitem}
\setlist{
	labelindent	=\parindent,
	leftmargin	=*,
	font		=\normalfont
}
\setlist[enumerate]{
	label		=(\arabic*),
	ref			=(\arabic*)
}
\allowdisplaybreaks
\begin{document}

\title[Partial corepresentations of Hopf Algebras]{Partial corepresentations of Hopf Algebras}

\author[M.M.S. Alves]{Marcelo \ Muniz \ S. \ Alves}
\address{Departamento de Matem\'atica, Universidade Federal do Paran{\'a}, Brazil}
\email{marcelo@mat.ufpr.br}
\author[E. Batista]{Eliezer Batista}
\address{Departamento de Matem\'atica, Universidade Federal de Santa Catarina, Brazil}
\email{ebatista@mtm.ufsc.br}
\author[F. Castro]{Felipe Castro}
\address{Departamento de Matem{\'a}tica, Universidade Federal de Santa Catarina, Brazil}
\email{f.castro@ufsc.br}
\author[G. Quadros]{Glauber Quadros}
\address{Departamento de matem{\'a}tica, Universidade Federal de Santa Maria, Brazil}
\email{glauber.quadros@ufsm.br}
\author[J. Vercruysse]{Joost Vercruysse}
\address{D\'epartement de Math\'ematiques, Universit\'e Libre de Bruxelles, Belgium}
\email{jvercruy@ulb.ac.be}

\flushbottom

\begin{abstract} We introduce the notion of a partial corepresentation of a given Hopf algebra $H$ over a coalgebra $C$ and the closely related concept of a partial $H$-comodule. We prove that there exists a universal coalgebra $H^{par}$, associated to the original Hopf algebra $H$, such that the category of regular partial $H$-comodules is isomorphic to the category of $H^{par}$-comodules. We introduce the notion of a Hopf coalgebroid and show that the universal coalgebra $H^{par}$ has the structure of a Hopf coalgebroid over a suitable coalgebra.
\end{abstract}

\maketitle

\tableofcontents

\section{Introduction}

From the very beginning, the theory of Hopf algebras relied heavily on the notion of duality. 
Indeed, Hopf algebras originate from the study of the cohomology ring of group manifolds \cite{Hopf}. The group operations induce by duality an additional set of new operations on this ring, being the comultiplication, counit and antipode. 
As a consequence, the notion of a Hopf algebra becomes self-dual, meaning that a Hopf algebra object in the opposite category of the category of vector spaces is just a usual Hopf algebra. 

This dual behavior has shown to play an important role in the representation theory of Hopf algebras. Indeed, the coalgebra structure of a Hopf algebra $H$ induces a monoidal structure on its category of modules $\Mod_H$. In the same way, the algebra structure of a $H$ induces a monoidal structure on its category of comodules $\Mod^H$. The Tannaka-Krein principle moreover allows to reconstruct a Hopf algebra from a suitable (rigid) monoidal category with fibre functor. The Hopf algebra which is in this way reconstructed from the category of finite dimensional $H$-comodules $\Mod^H_{fd}$ is exactly the original Hopf algebra $H$. However, remark that the Hopf algebra which is reconstructed from the category of finite dimensional $H$-modules is exactly the Sweedler dual $H^\circ$, which is in general different from $H$, but which is such that the category of finite dimensional $H$-modules is isomorphic to the category of finite-dimensional $H^\circ$-comodules. 

The above motivates that any theory (of representation or actions) of Hopf algebras should be studied in correspondence with its dual theory (of corepresentations or coactions). A second reason why corepresentations or comodules for Hopf algebras play an important role, is that in case of a Hopf algebra $H=\Oo(G)$ that arises as the coordinate algebra of an algebraic group $G$, the category of $H$-comodules coincides with the category of representations of this algebraic group $G$.
Henceforth the aim of the current paper is to develop the theory of partial corepresentations of Hopf algebras.

Partial actions and partial representations of groups first appeared in the context of operator algebras in the effort of characterizing $\mathbb{Z}$-graded $C^\ast$-algebras in terms of an appropriated crossed product \cite{ruy, ruy2}. A fully algebraic theory of partial group actions and partial group representations was established in \cite{dok}, leading to further developments, such as the theory of partial representations of groups as developed in  \cite{dok0} and \cite{dok1}. For a more detailed survey of the state of art of partial actions and partial representations of groups, see  \cite{D}.

In the realm of Hopf algebras, partial actions (or partial module algebras) and partial coactions (or partial comodule algebras) were introduced in \cite{caen06}. The primordial motivation was to describe the Galois theory for partial group actions \cite{DFP} in terms of Galois corings \cite{CdG}. 
The duality between partial actions and partial coactions of Hopf algebras was explored in \cite{BV}. There, the authors introduced also the notions of partial module coalgebras and partial comodule coalgebras, 
defined several constructions of partial smash products and partial cosmash coproducts, and explored their dualities. 
The notion of a partial representation of a Hopf algebra was introduced first in \cite{AB} and developed in \cite{ABV}. A partial representation of a Hopf algebra $H$ in an algebra $A$ is a linear map $\pi:H\to A$ preserving the unit but with a relaxed version of preservation of multiplication. When a $H$ allows a partial representation in the endomorphism ring of a $k$-module $M$, we say that $M$ is a partial $H$-module. In \cite{ABV}, it was shown that given a Hopf algebra $H$, one can define a universal algebra $H_{par}$ which factorizes every partial representation of $H$ on a given unital algebra $B$ by an algebra morphism between $H_{par}$ and $B$. Therefore, the category of partial $H$-modules turns out to be isomorphic to the category of $H_{par}$-modules. Moreover, this universal algebra $H_{par}$ is a 
a Hopf algebroid whose base algebra $A$ is a suitable subalgebra of $H_{par}$ upon which $H$ partially acts \cite{ABV}. As a consequence, the category of partial $H$-modules becomes closed monoidal and such that the forgetful functor to the category of $A$-bimodules is strict monoidal and preserves internal homs. Moreover, the algebra objects in the category of partial $H$-modules correspond to partial actions of $H$ according to the definition given in \cite{caen06}. 
For a more detailed presentation of recent developments in the theory of partial actions and partial representations of Hopf algebras see \cite{Bat}.

In this paper, after some preliminary observations concerning dual pairings, cofree coalgebras and partial representation in \seref{preliminaries}, we introduce in \seref{parcomod} the notion of a partial comodule over a Hopf algebra $H$. We study the relation with the alternative notion of partial comodule studied in \cite{HV}, and provide an example of a partial comodule over Sweedler's four-dimensional Hopf algebra which contains no finite dimensional partial subcomodules. As a consequence we find that partial comodules do not satisfy a version of the fundamental theorem of comodules, which is a first indication that the theory of partial comodules is more involved than the global case. We call the partial comodules that do satisfy the fundamental theorem (i.e.\ those partial comodules that are equal to the sum of their finite-dimensional subcomodules) {\em regular}. Natural examples of partial comodules arise from partial representations of algebraic groups.

In \seref{parcorep} we define the notion of a partial corepresentation from a coalgebra $C$ to a Hopf algebra $H$, as a relaxed version of coalgebra morphisms and we study the relation between partial comodules and partial corepresentations. 
More precisely, for any finitely generated and projective (or even any $R$-firmly projective) $k$-module $M$, we show that the partial comodule structures on $M$ are in correspondence with partial representations from the comatrix coalgebra (associated to $M$) to $H$. However, in contrast to the theory of partial representations and partial modules, not every partial comodule is induced in this way by a partial corepresentation. An explicit example of this phenomenon is obtained from the same example mentioned before of a non-regular partial comodule over the Sweedler Hopf algebra mentioned above. These subtleties show that the theory of partial comodules and corepresentations is not just a formal dualization of the theory developed in \cite{ABV}.
In the same section we provide examples of partial representations obtained from the partial cosmash coproduct and provide various duality theorems relating the partial modules and representations of a Hopf algebra $H$ with the partial corepresentations and comodules of its Sweedler dual.

In \seref{universal}, we construct a universal coalgebra $H^{par}$ which is a subcoalgebra of the cofree coalgebra $C(H)$, and which factorizes partial corepresentations of $H$ by coalgebra morphisms. The category of comodules over $H^{par}$ coincides exactly with the category of regular partial comodules over $H$.
We show that there is a coalgebra $C^{par} \subseteq H^{par}$ and a coalgebra epimorphism $\overline{E}:H^{par} \rightarrow C^{par}$ such that $C^{par}$ is a left partial $H$-comodule coalgebra and there is a coalgebra isomorphism between $H^{par}$ and the partial cosmash coproduct $\underline{C^{par}\cosmash H}$.

In the last section, we introduce the notion of Hopf-coalgebroid over a coalgebra $C$, building upon the notion of a bicoalgebroid, introduced by Brzezinski and Militaru in \cite{BM}, and dualizing the notion of a Hopf algebroid, in the sense of Gabriella B\"{o}hm \cite{Bohm}.
We show that
the universal coalgebra $H^{par}$ has a structure of a Hopf-coalgebroid over the coalgebra $C^{par}$. As a consequence, we find that the category of regular partial comodules over $H$ is a closed monoidal category, with a strict monoidal functor to the category of $C^{par}$-bicomodules. \\

{\bf Notations.} Throughout the paper $k$ denotes a fixed associative and commutative ring with unit. It will be mentioned explicitly when $k$ is supposed to be a field (which will be the overall assumption from \seref{universal} onward).
Unadorned tensor products are tensor products over $k$, and $\Hom(M,N)$ denotes the set of $k$-linear maps between the $k$-modules $M$ and $N$. The identity morphism on an object $X$ is denoted by $I_X$ or $I$.
If $M$ is a $k$-module ($k$-vector space), then we denote the dual $k$-module ($k$-vector space) as $M^*=\Hom(M,k)$.

\section{Algebraic preliminaries}\selabel{preliminaries}

In this section, we recall some known facts from literature that will be used frequently throughout the paper, along with some new observations.

\subsection{Dual pairings}

\begin{definition}
We call a {\em pairing} a triple $(V,\hat V,\bk{-,-})$, where $V$ and $\hat V$ are $k$-modules and $\bk{-,-}:V\ot \hat V\to k$ is a $k$-linear map. Consider the following conditions.
\begin{enumerate}
\item if $\langle v,\hat v\rangle =0$ for every $\hat v\in \hat V$, then $v=0$;
\item if $\langle v,\hat v\rangle =0$ for every $v\in V$, then $\hat v=0$.
\end{enumerate}
We call the triple $(V,\hat V,\bk{-,-})$ a  {\em left dual pairing} and the map $\bk{-,-}$ {\em left non-degenerate} if condition (1) above is satisfied; it is called a {\em  right dual pairing} if condition (2) is satisfied. 
If both conditions  (1) and (2) above are satisfied, we say that $(V,\hat V,\bk{-,-})$ is a
{\em dual pairing} and the map $\bk{-,-}$ is {\em non-degenerate}.

If $(V,\hat V,\bk{-,-})$ and $(W,\hat W,\bk{-,-})$ are (dual) pairings, then a couple of $k$-linear maps $(f:V\to W,g:\hat W\to \hat V)$ is called a {\em morphism of (dual) pairings} if for all $v\in V$ and $\hat w\in \hat W$, the following identity holds:
$$\bk{v,g(\hat w)}=\bk{f(v),\hat w}.$$

The category of pairings will be denoted by $\P_k$. The full subcategories of left, right and two-sided dual pairings are denoted respectively by $\LDP_k$, $\RDP_k$ and $\DP_k$.

We say a $k$-module $V$ is {\em dualizable} by $\hat V$ if there exists a dual pairing $(V,\hat V,\bk{-,-})$. A $k$-linear map $f:V\to W$ is called {\em dualizable} by $\hat V$ and $\hat W$ if and only if both $V$ and $W$ are dualizable by respectively $\hat V$ and $\hat W$, and there exists a linear map $g:\hat W\to \hat V$ such that $(f,g)$ is a morphism in $\DP_k$. In this situation, we call $\hat V$ a dual of $V$ and $g$ a dual (or {\em adjoint}) of $f$.
\end{definition}

An immediate example of a left dual pairing is given by $(V^*,V,\ev)$, where $V$ is any $k$-module and $\ev:V^*\ot V\to k$ is the evaluation map. If $V$ is moreover (locally) projective, then $(V^*,V,\ev)$ is even a two-sided dual pairing.
In other words, any (locally) projective $k$-module, and any map between such modules, is dualizable. In general, a $k$-module can have several duals, unless $V$ is finitely generated and projective, then $V^*$ is the unique dual of $V$. However, as the next proposition shows, the dual of a given morphism $f:V\to W$ is unique with respect to fixed duals $\hat V$ and $\hat W$ for $V$ and $W$. From the definitions it is clear that if $(V,\hat V,\bk{-,-})$ is a left dual pairing, then $(\hat V,V,\bk{-,-}\circ \tau)$, where $\tau:\hat V\ot V\to V\ot \hat V$ is the twist map, is a right dual pairing. In particular, if  $(V,\hat V,\bk{-,-})$ is a dual pairing, then so is $(\hat V,V,\bk{-,-}\circ\tau)$.

Given a dual pairing $(V,\hat V,\bk{-,-})$, and subspaces $W\subset V$, $Z\subset \hat V$, we define
$$W^\perp=\{\hat v\in \hat V~|~\bk{w,\hat v}=0, \forall w\in W\},$$
$$Z^\perp=\{v\in V~|~\bk{v,z}=0, \forall z\in Z\}.$$

\begin{lemma}\lelabel{jmaps}
Let $(V,\hat V,\bk{-,-})$ be a right pairing and consider the map
\[
 j_{\hat V} : \hat V  \rightarrow  V^*, \ \ \ \ \hat v  \mapsto   \langle - , \hat v \rangle . 
\]
Then $(V,\hat V,\bk{-,-})$ is a right dual pairing if and only if $j_{\hat V}$ is injective if and only if $V^\perp=\{0\}$. 

A symmetric result holds for left dual pairings.
\end{lemma}

When $k$ is a field and we endow a dual space $V^*$ with the finite topology, then it is known (see \cite[Corollary 1.2.9]{romenos}) that a subspace $W\subset V^*$ is dense in this finite topology if and only if $W^\perp=\{0\}$. Hence, the previous lemma implies that for a (two-sided) dual pairing $(V,\hat V,\bk{-,-})$, $\hat V$ can be identified with a dense subspace of $V^*$ (and symmetrically $V$ can be identified with a dense subspace of $\hat V^*$).
Slightly more generally, one has the following result which follows now as an easy consequence.

\begin{proposition} \label{pairingandmorphisms} %\cite{majid} 
Let $k$ be a field 
and $(\phi,\psi): (V,\hat V,\bk{-,-})\to (W,\hat W,(-,-))$ be a morphism in $\DP_k$.
\begin{enumerate}
\item $\phi$ is injective if and only if the image of  $\psi$ is dense in $\hat V$ (or equivalently, in $V^*$). 
\item $\psi$ is injective if and only if the image of  $\phi$ is dense in $W$ (or equivalently, in $\hat W^*$). 
\end{enumerate}
\end{proposition}

The following observations will be essential for this paper.

\begin{lemma}\lelabel{pairings}
\begin{enumerate}[(i)]
\item  The category $\P_k$ is a monoidal $k$-linear category and the functors
\[
\xymatrix{
& \P_k \ar[dl]_L \ar[dr]^R\\
\Mod_k && \Mod_k^{op}
}
\]
given by $L(V,\hat V,\bk{-,-})=V$, and $R(V,\hat V,\bk{-,-})= \hat V$ (and for a morphism $(f,g) :( V, \hat{V} , \bk{-,-}) \rightarrow ( W, \hat{W} , (-,-))$, $L(f,g)=f$ and $R(f,g)=g$) are both strict monoidal.
\item Let $(f,g):(V,\hat V,\bk{-,-})\to (W,\hat W,(-,-))$ be a morphism of pairings and suppose that $(V,\hat V,\bk{-,-})$ is a right dual pairing. Then $f=0$ implies $g=0$ (implies $(f,g)=0$).
\end{enumerate}
\end{lemma}

\begin{proof}
\ul{(i)}. The $k$-linearity is clear. Given two objects $(V,\hat V,\bk{-,-})$ and $(W,\hat W,\bk{-,-})$ in $\P_k$, we can define the tensor product as $(V\ot W,\hat V\ot \hat W,\bkk{-,-})$, where
$$\bkk{v\ot w,\hat v\ot \hat w}=\bk{v,\hat v}\bk{w,\hat w}.$$ 
The monoidal unit is given by $\ul k=(k,k,I_k)$. \\
\ul{(ii)}. We find
\begin{eqnarray*}
f=0 &{\Rightarrow} & f(v)=0, \forall v\in V\\
&{\Rightarrow} & \bk{f(v),\hat w}=0, \forall v\in V, \forall \hat w\in \hat W\\
&{\Rightarrow} & \bk{v,g(\hat w)}=0, \forall v\in V, \forall \hat w\in \hat W\\
&{\Rightarrow} & g(\hat w)=0,  \forall \hat w\in \hat W\\
&{\Rightarrow} & g=0. 
\end{eqnarray*}
\end{proof}

\begin{proposition}\label{dualpairing}
\begin{enumerate}[(i)]
\item The restriction of the functor $R$ to the full subcategory $\LDP_k$ creates kernels, is faithful and has a fully faithful right adjoint $G$. 
\item The restriction of the functor $L$ to the full subcategory $\RDP_k$ creates cokernels, is faithful and has a fully faithful left adjoint $F$.
\item The restrictions of $L$ and $R$ to the full subcategory for $\DP_k$ are both faithful. If $k$ is moreover a field, then $L$ has a right adjoint $F$ and $R$ has a left adjoint $G$.
\item If $k$ is a field then any of the categories $\LDP_k$, $\RDP_k$ and $\DP_k$  are monoidal subcategories of $\P_k$ (and the restrictions of the functors $R$ and $L$ remain strictly monoidal).
\end{enumerate}
\end{proposition}

\begin{proof}
\ul{(i)}. 
The faithfulness of the functor $R$ follows from (the symmetric version of) the preceding lemma. 
The left adjoint $G$ of $R$ is given by $G(V)=(V^*,V,\ev)$.

Let now $(f,\hat f): (V,\hat V)\to (W,\hat W)$ be a morphism in $\LDP_k$. Consider the kernel $g:Z\to V$ of $f$ in $\Mod_k$ and the cokernel $\hat g:\hat V\to \hat Z$ of $\hat f$ in $\Mod_k$ (which is the kernel of $\hat f$ in $\Mod_k^{op}$). We will show that $(Z,\hat Z)$ is a left dual pairing and $(g,\hat g)$ is a morphism of dual pairings.

Since $\hat g$ is a cokernel in $\Mod$, it is surjective. Hence for any $\hat z\in \hat Z$, we can write $\hat z=\hat g(\hat v)$ for a certain $\hat v\in\hat V$. With this notation we define for any $z\in Z$:
$$\bk{z,\hat z}=\bk{g(z),\hat v}.$$
Let us verify that this is independent of the choice of $\hat v$. Consider a second $\hat v'\in \hat V$ such that $\hat g(\hat v')=\hat z$. Then this means that $\hat v-\hat v'\in\Ker \hat g=\im \hat f$. So we find $\hat w\in W$ such that 
$\hat v-\hat v'=\hat f(\hat w)$. Then we have 
$$\bk{g(z),\hat v-\hat v'}=\bk{g(z),\hat f(\hat w)}=\bk{fg(z),\hat w}=0,$$
and therefore $\bk{g(z),\hat v}=\bk{g(z),\hat v'}$, hence $\bk{z,\hat z}$ is well-defined and $(g,\hat g)$ is a morphism of pairings by construction.

Let us now verify that the pairing $(Z,\hat Z)$ is left non-degenerate. Suppose that $z\in Z$ is such that 
$\bk{z,\hat z}=0$ for all $\hat z\in \hat Z$. Since $\hat g$ is surjective, this means that 
$\bk{g(z),\hat v}=0$ for all $\hat v\in\hat V$. By the left non-degeneracy of $(V,\hat V)$ we find that $g(z)=0$ and, since $g$ is injective, we conclude that $z=0$ as needed.\\
\ul{(ii)}.
This follows by symmetry from the previous part. In particular, the right adjoint $F$ of  $L$ is given by $F(V)=(V,V^*,\ev)$.\\
\ul{(iii)}. If $k$ is a field, then the pairings $F(V)$ and $G(V)$ defined above are non-degenerate on both sides.\\
\ul{(iv)}. Using the fact that $k$ is a field, the non-degeneracy of the pairing $(V\ot W,\hat W\ot \hat V,\bkk{-,-})$ defined in the previous lemma is easily checked. 
\end{proof}

\begin{remarks}\label{re:dualpairing}
The previous result will be very useful in this paper. It means that a diagram of dualizable morphisms in $\Mod_k$ commutes if and only if the diagram of their duals commutes in $\Mod_k^{op}$. Let us illustrate this by some examples.

If we have morphisms $(m,\Delta):(A,C,\bk{-,-})\ot (A,C,\bk{-,-})\to (A,C,\bk{-,-})$ and $(\eta,\epsilon):(A,C,\bk{-,-})\to \ul k$ in $\P_k$, and $(A,C,\bk{-,-})$ is left non-degenerate, then the observations from \leref{pairings} imply the classical result that $(A,m,\eta)$ is a $k$-algebra if $(C,\Delta,\epsilon)$ is a $k$-coalgebra. Conversely, if $k$ is a field and $(A,C,\bk{-,-})$ is right non-degenerate, then also $(A,C,\bk{-,-})\ot (A,C,\bk{-,-})$ is right non-degenerate by Proposition \ref{dualpairing}(iv), and we obtain that $(C,\Delta,\epsilon)$ is a $k$-coalgebra if  $(A,m,\eta)$ is a $k$-algebra. 

If all pairings are dual pairings, then we
find that $(A,m,\eta)$ is a $k$-algebra if and only if $(C,\Delta,\epsilon)$ is a $k$-coalgebra, if and only if $(A,C,\bk{-,-})$ is an algebra in $\DP_k$. We 
refer to this situation as {\em $\bk{-,-}$ is a dual pairing between the coalgebra $C$ and the algebra $A$}. 
Explicitly, this means that the following compatibility conditions between pairing and structure maps hold:
\begin{eqnarray*}
\langle c, 1_A \rangle &=& \epsilon_C (c), \\
\langle c,ab \rangle &=&\langle c_{(1)} ,a\rangle \langle c_{(2)} , b \rangle , 
\end{eqnarray*}
for all $c\in C$ and $a,b \in A$, in which $\Delta_C (c)=c_{(1)}\otimes c_{(2)}$.

In the same way, it follows that for a given morphism $(\phi,\psi): (A,C,\bk{-,-})\to (B,D,(-,-))$ in $\DP_k$, where $(A,C,\bk{-,-})$ and $(B,D,(-,-))$ are algebras in $\DP_k$, $\phi$ is an algebra morphism if and only if $\psi$ is a coalgebra morphism. Again, it is not necessary that all pairings are two-sided non-degenerate for such a result to be true.
For example, if $(\phi,\psi): (A,C,\bk{-,-})\to (B,D,(-,-))$ is a morphism in $\P_k$ and $(B,D,(-,-)$ is in $\LDP_k$, then  $\psi$ being a coalgebra morphism implies that $\phi$ is an algebra morphism. Conversely, if $k$ is a field and $(A,C,\bk{-,-})$ is in $\RDP_k$, then $\phi$ being an algebra morphism implies that $\psi$ is a coalgebra morphism. 

In case $k$ is a field,  we find that 
if $C$ is a $k$-coalgebra then $(C^*,C,\ev)$ is an algebra in $\DP_k$. On the other hand, if $A$ is a $k$-algebra, then 
$(A^\circ,A,\ev)$ is a coalgebra in $\LDP_k$, where $A^\circ$ denotes the finite dual of $A$. If moreover $A^\circ$ is dense in $A^*$, then $(A^\circ,A,\ev)$ is a coalgebra in $\DP_k$. 
Similarly, if $H$ is a Hopf algebra then $(H^\circ,H,\ev)$ is a Hopf algebra in $\LDP_k$. If moreover $H^\circ$ is dense in $H^*$ (this condition is known as ``$H$ is residually finite dimensional'' or ``$H^\circ$ separates points'') then $(H,H^\circ,\ev)$ and $(H^\circ,H,\ev)$ are Hopf algebras in $\DP_k$. We also say that in this case the evaluation map is a (non-degenerate) pairing of Hopf algebras between $H$ and $H^\circ$.
\end{remarks}

The following results will be useful in this work, a proof can be found for example in \cite{majid}.

\begin{thm} \label{reductionofpairing} Suppose that $k$ is a field and $(C,A,\bk{-,-})$ is a dual pairing between the algebra $A$ and the coalgebra $C$.
\begin{enumerate}
\item Let $J$ be an ideal of $A$. Then $J^\perp=\left\{ c \in C ~|~ \bk{c,a}=0, \forall a \in J\right\}$ is a subcoalgebra of $C$.  Moreover, if the pairing is non-degenerate, there is a well-defined non-degenerate pairing $(J^\perp, A/J,\bkk{-,-})$, such that $(\iota,\pi)$ is a morphism of pairings, where $\iota:J^\perp\to C$ is the canonical inclusion and $\pi:A\to A/J$ is the canonical surjection. In particular, 
$$\langle \!\langle c, a+J\rangle \! \rangle =\langle c,a\rangle,$$
for all $c\in C$ and $a\in A$.
\item Let $D$ be a subcoalgebra of $C$. Then $D^\perp=\left\{ a\in A ~|~\bk{d,a}=0, \forall d\in D\right\}$ is an ideal of $A$. Moreover, if the pairing is non-degenerate, there is a well-defined non-degenerate pairing $(D,A/D^\perp,\bkkk{-,-})$, such that $(\iota,\pi)$ is a morphism of pairings, where $\iota:D\to C$ is the cononical injection and $\pi:A\to A/D^\perp$ is the canonical surjection. In particular,
$$\langle \!\langle \! \langle d , a+D^\perp \rangle \! \rangle \! \rangle = \langle d,a\rangle,$$
for all $d\in D$ and $a\in A$.
\end{enumerate}
\end{thm}

\subsection{Cofree coalgebras}\selabel{cofree}

\begin{defi} Let $\mathbb{V}$ be a module over a commutative ring $k$. The cofree coalgebra over $\mathbb{V}$ is a pair $(C(\mathbb{V}), p)$, in which $C(\mathbb{V})$ is a $k$-coalgebra and $p:C(\mathbb{V})\rightarrow \mathbb{V}$ is a linear map such that, for any other coalgebra $C$ equipped with a linear map $f:C\rightarrow \mathbb{V}$, there is a unique coalgebra morphism $\overline{f}:C\rightarrow C(\mathbb{V})$ such that the following diagram commutes
\[
\xymatrix{C \ar[r]^-{f} \ar[rd]_-{\overline{f}} & \mathbb{V} \\ & C(\mathbb{V}) \ar[u]_-{p} }
\]
\end{defi}
Because of its definition as a universal object, the cofree coalgebra over a module is unique (up to isomorphism) whenever it exists.  Note that the cofree coalgebra provides a right adjoint for the forgetful functor from $k$-coalgebras to $k$-modules. The map $p:C(\VV)\to \VV$ is exactly the counit of the adjunction evaluated in $\VV$. 

It is well-known that that the cofree coalgebra over any module over a commutative ring indeed exists (an argument based on the special adjoint functor theorem is given in \cite[Theorem 4.1]{Barr}) and many explicit constructions have been obtained in the literature. Under the assumption that $k$ is a field, in \cite{BL} it is shown that the cofree coalgebra of a vector space $\VV$ can be viewed given as the space of $k$-linear functions $f:k[t]\to T(\VV)$ (where $k[t]$ denotes the polynomial algebra and $T\mathbb{V}=\bigoplus_{n=0}^\infty \mathbb{V}^{\ot n}$ the tensor algebra)  that are of degree zero (i.e. $f(t^n)\in \mathbb{V}^{\ot n}$) and representative (i.e. for which there is a finite family $g_i , h_i: k[t]\to T(\VV)$ such that, for every $n,m\in\NN$ we have 
$f(t^{nm})=\sum_{i=1}^n g_i (t^n) h_i (t^m)$). In this case the map $p$ is given by evaluation at $t$.

From the above construction it easily follows that if $k$ is a field, the cofree coalgebra $C(\VV)$ is {\em cogenerated as a coalgebra} by the vector space $\VV$ in the following sense. The family of morphisms 
$$p_n: \xymatrix{C(\VV) \ar[rr]^-{\Delta^{n}} && C(\VV)^{\ot n} \ar[rr]^-{p^{\ot n}} && \VV^{\ot n}}$$
for all $n\in \NN$ is jointly monic in $\Vect$.  (Here we denoted by $\Delta^n$ the $n$-fold comultiplication map, i.e.\ $\Delta^2=\Delta$,  $\Delta^n=(id\ot \Delta^{n-1})\circ \Delta$). The last property means that two linear maps $g,h:X\to C(\VV)$ are identical if and only if their composition with all $p_n$ is so. Indeed, from the description of the cofree coalgebra as above, we see that $p_n(f)=f(t^n)$ for any $f\in C(\VV)$. Hence $f=0$ if and only if $p_n(f)=0$ for all $n$.

The following result will be used frequently in this paper.
\begin{proposition} (see e.g.\ \cite{hazel})\label{pairingcofree}
If $k$ is a field, the linear map 
\[
\begin{array}{rccl} \langle \; , \; \rangle :& C(\mathbb{V}) \otimes T\mathbb{V}^* & \rightarrow  & k, \\ & x\otimes v^*_{i_1} \otimes \cdots \otimes v^*_{i_n} & \mapsto & 
v_{i_1}^*(p(x_{(1)}))\cdots v_{i_n}^*(p(x_{(n)}))
\end{array}
\]
defines a nondegenerate pairing between $C(\mathbb{V})$ and the tensor algebra $T\mathbb{V}^*$.
\end{proposition}

In fact, the previous result can obviously be improved in the following way: any non-degenerate dual pairing  $(V,\hat V,\bk{-,-})$ induces a non-degenerate dual pairing between the cofree coalgebra $C(V)$ and the free algebra $T(\hat V)$.

\subsection{Partial representations and partial (co)actions of Hopf algebras}\label{partrep}\selabel{partrep}

The notion of a partial representation of a Hopf algebra was introduced in \cite{ABV}. The definition was motivated by the corresponding notion for groups \cite{dok0}, which can be recovered from the definition below by taking the group Hopf algebra $kG$. In fact, the original definition in \cite{ABV} stated 5 axioms. As it was observed by Paolo Saracco \cite{Paolo}, the other axioms can be obtained from the current ones (see Lemma~\ref{axiomsPRep}).

\begin{defi} 
Let $H$ be a Hopf $k$-algebra, and let $B$ be a unital $k$-algebra. A \emph{partial representation} of $H$ in $B$ is a linear map $\pi: H \rightarrow B$ such that 
\begin{enumerate}[label=(PR\arabic*),ref=PR\arabic*]%0k
\item\label{def:PR1} $\pi (1_H)  =  1_B$; 
\item\label{def:PR2} $\pi (h) \pi (k_{(1)}) \pi (S(k_{(2)}))  =   \pi (hk_{(1)}) \pi (S(k_{(2)})) $, for every $h,k\in H$;
\item\label{def:PR3} $\pi (h_{(1)}) \pi (S(h_{(2)})) \pi (k)  =   \pi (h_{(1)}) \pi (S(h_{(2)})k)$, for every $h,k\in H$.
\end{enumerate}
If $(B,\pi)$ and $(B',\pi')$ are two partial representations of $H$, then we say that an algebra morphism $f:B\rightarrow B'$ is a morphism of partial representations if $\pi ' =f\circ \pi$. 

The category whose objects are partial representations of $H$ and whose morphisms are morphisms of partial representations is denoted by $\mbox{ParRep}_H$.
\end{defi}

\begin{lemma}\cite{Paolo}\label{axiomsPRep}
Let $\pi: H \rightarrow B$ be a partial representation. Then the following axioms are satisfied as well
\begin{enumerate}[label=(PR\arabic*),ref=PR\arabic*,start=4]
	\item\label{def:PR4} $\pi (h) \pi (S(k_{(1)})) \pi (k_{(2)}) = \pi (hS(k_{(1)})) \pi (k_{(2)})$, for every $h,k\in H$;
	\item\label{def:PR5} $\pi (S(h_{(1)}))\pi (h_{(2)}) \pi (k) = \pi (S(h_{(1)}))\pi (h_{(2)} k)$, for every $h,k\in H$.
\end{enumerate}
Conversely, if a linear map $\pi: H \rightarrow B$ satisfies (\ref{def:PR1}), (\ref{def:PR4}) and (\ref{def:PR5}), then it is a partial representation.
\end{lemma}

\begin{proof}
Suppose that (\ref{def:PR1}), (\ref{def:PR2}) and (\ref{def:PR3}) are satisfied. Then for any $h,k\in H$ we find that
\begin{eqnarray*}
\pi (h) \pi (S(k_{(1)})) \pi (k_{(2)})&=&\pi (h\epsilon(k_{(1)})) \pi (S(k_{(2)})) \pi (k_{(3)})\\
&=& \pi (h S(k_{(1)}) k_{(2)}) \pi (S(k_{(3)})) \pi (k_{(4)})\\
&\stackrel{(\ref{def:PR2})}{=}& \pi (h S(k_{(1)}))\pi(k_{(2)}) \pi (S(k_{(3)})) \pi (k_{(4)})\\
&\stackrel{(\ref{def:PR3})}{=}& \pi (h S(k_{(1)}))\pi(k_{(2)}) \pi (S(k_{(3)}) k_{(4)})\\
&\stackrel{(\ref{def:PR1})}{=}& \pi (h S(k_{(1)}))\pi(k_{(2)}).
\end{eqnarray*}
In a similar way, 
\begin{eqnarray*}
\pi (S(h_{(1)}))\pi (h_{(2)}) \pi (k) &=& \pi (S(h_{(1)}))\pi (h_{(2)}) \pi (S(h_{(3)})h_{(4)}k)\\
&\stackrel{(\ref{def:PR3})}{=}& \pi (S(h_{(1)}))\pi (h_{(2)}) \pi (S(h_{(3)}))\pi(h_{(4)}k)\\
&\stackrel{(\ref{def:PR2})}{=}& \pi (S(h_{(1)})h_{(2)}) \pi (S(h_{(3)}))\pi(h_{(4)}k)\\
&\stackrel{(\ref{def:PR1})}{=}& \pi (S(h_{(1)}))\pi (h_{(2)} k).
\end{eqnarray*}
The converse statement is proven in the same way.
\end{proof}

With this definition, one can define what is a partial $H$-module.

\begin{defi} Let $H$ be a Hopf $k$-algebra. A partial $H$-module is a pair $(M, \pi)$ in which $M$ is a $k$-module and $\pi :H \rightarrow \mbox{End}_k (M)$ is a partial representation of $H$. A morphism between two partial $H$-modules $(M, \pi)$ and $(N, \pi')$ is a $k$-linear map $f:M\rightarrow N$ such that, for each $h\in H$, we have $f\circ \pi (h)=\pi' (h) \circ f$. The category of partial $H$-modules will be denoted by ${}_H \Mod^{par}$.
\end{defi}

Examples of partial representations and partial modules, arise from partial representations of groups, from partial actions of groups and Hopf algebras (on algebras) and from global modules by means of a suitable projection (see \cite{ABV2}).

\begin{defi} \label{hpar}  Let $H$ be a Hopf algebra 
and let $T(H)$ be the tensor algebra of the $k$-module $H$. The {\em partial ``Hopf'' algebra} $H_{par}$ is the quotient
of $T(H)$ by the ideal $I$ generated by elements of the form 
\begin{enumerate}
\item $1_H - 1_{T(H)}$; 
\item $h \otimes k_{(1)} \otimes S(k_{(2)}) - hk_{(1)} \otimes S(k_{(2)})$, for all $h,k \in H$;
\item $h_{(1)} \otimes S(h_{(2)}) \otimes k - h_{(1)} \otimes S(h_{(2)})k$, for all $h,k \in H$.
\end{enumerate}
\end{defi}

\begin{remark}
In \cite{ABV}, the algebra $H_{par}$ was defined by means of 5 families of relations, rather than the 3 above. In view of Lemma~\ref{axiomsPRep}, the definition presented here is equivalent to the one of \cite{ABV}.
\end{remark}

Denoting the class of $h\in H$ in $H_{par}$ by $[h]$, it is easy to see that the map 
\[
 [\underline{\; }]:H  \rightarrow  H_{par},   \ \ \ \   h  \mapsto  [h]  
\]
is a partial representation of the Hopf algebra $H$ on $H_{par}$. Moreover, for every partial representation $\pi :H\rightarrow B$ there exists a unique morphism of algebras $\hat{\pi}:H_{par} \rightarrow B$ such that $\pi =\hat{\pi} \circ [\underline{\;} ]$. As a consequence, the category ${}_H \Mod^{par}$, of partial $H$-modules is isomorphic to the category ${}_{H_{par}}\Mod$ of $H_{par}$-modules. It should be remarked that despite its name, the partial ``Hopf'' algebra $H_{par}$ is not a Hopf algebra, but rather a Hopf algebroid (in the sense of B\"ohm) over the base algebra $A_{par}$ which is the subalgebra of $H_{par}$ generated by elements of the form $\varepsilon_h =[h_{(1)}][S(h_{(2)})]$. Consequently, the category of partial representations is a closed monoidal category with a strict monoidal forgetful functor to the category of $A_{par}$-bimodules. Finally, there is a partial action of the Hopf algebra $H$ on $A_{par}$ (i.e.\ $A_{par}$ is a partial module algebra over $H$) such that the partial Hopf algebra $H_{par}$ is isomorphic to the partial smash product $\underline{A_{par} \# H}$. In \cite{ABV}, it was also shown that the algebra $A_{par}$ is the quotient of the free algebra $T(H)$ by the ideal generated by elements of the form
\begin{enumerate}[label=(J\arabic*),ref=J\arabic*]%0k
\item $1_H - 1_{T(H)}$;
\item $h-h_{(1)}\ot h_{(2)}$;
\item $h_{(1)}\ot h_{(2)}k-h_{(1)}k\ot h_{(2)}$;
\end{enumerate}
where $h,k\in H$.

Finally, let us also recall the notion of a symmetric right partial coaction (or partial comodule algebra) of a Hopf algebra $H$ on a unital algebra $A$, and of a symmetric right partial coaction (or partial comodule coalgebra) of a Hopf algebra $H$ on a  coalgebra $C$. 

\begin{defi} \cite{BV}, \cite{caen06}.  A symmetric right partial coaction of a bialgebra $H$ on a unital algebra $A$ is a linear map $\rho :A\rightarrow A\otimes H$, denoted as $\rho (a) =a^{[0]}\otimes a^{[1]}$, for every $a\in A$ such that the following identities are satisfied:
\begin{enumerate}[label=(PCA\arabic*),ref=PCA\arabic*]
\item\label{def:PCA1} $\rho (ab)=\rho (a) \rho (b)$, that is $(ab)^{[0]}\otimes (ab)^{[1]} =a^{[0]}b^{[0]} \otimes a^{[1]}b^{[1]}$ for every $a, b \in A$;
\item\label{def:PCA2} $(I\otimes \epsilon)\rho (a) =a$, that is $a^{[0]}\epsilon(a^{[1]})=a$ for every $a\in A$;
\item\label{def:PCA3} $(\rho \otimes I)\rho (a)=(\rho (\um )\otimes 1_H)[(I\otimes \Delta)\rho (a)]=[(I\otimes \Delta)\rho (a)](\rho (\um )\otimes 1_H)$, that is, \\ $ a^{[0][0]}\otimes a^{[0][1]}\otimes a^{[1]} =\um^{[0]}a^{[0]}\otimes \um^{[1]}{a^{[1]}}_{(1)}\otimes {a^{[1]}}_{(2)} =a^{[0]}\um^{[0]}\otimes {a^{[1]}}_{(1)}\um^{[1]}\otimes {a^{[1]}}_{(2)}$ .
\end{enumerate}
The second equality in axiom (\ref{def:PCA3}) means that the partial coaction is symmetric. The algebra $A$ is called a partial right $H$-comodule algebra.
\end{defi}

\begin{defi} \cite{BV} \label{comodule-coalgebra}
A symmetric left partial coaction of a bialgebra $H$ on a coalgebra $C$ is a linear map 
\[
 \lambda : C  \rightarrow  H\otimes C, \ \ \ \   c  \mapsto  c^{[-1]}\otimes c^{[0]}, 
\]
such that
\begin{enumerate}[label=(LPCC\arabic*),ref=LPCC\arabic*]
\item\label{def:LPCC1} $c^{[-1]}\ot {c^{[0]}}_{(1)}\ot {c^{[0]}}_{(2)} = {c_{(1)}}^{[-1]} {c_{(2)}}^{[-1]} \otimes {c_{(1)}}^{[0]} \otimes {c_{(2)}}^{[0]}$,
\item\label{def:LPCC2} $\epsilon_H(^{[-1]})c^{[0]} =c$,
\item\label{def:LPCC3}
$\begin{array}[t]{rcl}
	c^{[-1]}\ot c^{[0][-1]}\ot c^{[0][0]}
	&=&  {c_{(1)}}^{[-1]}
	\epsilon_C ({c_{(1)}}^{[0]}) {{c_{(2)}}^{[-1]}}_{(1)} \otimes {{c_{(2)}}^{[-1]}}_{(2)} \otimes {c_{(2)}}^{[0]}\\
	&=&  {{c_{(1)}}^{[-1]}}_{(1)} {c_{(2)}}^{[-1]} 
	\epsilon_C ({c_{(2)}}^{[0]})\otimes {{c_{(1)}}^{[-1]}}_{(2)} \otimes {c_{(1)}}^{[0]}
\end{array}$
\end{enumerate}
for all $c\in C$. We call $C$ a partial left $H$-comodule coalgebra.
\end{defi}

In \cite{BV}, the partial cosmash coproduct of a Hopf algebra $H$ with a left partial $H$-comodule coalgebra $C$, was defined as the subspace  $\underline{C\cosmash H}\subseteq C\otimes H$, spanned by elements of the form
\[
x\cosmash h =x_{(1)}\otimes {x_{(2)}}^{[-1]} \epsilon( {x_{(2)}}^{[0]}) h.
\]
It then follows that
\[
x\cosmash h  = x_{(1)}\cosmash {x_{(2)}}^{[-1]} \epsilon( {x_{(2)}}^{[0]}) h.
\]
The partial cosmash coproduct $\underline{C\cosmash H}$ has a structure of a coalgebra with the comultiplication given by 
\[
\underline{\Delta} (x\cosmash h) =x_{(1)}\cosmash {x_{(2)}}^{[-1]} h_{(1)}\otimes {x_{(2)}}^{[0]} \cosmash h_{(2)} ,
\]
and counit given by
\[
\underline{\epsilon} (x\cosmash h) =\epsilon_C (x) \epsilon_H (h) .
\]

Right partial coactions on a coalgebra are defined symmetrically. Since we will use both versions of this notion, we state the right-hand version explicitly for the reader's convenience.

\begin{defi}\label{rcomodule-coalgebra}
A symmetric right partial coaction of a bialgebra $H$ on a coalgebra $C$ is a linear map 
\[
 \rho :  C  \rightarrow  C\otimes H, \ \ \ \ 
 c  \mapsto  c^{[0]}\otimes c^{[1]}, 
\] 
such that for all $c\in C$:
\begin{enumerate}[label=(RPCC\arabic*),ref=RPCC\arabic*]
\item\label{def:RPCC1} ${c^{[0]}}_{(1)} \otimes {c^{[0]}}_{(2)} \otimes c^{[1]} ={c_{(1)}}^{[0]} \otimes {c_{(2)}}^{[0]} \otimes {c_{(1)}}^{[1]} {c_{(2)}}^{[1]}$;
\item\label{def:RPCC2} $c^{[0]}\epsilon_H (c^{[1]}) =c$;
\item\label{def:RPCC3}
$\begin{array}[t]{rcl}
	c^{[0][0]}\otimes c^{[0][1]} \otimes c^{[1]}
	& = & \epsilon_C ({c_{(1)}}^{[0]}){c_{(2)}}^{[0]} \otimes {{c_{(2)}}^{[1]}}_{(1)} \otimes {c_{(1)}}^{[1]}{{c_{(2)}}^{[1]}}_{(2)}\\
	& = & \epsilon_C ({c_{(2)}}^{[0]}){c_{(1)}}^{[0]} \otimes {{c_{(1)}}^{[1]}}_{(1)} \otimes {{c_{(1)}}^{[1]}}_{(2)}{c_{(2)}}^{[1]} \ .
\end{array}$
\end{enumerate}
We call $C$ a partial right $H$-comodule coalgebra.
\end{defi}

\section{Partial comodules}\selabel{parcomod}

The definition of a partial module over a Hopf algebra makes sense for Hopf algebras in an arbitrary braided monoidal category. Since furthermore the notion of a Hopf $k$-algebra is self-dual, one can consider a partial module over a Hopf algebra in the dual of the category of $k$-modules. Explicitly, this leads us to the following definition.

\begin{defi} Let $H$ be a Hopf $k$-algebra. An {\em (algebraic) partial right $H$-comodule} is a $k$-module $M$ endowed with a $k$-linear map $\rho:M\to M\ot H$, satisfying
\begin{enumerate}[label=(PCM\arabic*),ref=PCM\arabic*]
	\item\label{def:PCM1} $(I\otimes \epsilon )\rho = I$;
	\item\label{def:PCM2} $(I\otimes I\otimes \mu )(I\otimes I \otimes I \otimes S) (I\otimes \Delta \otimes I)(\rho \otimes I)\rho = (I\otimes I\otimes \mu )(I\otimes I \otimes I \otimes S)(\rho \otimes I \otimes I)(\rho \otimes I)\rho$;
	\item\label{def:PCM3} $(I\otimes \mu \otimes I )(I\otimes I \otimes S \otimes I) (I\otimes I \otimes \Delta)(\rho \otimes I)\rho = (I\otimes \mu \otimes I )(I\otimes I \otimes S \otimes I)(\rho \otimes I \otimes I)(\rho \otimes I)\rho$.
\end{enumerate}
A morphism between two partial $H$-comodules $(M, \rho)$ and $(N, \sigma)$ is a linear map $f:M\rightarrow N$ satisfying $\sigma \circ f=(f\otimes I)\circ \rho$. The category of partial right $H$-comodules is denoted by $\Mod^H_{par}$.
\end{defi}

\begin{remarks}
\begin{enumerate}[(1)]
\item Using Sweedler notation, 
\[
 \rho : M  \rightarrow  M\otimes H, \ \ \ \   m  \mapsto  m^{[0]} \otimes m^{[1]}, 
\]
the axioms of a partial $H$-comodule can be rephrased as 
\begin{enumerate}
	\item[(\ref*{def:PCM1})] $(I\otimes \epsilon )\rho (m)=m$, 
	\item[(\ref*{def:PCM2})] $m^{[0][0]}\ot {m^{[0][1]}}_{(1)} \ot {m^{[0][1]}}_{(2)} S(m^{[1]}) = m^{[0][0][0]}\ot m^{[0][0][1]}\ot m^{[0][1]}S(m^{[1]})$,
	\item[(\ref*{def:PCM3})] $m^{[0][0]}\ot {m^{[0][1]}} S({m^{[1]}}_{(1)})\ot {m^{[1]}}_{(2)} = m^{[0][0][0]}\ot m^{[0][0][1]} S(m^{[0][1]})\ot m^{[1]}$,
\end{enumerate}
which have to be satisfied for all $m\in M$.

\item We used the prefix ``algebraic'' to distinguish the introduced notion from the ``geometric'' and ``quasi'' partial comodules from \cite{HV}. We refer to Proposition \ref{algebraicisquasi} for the relation between these notions.
\end{enumerate}
\end{remarks}

\begin{lemma}\label{axiomsPCoRep}
Let $H$ be a Hopf $k$-algebra, $M$ a $k$-module and $\rho:M\to M\otimes H$ a $k$-linear map. Then $(M,\rho)$ is a partial right $H$-comodule if and only if \eqref{def:PCM1} and the following two axioms hold
\begin{enumerate}[label=(PCM\arabic*),ref=PCM\arabic*,start=4]
	\item\label{le:PCM4} $(I\otimes I\otimes \mu )(I\otimes I \otimes S \otimes I) (I\otimes \Delta \otimes I)(\rho \otimes I)\rho=
	(I\otimes I\otimes \mu )(I\otimes I \otimes S \otimes I)(\rho \otimes I \otimes I)(\rho \otimes I)\rho$;
	\item\label{le:PCM5} $(I\otimes \mu \otimes I )(I\otimes S \otimes I \otimes I) (I\otimes I \otimes \Delta)(\rho \otimes I)\rho=
	(I\otimes \mu \otimes I )(I\otimes S \otimes I \otimes I)(\rho \otimes I \otimes I)(\rho \otimes I)\rho$.
\end{enumerate}
Or in Sweedler notation
\begin{enumerate}
	\item[(\ref*{le:PCM4})] $m^{[0][0]}\ot {m^{[0][1]}}_{(1)} \ot S({m^{[0][1]}}_{(2)}) m^{[1]} = m^{[0][0][0]}\ot m^{[0][0][1]}\ot S(m^{[0][1]})m^{[1]}$,
	\item[(\ref*{le:PCM5})] $m^{[0][0]}\ot S({m^{[0][1]}}) {m^{[1]}}_{(1)}\ot {m^{[1]}}_{(2)} = m^{[0][0][0]}\ot S(m^{[0][0][1]}) m^{[0][1]}\ot m^{[1]}$,
\end{enumerate}
for all $m\in M$.
\end{lemma}

\begin{proof}Suppose that $(M, \rho)$ is a partial right $H$-comodule and consider an element $m \in M$, then we find
	\begin{eqnarray*}
		&									& \hspace*{-2cm}
											  m^{[0][0][0]} \otimes m^{[0][0][1]} \otimes S(m^{[0][1]})m^{[1]} =\\
		&								=	& m^{[0][0][0]} \otimes {m^{[0][0][1]}}_{(1)} \otimes S({m^{[0][0][1]}}_{(2)}){m^{[0][0][1]}}_{(3)}S(m^{[0][1]})m^{[1]}\\
		&\stackrel{(\ref{def:PCM2})}{	=}	& m^{[0][0][0][0]} \otimes {m^{[0][0][0][1]}}_{(1)} \otimes S({m^{[0][0][0][1]}}_{(2)}){m^{[0][0][1]}}S(m^{[0][1]})m^{[1]}\\
		&\stackrel{(\ref{def:PCM3})}{	=}	& m^{[0][0][0]} \otimes {m^{[0][0][1]}}_{(1)} \otimes S({m^{[0][1]}}_{(2)}){m^{[0][1]}}S({m^{[1]}}_{(1)}){m^{[1]}}_{(2)}\\
		&\stackrel{(\ref{def:PCM1})}{	=}	& m^{[0][0]} \otimes {m^{[0][1]}}_{(1)} \otimes S({m^{[0][1]}}_{(2)}){m^{[1]}}.
	\end{eqnarray*}
So (\ref{le:PCM4}) is satisfied. Similarly, let us check (\ref{le:PCM5}):
	\begin{eqnarray*}
		&									& \hspace*{-2cm}
											  m^{[0][0][0]} \otimes S(m^{[0][0][1]}) m^{[0][1]} \otimes m^{[1]} =\\
		&								=	& m^{[0][0][0]} \otimes S(m^{[0][0][1]}) m^{[0][1]} S({m^{[1]}}_{(1)}) {m^{[1]}}_{(2)} \otimes {m^{[1]}}_{(3)}\\
		&\stackrel{(\ref{def:PCM3})}{	=}	&m^{[0][0][0][0]} \otimes S(m^{[0][0][0][1]}) m^{[0][0][1]} S({m^{[0][1]}}) {m^{[1]}}_{(1)} \otimes {m^{[1]}}_{(2)}\\
		&\stackrel{(\ref{def:PCM2})}{	=}	&m^{[0][0][0]} \otimes S({m^{[0][0][1]}}_{(1)}) {m^{[0][0][1]}}_{(2)} S({m^{[0][1]}}) {m^{[1]}}_{(1)} \otimes {m^{[1]}}_{(2)}\\
		&\stackrel{(\ref{def:PCM1})}{	=}	& m^{[0][0]} \otimes S(m^{[0][1]}) {m^{[1]}}_{(1)} \otimes {m^{[1]}}_{(2)}.
	\end{eqnarray*}
	The converse can be proven in a similar way.
\end{proof}

\begin{exmp} Every (global) right $H$-comodule is obviously a partial right $H$-comodule. It suffices to apply (higher) coassociativity to verify that axioms of a partial comodule are satisfied.
\end{exmp}

\begin{exmp} 
In case the Hopf $k$-algebra $H$ is finitely generated and projective as a $k$-module (i.e.\ finite dimensional if $k$ is a field), one can easily verify that partial comodules over $H$ are exactly partial modules over the dual Hopf algebra $H^*$ (for more details see \seref{dual}). In particular, if $G$ is a finite group, then  partial comodules over $kG$ coincide with partial modules over $kG^*$, that is, they are partially $G$-graded modules, as introduced in \cite{ABV}.
\end{exmp}

\begin{exmp} Let $G$ be an affine algebraic group and denote by $\mathcal{O}(G)$ the coordinate algebra of the variety $G$, which is a commutative Hopf algebra.
Let now $V$ be a finite dimensional vector space, then $\End_k (V)\cong M_n (k)$ can be regarded as the $n^2$-dimensional affine space, where $n=\dim_kV$ and we have that $\Oo(\End_k(V))$ is isomorphic to a full polynomial algebra in $n^2$ variables $x_{ij}$. There is a series of natural isomorphisms
$$\End_k(V,V \ot \Oo(G)) \cong  \Hom_k (\End(V), \Oo(G)) \cong \Hom_{\sf Alg}(\Oo(\End_k(V)), \Oo(G)) \cong \Hom_{\sf Aff}(G,\End_k(V))$$
where 
$\Hom_{\sf Aff}(G,\End_k(V))$ denotes the set of polynomial maps between the respective varieties.

Explicitly, a polynomial map $\pi :G \rightarrow \End_k (V)$ corresponds to a linear map  
\[
 \rho :  V  \rightarrow  V\otimes \mathcal{O}(G),  \ \ \ \ 
\,  e_i  \mapsto  \sum_{j=1}^n e_j \otimes a_{ji}
\]
where $a_{ij}=x_{ij}\circ \pi\in \Oo(G)$.

Then $\pi$ is a partial representation of the group $G$ if and only if it satisfies 
\begin{align}
	\pi(e)&=id_V, \\
	\label{partialgroup1}\pi (g)\pi (h) \pi (h^{-1})&=\pi (gh)\pi (h^{-1}),\\
	 \label{partialgroup2}
	\pi (g) \pi (g^{-1}) \pi (h) &=\pi (g) \pi (g^{-1}h),
\end{align}
for all $g,h\in G$ and where $e\in G$ is the unit element. By a classical calculation these equalities are respectively equivalent to the following identities for $\rho$
\begin{align}\label{algebraicgroup1}
\sum_{j=1}^n e_j a_{ji}(e) &= e_i,\\
\sum_{j,k,l=1}^{n} e_l \otimes a_{lk} \otimes a_{kj}  S(a_{ji}) &= \sum_{j,l=1}^n e_l \otimes {a_{lj}}_{(1)} \otimes {a_{lj}}_{(2)}  S(a_{ji}) ,\\
\label{algebraicgroup2}
\sum_{j,k,l=1}^n e_l \otimes a_{lk}  S(a_{kj}) \otimes a_{ji} &= \sum_{j,l=1}^n e_l \otimes a_{lj} S({a_{ji}}_{(1)}) \otimes {a_{ji}}_{(2)} ,
\end{align}
which are exactly (\ref{def:PCM1}), (\ref{def:PCM2}) and (\ref{def:PCM3}). 
\end{exmp}

The previous example leads to the following result
\begin{proposition}
Let $G$ be an affine algebraic group over a field $k$ and $\Oo(G)$ its coordinate Hopf algebra. The category of finite dimensional partial representations of an affine group $G$ is isomorphic to the category of finite dimensional partial comodules over the commutative Hopf algebra $\Oo(G)$.
\end{proposition}

\begin{exmp} Consider a symmetric right partial coaction $\rho :A\rightarrow A\otimes H$, of a Hopf algebra $H$ on a unital algebra $A$, then $(A, \rho )$ is a partial right $H$-comodule. Indeed, note first that the axiom (\ref{def:PCA2}) corresponds to (\ref{def:PCM1}). 

For the axiom (\ref{def:PCM2}), Take $a\in A$ and denote $\rho (a)=a^{[0]}\otimes a^{[1]}$, then
\begin{align*}
	a^{[0][0]}\ot {a^{[0][1]}}_{(1)} \ot {a^{[0][1]}}_{(2)} S(a^{[1]}) 
		& = 1^{[0]}a^{[0]} \otimes {1^{[1]}}_{(1)} {a^{[1]}}_{(1)} \otimes {1^{[1]}}_{(2)} {a^{[1]}}_{(2)} S({a^{[1]}}_{(3)})\\
		& = 1^{[0]}a^{[0]} \otimes {1^{[1]}}_{(1)} a^{[1]} \otimes {1^{[1]}}_{(2)}.
\end{align*}
On the other hand,
\begin{align*}
	a^{[0][0][0]}\ot a^{[0][0][1]}\ot a^{[0][1]}S(a^{[1]})
		&= 1^{[0][0]}a^{[0][0]} \otimes 1^{[0][1]}a^{[0][1]} \otimes 1^{[1]} {a^{[1]}}_{(1)}S({a^{[1]}}_{(2)}) \\
		&= 1^{[0][0]}a^{[0]} \otimes 1^{[0][1]}a^{[1]} \otimes 1^{[1]} \\
		&= 1^{[0]}1^{[0']}a^{[0]} \otimes {1^{[1]}}_{(1)} 1^{[1']} {a^{[1]}} \otimes {1^{[1]}}_{(2)} \\
		&= 1^{[0]}a^{[0]} \otimes {1^{[1]}}_{(1)} a^{[1]} \otimes {1^{[1]}}_{(2)} .
\end{align*}
Analogously, one can prove (\ref{def:PCM3}). 
Note that, in this proof, we strongly used the symmetry of the partial coaction in order to verify the axioms for partial comodule.
\end{exmp}

\begin{exmp} Consider a symmetric right partial comodule coalgebra with coaction $\rho:C\to C\ot H$. We will show that $\rho$ endows $C$ with a partial comodule structure.
Axiom (\ref{def:PCM1}) is exactly axiom (\ref{def:PCA2}). 
For (\ref{def:PCM2}), take any $c\in C$, then we find
\begin{eqnarray*}
	& & \hspace*{-2cm} c^{[0][0]}\otimes {c^{[0][1]}}_{(1)} \otimes {c^{[0][1]}}_{(2)}  S(c^{[1]}) =\\
	& \stackrel{(\ref{def:RPCC3})}{=} & \epsilon_C ({c_{(1)}}^{[0]}){c_{(2)}}^{[0]} \otimes {{c_{(2)}}^{[1]}}_{(1)} \otimes {{c_{(2)}}^{[1]}}_{(2)}S( {c_{(1)}}^{[1]} {{c_{(2)}}^{[1]}}_{(3)}) \\
	& = & \epsilon_C ({c_{(1)}}^{[0]}){c_{(2)}}^{[0]} \otimes {{c_{(2)}}^{[1]}}_{(1)} \otimes {{c_{(2)}}^{[1]}}_{(2)}S({{c_{(2)}}^{[1]}}_{(3)})  S( {c_{(1)}}^{[1]}) \\
	& = & \epsilon_C ({c_{(1)}}^{[0]}){c_{(2)}}^{[0]} \otimes {{c_{(2)}}^{[1]}}_{(1)} \otimes \epsilon_H ({{c_{(2)}}^{[1]}}_{(2)})  S( {c_{(1)}}^{[1]}) \\
	& = & \epsilon_C ({c_{(1)}}^{[0]}){c_{(2)}}^{[0]} \otimes {{c_{(2)}}^{[1]}} \otimes S( {c_{(1)}}^{[1]}) \; .
\end{eqnarray*}
On the other hand,
\begin{eqnarray*}
	& & \hspace*{-2cm} c^{[0][0][0]}\otimes {c^{[0][0][1]}} \otimes {c^{[0][1]}}  S(c^{[1]}) =\\
	& \stackrel{(\ref{def:RPCC3})}{=} & \epsilon_C ({c_{(1)}}^{[0]}){c_{(2)}}^{[0][0]} \otimes {{c_{(2)}}^{[0][1]}}\otimes {{c_{(2)}}^{[1]}}_{(1)}S( {c_{(1)}}^{[1]} {{c_{(2)}}^{[1]}}_{(2)}) \\
	& = & \epsilon_C ({c_{(1)}}^{[0]}){c_{(2)}}^{[0][0]} \otimes {{c_{(2)}}^{[0][1]}}\otimes {{c_{(2)}}^{[1]}}_{(1)} S({{c_{(2)}}^{[1]}}_{(2)}) 
	S( {c_{(2)}}^{[1]} ) \\
	& = & \epsilon_C ({c_{(1)}}^{[0]}){c_{(2)}}^{[0][0]} \otimes {{c_{(2)}}^{[0][1]}}\otimes \epsilon_H ({{c_{(2)}}^{[1]}}) 
	S( {c_{(1)}}^{[1]} ) \\
	& = & \epsilon_C ({c_{(1)}}^{[0]}){c_{(2)}}^{[0]} \otimes {{c_{(2)}}^{[1]}}\otimes S( {c_{(1)}}^{[1]} ).
\end{eqnarray*}
Axiom (\ref{def:PCM3}) is proven in the same way.
\end{exmp}

The previous examples lead to functors $F$ and $G$ as in the following proposition.
\begin{proposition}
Denote by ${\sf PCA^H}$ the category of partial comodule algebras over the Hopf algebra $H$ and by ${\sf PCC^H}$ the category of partial comodule coalgebra over $H$. Then there is a commutative diagram of functors
\[
\xymatrix{
{\sf PCA^H} \ar[drr]_{U'} \ar[rr]^F && \Mm^H_{par} \ar[d]^U && {\sf PCC^H} \ar[ll]_G \ar[lld]^{U''}\\
&& {\sf Mod}_k
}
\]
where $U$, $U'$ and $U''$ denote the obvious forgetful functors.
\end{proposition}

\begin{lemma} \lelabel{projection}
Given a partial $H$-comodule $M$, consider the linear map
\[
\begin{array}{rccl} \pi : & M\otimes H & \rightarrow & M\otimes H, \\
\, & m\otimes h & \mapsto & m^{[0][0]} \otimes m^{[0][1]} S(m^{[1] }) h . 
\end{array}
\]
\begin{enumerate}[(i)]
\item The map $\pi$ is a projection in the $k$-module $M\otimes H$, i.e.\ $\pi\circ\pi=\pi$.
\item Denoting by $M\bullet H =\pi (M\otimes H )$, we have that $\rho (M)\subseteq M\bullet H$.
\item For any $m \in M$, we have that $(\pi\ot I)(I\otimes \Delta)\rho (m)=(\rho \otimes I)\rho (m)$.
\item $(\rho \otimes I)\rho (m) -(I\otimes \Delta)\rho (m) \in \mbox{Ker} (\pi \otimes I )$.
\end{enumerate}  
\end{lemma}

\begin{proof} $\ul{(i)}$. Consider, $m\otimes h \in M\otimes H$, then
\begin{eqnarray*}
\pi \circ \pi (m\otimes h) 
& = & m^{[0][0][0][0]} \otimes m^{[0][0][0][1]} S(m^{[0][0][1]})m^{[0][1]}S(m^{[1]})h \\
& \stackrel{(\ref{def:PCM3})}{=} & m^{[0][0][0]} \otimes m^{[0][0][1]} S({m^{[0][1]}}_{(1)}){m^{[0][1]}}_{(2)} S(m^{[1]}) h \\
& = &  m^{[0][0][0]} \otimes m^{[0][0][1]} \epsilon (m^{[0][1]} )S(m^{[1]}) h \\
& \stackrel{(\ref{def:PCM1})}{=} & m^{[0][0]} \otimes m^{[0][1]} S(m^{[1]})h =
\pi (m\otimes h).
\end{eqnarray*}
Therefore, $\pi$ is a projection.\\
$\ul{(ii)}$. Take any element $m\in M$, then
\begin{eqnarray*}
\pi (\rho (m)) 
& = & m^{[0][0][0]} \otimes m^{[0][0][1]} S(m^{[0][1]})m^{[1]} \\
& \stackrel{(\ref{le:PCM4})}{=} & m^{[0][0]} \otimes {m^{[0][1]}}_{(1)} S({m^{[0][1]}}_{(2)}) m^{[1]} \\
& = &  m^{[0][0]} \otimes \epsilon (m^{[0][1]}) m^{[1]} \\
& \stackrel{(\ref{def:PCM1})}{=} & m^{[0]} \otimes m^{[1]} = 
\rho (m).
\end{eqnarray*}
This leads to $\rho (m) \in M\bullet H$.\\
$\ul{(iii)}$. Consider $m\in M$, then
\begin{eqnarray*}
(\pi\ot I)(I\otimes \Delta)\rho (m) 
&=& m^{[0][0][0]}\otimes m^{[0][0][1]} S(m^{[0][1]}){m^{[1]}}_{(1)}\otimes {m^{[1]}}_{(2)}\\
&\stackrel{(\ref{def:PCM3})}{=}& m^{[0][0]}\otimes m^{[0][1]} S(m^{[1]}_{(1)}){m^{[1]}}_{(2)}\otimes {m^{[1]}}_{(3)}\\
&=& m^{[0][0]}\otimes {m^{[0][1]}}\otimes {m^{[1]}}=(\rho \otimes I)\rho (m).
\end{eqnarray*}
$\ul{(iv)}$. Follows immediately from parts (ii) and (iii).
\end{proof}

\begin{proposition}\label{algebraicisquasi}
Every (algebraic) partial comodule is a quasi partial comodule in the sense of \cite{HV}. 
More precisely, the category of (algebraic) partial comodules is a full subcategory of the category of quasi partial comodules.
\end{proposition}

\begin{proof}
Let $M$ be a partial $H$-comodule. By \leref{projection}(ii), we know that $(M,M\bul H,\pi,\rho)$ is a partial comodule datum in the sense of \cite{HV}. Furthermore, from \leref{projection} (ii) it follows that the coassociativity holds in $(M\bul H)\ot H$, and therefore as well in any quotient of $(M\bul H)\ot H$, so $M$ is a quasi partial comodule.

One can easily verify that any morphism of quasi partial comodules between algebraic partial comodules is a morphism of (algebraic) partial comodules, and this concludes the proof. 
\end{proof}

In \cite{HV}, a fundamental theorem for (quasi and geometric) partial comodules was proven. Since we already observed earlier that any algebraic partial comodule is also a quasi partial comodule, we immediately obtain the following result

\begin{corollary}\label{fundamentaltheorem} Let $H$ be a Hopf algebra over a field $k$ and $M$ be a partial right $H$-comodule. Then, any element $m\in M$ is contained in a finite dimensional quasi partial $H$-subcomodule $M_m \subseteq M$. 
\end{corollary}

Unfortunately, there is no reason why the quasi partial $H$-subcomodule $M_m$ should be algebraic again. In fact, the following example shows that a fundamental theorem for (algebraic) partial comodules does not hold.

\begin{example}\label{irregular}
Let $k$ be a field and consider Sweedler's four dimensional Hopf algebra $H_4$, which is generated as a vector space by the unit $1$, the grouplike element $g$, and the skew primitive elements $x$ and $y$ satisfying 
\begin{eqnarray*}
g^2=1; &\qquad& xg=y=-gx;\\
\Delta(x)=g\ot x+ x\ot 1; && \Delta(y)=1\ot y + y\ot g;\\
S(x)=-y; && S(y)=x.
\end{eqnarray*}
Let $k[z]$ be the polynomial algebra in the variable $z$. We define the map
\begin{eqnarray*}
&&\rho: k[z]\to k[z]\ot H_4, \\
&&\rho(z^n)=z^{n+1}\ot y + {1\over 2}z^n\ot 1 + {1\over 2}z^n\ot g . 
\end{eqnarray*}
Then one can verify that $(k[z],\rho)$ is a right partial $H_4$-comodule. Moreover, from the formula of the definition, it is clear that the subcomodule generated by any (positive) power of $z$ is infinite dimensional.
\end{example}

The previous example motivates the following definition.

\begin{definition}
Let $M$ be partial $H$-comodule. The sum of all subcomodules of $M$ that are finitely generated as $k$-module is called the {\em regular} part of $M$. If $M$ equals its regular part, we say that $M$ is a {\em regular} partial comodule. A partial comodule that is not regular is called {\em irregular}.
\end{definition}

\section{Partial corepresentations and partial comodules of Hopf algebras}
\selabel{parcorep}

\subsection{Partial corepresentations}

We will now introduce the dual notion of a partial representation of a Hopf algebra. 
Again, this definition can be obtained by simply transferring the definition of a partial representation to the setting of monoidal categories and applying this general case to the opposite of the category of $k$-modules.
The basic idea is that, in the same way that partial representations measure how a $k$-linear map $\pi$ from a Hopf algebra $H$ to an algebra $B$ fails to be multiplicative, the partial corepresentations will measure how a $k$-linear map $\omega$ between a coalgebra $C$ and a Hopf algebra $H$ fails to be a morphism of coalgebras.

\begin{defi} Let $H$ be a Hopf algebra. A partial corepresentation of $H$ over a coalgebra $C$ is a linear map $\omega :C\rightarrow H$ satisfying the following conditions for all $c\in C$:
\begin{enumerate}[label=(PC\arabic*),ref=PC\arabic*]
	\item\label{def:PC1} $\epsilon_H \circ \omega(c) =\epsilon_C(c)$.
	\item\label{def:PC2} $\omega (c_{(1)})_{(1)} \otimes \omega (c_{(1)})_{(2)} S(\omega (c_{(2)})) = \omega (c_{(1)}) \otimes \omega (c_{(2)})S(\omega (c_{(3)}))$.
	\item\label{def:PC3} $\omega (c_{(1)}) S(\omega (c_{(2)})_{(1)}) \otimes \omega (c_{(2)})_{(2)} = \omega (c_{(1)}) S(\omega (c_{(2)})) \otimes \omega (c_{(3)})$.
\end{enumerate}

If $(C, \omega )$ and $(C' , \omega' )$ are two partial corepresentations of $H$, we say that a coalgebra morphism $f:C\rightarrow C'$ is a morphism of partial corepresentations if $\omega =\omega' \circ f$.

The category whose objects are partial corepresentations of $H$ and whose morphisms are morphisms of partial corepresentations is denoted by $\mbox{PCorep}^H$.
\end{defi}

\begin{lemma}
	Let $\omega\colon C \to H$ be a partial corepresentation. Then the following axioms are satisfied as well: 
	\begin{enumerate}[label=(PC\arabic*),ref=PC\arabic*,start=4]
		\item\label{le:PC4} $\omega (c_{(1)})_{(1)} \otimes S(\omega (c_{(1)})_{(2)}) \omega (c_{(2)}) = \omega (c_{(1)}) \otimes S(\omega (c_{(2)}))\omega (c_{(3)})$.
		\item\label{le:PC5} $S(\omega (c_{(1)})) \omega (c_{(2)})_{(1)} \otimes \omega (c_{(2)})_{(2)} = S(\omega (c_{(1)})) \omega (c_{(2)}) \otimes \omega (c_{(3)})$.
	\end{enumerate}
	Conversely, if a linear map $\omega\colon C \to H$ satisfies (\ref{def:PC1}), (\ref{le:PC4}) and (\ref{le:PC5}), then it is a partial corepresentation.
\end{lemma}
\begin{proof}Suppose that (\ref{def:PC1}), (\ref{def:PC2}) and (\ref{def:PC3}) are satisfied. Then for any $c\in C$ we have that
	\begin{eqnarray*}
		&								& \hspace*{-2cm}
										  \omega(c_{(1)})_{(1)} \otimes S(\omega(c_{(1)})_{(2)}) \omega(c_{(2)})\\
		&\stackrel{(\ref{def:PC1})}{=}	& \omega(c_{(1)})_{(1)} \otimes S(\omega(c_{(1)})_{(2)}) \omega(c_{(2)})S(\omega(c_{(3)})_{(1)})\omega(c_{(3)})_{(2)}\\
		&\stackrel{(\ref{def:PC3})}{=}	& \omega(c_{(1)})_{(1)} \otimes S(\omega(c_{(1)})_{(2)}) \omega(c_{(2)})S(\omega(c_{(3)}))\omega(c_{(4)})\\
		&\stackrel{(\ref{def:PC2})}{=}	& \omega(c_{(1)})_{(1)} \otimes S(\omega(c_{(1)})_{(2)}) \omega(c_{(1)})_{(3)}S(\omega(c_{(2)})\omega(c_{(3)})\\
		&							=	& \omega(c_{(1)}) \otimes S(\omega(c_{(2)})\omega(c_{(3)}).
	\end{eqnarray*}
	In a similar way
	\begin{eqnarray*}
		&								& \hspace*{-2cm}
										  S(\omega(c_{(1)})) \omega(c_{(2)})_{(1)} \otimes \omega(c_{(2)})_{(2)}\\
		&\stackrel{(\ref{def:PC1})}{=}	& S(\omega(c_{(1)})) \omega(c_{(2)})_{(1)} S(\omega(c_{(2)})_{(2)}) \omega(c_{(3)})_{(1)} \otimes \omega(c_{(3)})_{(2)}\\
		&\stackrel{(\ref{def:PC2})}{=}	& S(\omega(c_{(1)})) \omega(c_{(2)}) S(\omega(c_{(3)})) \omega(c_{(4)})_{(1)} \otimes \omega(c_{(4)})_{(2)}\\
		&\stackrel{(\ref{def:PC3})}{=}	& S(\omega(c_{(1)})) \omega(c_{(2)}) S(\omega(c_{(3)})_{(1)}) \omega(c_{(3)})_{(2)} \otimes \omega(c_{(3)})_{(3)}\\
		&							=	& S(\omega(c_{(1)})) \omega(c_{(2)}) \otimes \omega(c_{(3)}).
	\end{eqnarray*}
	The proof of the converse statement is analogous. 
\end{proof}

We will now show that there is a close relationship between partial comodules and partial corepresentations, which will allow us to provide some explicit examples of partial corepresentations. 

Recall that a (possibly non-unital) ring $R$ is called {\em firm} (also known as tensor-idempotent) if the multiplication map induces an isomorphism $R\ot_R R\to R$. Similarly, a left $R$-module is called {\em firm} if the action induces an isomorphism $R\ot_RM\cong M$, whose inverse we denote as $\varpi:M\to R\ot_RM,\ \varpi(m)=r\ot_R m^r$ (summation understood).
Furthermore, a right $A$-module $M$ is called {\em $R$-firmly projective} (see \cite{V}), if there is a ring morphism $\phi:R\to M\ot_A M^*,\ \phi(r)=e_r\ot f_r$ (summation understood) such that $M$ becomes a firm left $R$-module under the induced action $R\ot M\to M,\ r\cdot m=e_rf_r(m)$. Hence for an $R$-firmly projective $A$-module $M$, we have that $e_rf_r(m^r)=m$ for all $m\in M$.
In case $R$ has a unit $1$, then an $R$-firmly projective module is exactly a finitely generated and projective module, whose finite dual base is given by $e_1\ot f_1$ (where $1$ is the unit of $R$ and a summation is understood).

As follows from the results of \cite{V}, if $M$ is $R$-firmly projective as a right $A$-module then the functor $-\ot_RM:\Mod_R\to \Mod_A$ has a right adjoint given by $-\ot_AM^*$, and symmetrically, the functor $M^*\ot_R-:{}_R\Mod\to {}_A\Mod$ has a right adjoint $M\ot_A-$.

If $M$ is an $R$-firmly projective module, then we can construct (see \cite{GTV}) its {\em infinite comatrix coring} $C=M^*\ot_RM$, which is an $A$-coring with counit given by the evaluation map and comultiplication given by 
$$\Delta:M^*\ot_RM\to M^*\ot_RM\ot_A M^*\ot_RM,\ \ \ \ \Delta(f\ot_R u)=f\ot_R e_r\ot_A f_r\ot_R u^r.$$

\begin{theorem}\thlabel{comodcorep}
Let $H$ be a Hopf algebra over a commutative ring $k$, $M$ an $R$-firmly projective $k$-module and consider the associated infinite comatrix coalgebra $M^* \otimes_R M$.
Then there is a bijective correspondence between partial corepresentations $\omega :M^* \otimes_R M \rightarrow H$ and 
left $R$-linear maps $\rho :M\rightarrow M\otimes H, \rho(m)=m^{[0]}\ot m^{[1]}$ endowing $M$ with the structure of a partial $H$-comodule.
\end{theorem}

\begin{proof}
This is an immediate consequence of the adjunction $(M^*\ot_R-,M\ot-)$ described above.
More explicitly, given any linear map $\omega :M^* \otimes_R M \rightarrow H$, we define a linear map $\rho :M\rightarrow M\otimes H$ as
$$\rho(m)=e_r\omega(f_r\ot m^r).$$
Conversely, given $\rho:M\to M\ot H$, we define $\omega:M^* \otimes_R M \rightarrow H$ as
$$\omega(f\ot_R m)=f(m^{[0]})\ot m^{[1]}.$$
One can easily check that these constructions yield an isomorphism
$${_R\Hom}(M,M\ot H)\cong \Hom_k(M^*\ot_R M,H).$$
One can now verify that that under this correspondence, the axioms for a partial comodule translate exactly to the axioms for a partial corepresentation. 
\end{proof}

If the equivalent conditions of \thref{comodcorep} are satisfied for a $k$-module $M$, then we say that $M$ is firmly projective as partial right $H$-comodule. This means that $M$ is a partial $H$-comodule that is $R$-firmly projective as $k$-module for a firm ring $R\subset \End^H(M)$.

\begin{remark}\relabel{firm}
\begin{enumerate}
\item
The condition on the module $M$ to be $R$-firmly projective is not as strong as it might appear. Indeed, if $k$ is a field and $M$ is a regular partial $H$-comodule (e.g.\ $M$ is a global $H$-comodule), then $M$ is the sum of its finite dimensional subcomodules $M=\sum_{N\subset_{fd} M} N$. If we consider now $\Sigma=\oplus_{N\subset_{fd} M} N$ and $R=\oplus_{N\subset_{fd} M}\End^C(N)$, then $R$ is a ring with enough idempotents, in particular $R$ is a firm ring, and $\Sigma$ is $R$-firmly projective as $k$-module. Hence the previous theorem applies to this situation, and (partial) comodule structures on $M$ for which all $N$ are (partial) subcomodules correspond exactly to (partial) corepresentations from $\Sigma^*\ot_R\Sigma$ to $H$. This is exactly the viewpoint of the infinite comatrix corings developed in \cite{EGT}. 
\item 
\thref{comodcorep} can in particular be applied to finitely generated and projective modules $M$. It hence follows that there is a bijective correspondence between partial $C$-comodule structures on a finitely generated and projective $k$-module $M$ and partial corepresentations from the (finite) comatrix coalgebra $M^*\ot M$ into $C$.
\end{enumerate}
\end{remark}

Combining the previous result with the examples of partial comodules in the previous section, we immediately obtain the following examples of partial corepresentations.

\begin{examples}
\begin{enumerate}
\item
Let $H$ be a Hopf algebra and $\rho :A\rightarrow A\otimes H$ a symmetric partial coaction of $H$ on a unital algebra $A$ which is finitely generated and projective as a $k$-module. Then this partial coaction defines a partial corepresentation of $H$ relative to the coalgebra $A^*\otimes A$, given by $\omega(f\ot a)=f(a^{[0]})a^{[1]}$.
\item
Let $H$ be a Hopf algebra and $C$ be a partial right comodule coalgebra with coaction $\rho : C\rightarrow C\ot H$ written as $\lambda (c) =c^{[0]}\otimes c^{[1]}$. Suppose also that $C$ is finitely generated and projective as a $k$-module. Then there is a partial corepresentation of $H$ with relation to the coalgebra $C^* \otimes C$, given by $\omega (f \otimes c)=f(c^{[0]})c^{[1]}$.
\end{enumerate}
\end{examples}

The following result characterizes global corepresentations (i.e.\ coalgebra maps) amongst partial ones.

\begin{proposition}
Let $H$ be a Hopf algebra and $C$ a coalgebra.
A $k$-linear map $\omega :C\rightarrow H$ is a morphism of coalgebras if and only if it is a partial corepresentation of $H$ over $C$ that satisfies moreover 
\begin{equation}\label{PC6}
	S(\omega (c_{(1)})) \omega (c_{(2)}) =\epsilon (c)1_H ,
\end{equation}
for every $c\in C$.
\end{proposition}

\begin{proof}
Obviously, if $\omega :C\rightarrow H$ is a morphism of coalgebras, then it is a partial corepresentation of $H$ over $C$ and it satisfies \eqref{PC6}. Conversely, if $\omega:C\to H$ is a partial corepresentation satisfying \eqref{PC6}, then for any $c\in C$ we find
\beqnast
\omega (c_{(1)}) \otimes \omega (c_{(2)}) & = & \omega (c_{(1)}) \otimes \omega (c_{(2)}) \epsilon (c_{(3)})\\
& \stackrel{\eqref{PC6}}{=} & \omega (c_{(1)}) \otimes \omega (c_{(2)}) S(\omega (c_{(3)})) \omega (c_{(4)})\\
& \stackrel{(\ref{def:PC2})}{=} & \omega (c_{(1)})_{(1)} \otimes \omega (c_{(1)})_{(2)} S(\omega (c_{(2)})) \omega (c_{(3)})\\
& \stackrel{\eqref{PC6}}{=} & \omega (c_{(1)})_{(1)} \otimes \omega (c_{(1)})_{(2)} \epsilon (c_{(2)})\\
& = & \omega (c)_{(1)} \otimes \omega (c)_{(2)} 
\eqnast
and hence $\omega$ is a coalgebra morphism.
\end{proof}

The next result will be of use later.

\begin{prop} Let $H$ be a Hopf algebra, $C$ a coalgebra and $\omega :C\rightarrow H$ a partial corepresentation of $H$ over $C$. Then, for every $c\in C$ we have
\begin{eqnarray}
\omega (c_{(1)}) S(\omega (c_{(2)} ) ) \omega (c_{(3)}) & = & \omega (c), \label{ISI}\\
S(\omega (c_{(1)})) \omega( c_{(2)}) S(\omega (c_{(3)})) & = & S(\omega (c)). \label{SIS}
\end{eqnarray}
\end{prop}

\begin{proof} For identity (\ref{ISI}), we have
\begin{eqnarray*}
\omega (c_{(1)}) S(\omega (c_{(2)} ) ) \omega (c_{(3)}) & \stackrel{(\ref{le:PC4})}{=} &\omega (c_{(1)})_{(1)}  S(\omega (c_{(1)})_{(2)}) \omega (c_{(2)}) \\
& = &\epsilon (\omega (c_{(1)}))\omega (c_{(2)}) \stackrel{(\ref{def:PC1})}{=}\epsilon (c_{(1)})\omega (c_{(2)})=\omega (c).
\end{eqnarray*}
And for identity (\ref{SIS}), 
\begin{eqnarray*}
S(\omega (c_{(1)})) \omega( c_{(2)}) S(\omega (c_{(3)})) & \stackrel{(\ref{le:PC5})}{=} & S(\omega (c_{(1)})) \omega (c_{(2)})_{(1)}  S(\omega (c_{(2)})_{(2)} )\\
& = & S(\omega (c_{(1)})) \epsilon (\omega (c_{(2)})) \stackrel{(\ref{def:PC1})}{=} S(\omega (c_{(1)}))\epsilon (c_{(2)}) = S(\omega (c)).
\end{eqnarray*}
\end{proof}

\subsection{The partial cosmash coproduct}

All examples of partial corepresentations we have encountered so far are partial corepresentations over a comatrix coalgebra. 
Recall from \cite[Section 6]{BV}, that for a left partial $H$-comodule coalgebra $C$ with coaction $\lambda$, one can define a new coalgebra, namely the partial cosmash coproduct 
\[
\underline{C\cosmash H} =\text{span} \{ x\cosmash h = x_{(1)} \otimes {x_{(2)}}^{[-1]} \epsilon ({x_{(2)}}^{[0]}) \in C\otimes H \; | \; x\in C , \quad h\in H \} ,
\]
whose comultiplication $\underline{\Delta}: \underline{C\cosmash H} \rightarrow \underline{C\cosmash H} \otimes \underline{C\cosmash H}$ is given by
\[
\underline{\Delta} (x\cosmash h)=x_{(1)} \cosmash {x_{(2)}}^{[-1]} h_{(1)} \otimes {x_{(2)}}^{[0]} \cosmash h_{(2)} 
\]
and the counit $\underline{\epsilon} :\underline{C\cosmash H} \rightarrow k$ is given by $\epsilon (x\cosmash h)=\epsilon_C (x) \epsilon_H(h)$. The partial cosmash coproduct gives rise to another example of partial corepresentation of a Hopf algebra. Recall that the map $\nabla : C\rightarrow H$ given by $\nabla (c) =c^{[-1]} \epsilon (c^{[0]})$ is an idempotent in the convolution algebra $\Hom_k (C,H)$.

\begin{prop} \label{cosmash} Let $H$ be a Hopf algebra and $C$ be a left partial $H$-comodule coalgebra. Then the map
\[
\begin{array}{rccc} \omega_0 : & \underline{C\cosmash H} & \rightarrow & H \\
\, & x\cosmash h & \mapsto & \nabla (x) h 
\end{array}
\]
in which $\nabla (x)=x^{[-1]} \epsilon (x^{[0]})$, is a partial corepresentation of the Hopf algebra $H$ relative to the partial cosmash coproduct $\underline{C\cosmash H}$.
\end{prop}

\begin{proof} First, it is easy to see that $\omega_0$ is a well defined linear map, indeed
\begin{eqnarray*}
\omega_0 (x_{(1)}\cosmash {x_{(2)}}^{[-1]} \epsilon( {x_{(2)}}^{[0]}) h)& = & \omega_0 (x_{(1)}\cosmash \nabla (x_{(2)})h) = \nabla (x_{(1)})\nabla (x_{(2)})h \\
& = & \nabla (x) h=\omega_0 (x\cosmash h).
\end{eqnarray*}
Axiom (\ref{def:PC1}) is easily verified. Indeed, consider $x\cosmash h \in \underline{C\cosmash H}$, then
\[
\epsilon_H (\omega_0 (x\cosmash h)) =\epsilon_H (\nabla (x)h) =\epsilon_H (x^{[-1]}) \epsilon_C (x^{[0]}) \epsilon (h)\stackrel{(\ref{def:LPCC2})}{=} \epsilon_C (x) \epsilon_H (h) =\underline{\epsilon}(x\cosmash h).
\]
In order to check the other axioms of partial copreresentations, it will be useful to write 
\[
\underline{\Delta}^2 (x\cosmash h)= (\underline{\Delta} \otimes I)\circ \underline{\Delta}(x\cosmash h) =(I\otimes \underline{\Delta})\circ \underline{\Delta} (x\cosmash h) 
\]
in two different ways:
\begin{eqnarray*}
\underline{\Delta}^2 (x\cosmash h) & = & x_{(1)} \cosmash {x_{(2)}}^{[-1]} {{x_{(3)}}^{[-1]}}_{(1)} h_{(1)} \otimes {x_{(2)}}^{[0]} \cosmash {{x_{(3)}}^{[-1]}}_{(2)} h_{(2)} \otimes {x_{(3)}}^{[0]} \cosmash h_{(3)} \\
& = & x_{(1)} \cosmash {x_{(2)}}^{[-1]}{x_{(3)}}^{[-1]} h_{(1)} \otimes {x_{(2)}}^{[0]} \cosmash {x_{(3)}}^{[0][-1]} h_{(2)} \otimes {x_{(3)}}^{[0][0]} \cosmash h_{(3)} .
\end{eqnarray*}
Furthermore, let us recall the identity of  \cite[Lemma 6.3]{BV}
\begin{equation} \label{nablaid}
\nabla (x_{(1)}) {x_{(2)}}^{[-1]} \otimes {x_{(2)}}^{[0]} ={x_{(1)}}^{[-1]} \nabla (x_{(2)}) \otimes {x_{(1)}}^{[0]} = x^{[-1]}\otimes x^{[0]} .
\end{equation}
Now we are ready to verify axiom (\ref{def:PC2}). For $x\cosmash h \in \underline{C\cosmash H}$ we have
\begin{eqnarray*}
& \,  & \omega_0 (x_{(1)}\cosmash {x_{(2)}}^{[-1]} h_{(1)})_{(1)} \otimes \omega_0 (x_{(1)}\cosmash {x_{(2)}}^{[-1]} h_{(1)})_{(2)} S(\omega_0 ({x_{(2)}}^{[0]} \cosmash h_{(2)})) \\
& = & (\nabla (x_{(1)}) {x_{(2)}}^{[-1]} h_{(1)})_{(1)} \otimes (\nabla (x_{(1)}) {x_{(2)}}^{[-1]} h_{(1)})_{(2)} S(\nabla ({x_{(2)}}^{[0]}) h_{(2)})\\
& = & {{x_{(1)}}^{[-1]}}_{(1)} {{x_{(2)}}^{[-1]}}_{(1)} \epsilon ({x_{(1)}}^{[0]}) h_{(1)} \otimes {{x_{(1)}}^{[-1]}}_{(2)} {{x_{(2)}}^{[-1]}}_{(2)} h_{(2)} S(h_{(3)}) S(\nabla ({x_{(2)}}^{[0]}))\\
& \stackrel{(\ref{def:LPCC1})}{=} & {x^{[-1]}}_{(1)} h \otimes {x^{[-1]}}_{(2)} S(\nabla (x^{[0]}))\\
& = & {x^{[-1]}}_{(1)} h \otimes {x^{[-1]}}_{(2)} S(x^{[0][-1]}) \epsilon (x^{[0][0]})\\
& \stackrel{(\ref{def:LPCC3})}{=} & {{x_{(1)}}^{[-1]}}_{(1)} \epsilon ({x_{(1)}}^{[0]}){{x_{(2)}}^{[-1]}}_{(1)} h\otimes {{x_{(1)}}^{[-1]}}_{(2)} {{x_{(2)}}^{[-1]}}_{(2)} S({{x_{(2)}}^{[-1]}}_{(3)}) \epsilon ({x_{(2)}}^{[0]})\\
& = & {{x_{(1)}}^{[-1]}}_{(1)} \nabla (x_{(2)}) h \otimes {{x_{(1)}}^{[-1]}}_{(2)} \epsilon ({x_{(1)}}^{[0]}).
\end{eqnarray*}
On the other hand,
\begin{eqnarray*}
& \, & \omega_0 (x_{(1)} \cosmash {x_{(2)}}^{[-1]}{x_{(3)}}^{[-1]} h_{(1)}) \otimes \omega_0 ({x_{(2)}}^{[0]} \cosmash {x_{(3)}}^{[0][-1]} h_{(2)}) S(\omega_0 ( {x_{(3)}}^{[0][0]} \cosmash h_{(3)})) \\
& = & \nabla (x_{(1)}) {x_{(2)}}^{[-1]}{x_{(3)}}^{[-1]} h_{(1)} \otimes \nabla ( {x_{(2)}}^{[0]}) {x_{(3)}}^{[0][-1]} h_{(2)} S(h_{(3)}) S(\nabla ({x_{(3)}}^{[0][0]}))\\
&\stackrel{(\ref{nablaid})}{=} & {x_{(1)}}^{[-1]}{x_{(2)}}^{[-1]} h \otimes \nabla ({x_{(1)}}^{[0]}) {x_{(2)}}^{[0][-1]} S(\nabla ({x_{(2)}}^{[0][0]}))\\
& \stackrel{(\ref{def:LPCC1})}{=} & x^{[-1]}h \otimes \nabla ({x^{[0]}}_{(1)}){{x^{[0]}}_{(2)}}^{[-1]} S(\nabla ({{x^{[0]}}_{(2)}}^{[0]}))\\
& = &  x^{[-1]}h \otimes {{x^{[0]}}_{(1)}}^{[-1]} {{x^{[0]}}_{(2)}}^{[-1]} \epsilon ({{x^{[0]}}_{(1)}}^{[0]}) S(\nabla ({{x^{[0]}}_{(2)}}^{[0]}))\\
& \stackrel{(\ref{def:LPCC1})}{=} & x^{[-1]}h \otimes x^{[0][-1]} \epsilon ({x^{[0][0]}}_{(1)})S(\nabla ({x^{[0][0]}}_{(2)}))\\
& = & x^{[-1]}h \otimes x^{[0][-1]} S(\nabla (x^{[0][0]}))\\
& \stackrel{(\ref{def:LPCC3})}{=} & {{x_{(1)}}^{[-1]}}_{(1)} \nabla (x_{(2)})h\otimes {{x_{(1)}}^{[-1]}}_{(2)} S(\nabla ({x_{(1)}}^{[0]}))\\
& = & {{x_{(1)}}^{[-1]}}_{(1)} \nabla (x_{(2)})h\otimes {{x_{(1)}}^{[-1]}}_{(2)} S({x_{(1)}}^{[0][-1]}) \epsilon ({x_{(1)}}^{[0][0]}) \\
& \stackrel{(\ref{def:LPCC3})}{=} & {{x_{(1)}}^{[-1]}}_{(1)} \epsilon ({x_{(1)}}^{[0]}) {{x_{(2)}}^{[-1]}}_{(1)} \nabla (x_{(3)})h \otimes {{x_{(1)}}^{[-1]}}_{(2)} {{x_{(2)}}^{[-1]}}_{(2)} S({{x_{(2)}}^{[-1]}}_{(3)}) \epsilon ({x_{(2)}}^{[0]})\\
& = & {{x_{(1)}}^{[-1]}}_{(1)} \epsilon ({x_{(1)}}^{[0]}) \nabla (x_{(2)}) \nabla (x_{(3)})h\otimes {{x_{(1)}}^{[-1]}}_{(2)}\\
& = & {{x_{(1)}}^{[-1]}}_{(1)} \nabla (x_{(2)}) h \otimes {{x_{(1)}}^{[-1]}}_{(2)} \epsilon ({x_{(1)}}^{[0]}).
\end{eqnarray*}
The other axiom of partial corepresentations can be verified by similar computations. 
\end{proof}

\begin{prop} Let $H$ be a Hopf algebra and $C$ be a partial left $H$-comodule coalgebra. Then the map $\phi_0 : \underline{C\cosmash H} \rightarrow C$ given by $\phi_0 (x\cosmash h)=x\epsilon_H (h)$ is a coalgebra morphism.
\end{prop}

\begin{proof} First, to verify that the map $\phi_0$ is well defined, take an element $x\cosmash h \in \underline{C\cosmash H}$, then
\begin{eqnarray*}
\phi_0 (x_{(1)} \cosmash {x_{(2)}}^{[-1]} \epsilon_C ({x_{(2)}}^{[0]})h) & = & x_{(1)} \epsilon_H ({x_{(2)}}^{[-1]}) \epsilon_C ({x_{(2)}}^{[0]}) \epsilon_H (h)\\
& = & x_{(1)} \epsilon_C (x_{(2)})\epsilon_H (h)=x\epsilon_H (h) =\phi_0 (x\cosmash h).
\end{eqnarray*}
Let us now check that $\phi_0$ preserves the counit.
\[
\epsilon_C (\phi_0 (x\cosmash h))=\epsilon_C (x)\epsilon_H (h)=\underline{\epsilon} (x\cosmash h) .
\]
Finally, let us check that $\phi_0$ is comultiplicative.
\begin{eqnarray*}
(\phi_0 \otimes \phi_0 )\circ \underline{\Delta}(x\cosmash h) & = & \phi_0 (x_{(1)} \cosmash {x_{(2)}}^{[-1]} h_{(1)}) \otimes \phi_0 ( {x_{(2)}}^{[0]} \cosmash h_{(2)}) \\
& = & x_{(1)} \epsilon_H ({x_{(2)}}^{[-1]}) \epsilon_H (h_{(1)}) \otimes {x_{(2)}}^{[0]} \epsilon_H (h_{(2)})\\
& = &  x_{(1)} \otimes  x_{(2)} \epsilon_H (h)\\
& = & \Delta_C (x\epsilon_H (h)) = \Delta_C \circ \phi_0 (x\cosmash h).
\end{eqnarray*}
This concludes the proof.
\end{proof}

\begin{defi} \label{contravariant-pair} Let $H$ be a Hopf algebra, $C$  be a partial left $H$-comodule coalgebra and $D$ a coalgebra over a commutative ring $k$. We say that the pair of $k$-linear maps $(\phi , \omega )$ is a left contravariant pair relative to the coalgebra $D$ if
\begin{enumerate}[label=(CP\arabic*),ref=CP\arabic*]
\item\label{def:CP1} $\phi :D\rightarrow C$ is a coalgebra morphism and $\omega :D \rightarrow H$ is a partial corepresentation of $H$ relative to $D$.
\item\label{def:CP2} For every $x\in D$ we have 
\[
 \phi (x)^{[-1]}\otimes \phi (x)^{[0]} =\omega (x_{(1)}) S(\omega (x_{(3)})) \otimes \phi (x_{(2)}).
 \]
\item\label{def:CP3} For every $x\in D$, 
\[
\phi (x_{(1)}) \otimes  S(\omega (x_{(2)})) \omega (x_{(3)}) =\phi (x_{(3)}) \otimes  S(\omega (x_{(1)})) \omega (x_{(2)}) .
\]
\end{enumerate}
\end{defi}

The partial cosmash coproduct has a universal property with respect to left contravariant pairs as shown in the next theorem.

\begin{thm} Let $H$ be a Hopf algebra and  $C$  be a partial left $H$-comodule coalgebra. If $(\phi , \omega )$ is a left contravariant pair relative to a coalgebra $D$, then there exists a unique morphism of coalgebras $\Phi :D\rightarrow \underline{C\cosmash H}$ such that $\phi = \phi_0 \circ \Phi$ and $ \omega =\omega_0 \circ \Phi$.
\end{thm}

\begin{proof} Consider an element $x\in D$ then define
\[
\Phi (x)= \phi (x_{(1)}) \cosmash  \omega (x_{(2)}).
\]
Let us verify that this map is indeed a morphism of coalgebras. First, we have,
\[
\underline{\epsilon}(\Phi (x))=\epsilon_C (\phi (x_{(1)}))\epsilon_H (\omega (x_{(2)}))=\epsilon_D (x_{(1)})\epsilon_D (x_{(2)})=\epsilon_D (x).
\]
And, also
\begin{eqnarray*}
\underline{\Delta }(\Phi (x)) & = & \underline{\Delta}(\phi (x_{(1)}) \cosmash \omega (x_{(2)}))\\
& = & \phi (x_{(1)})_{(1)} \cosmash {\phi (x_{(1)})_{(2)}}^{[-1]} \omega (x_{(2)})_{(1)} \otimes {\phi (x_{(1)})_{(2)}}^{[0]} \cosmash \omega (x_{(2)})_{(2)}\\
& \stackrel{(\ref{def:CP1})}{=} & \phi (x_{(1)}) \cosmash {\phi (x_{(2)})}^{[-1]} \omega (x_{(3)})_{(1)} \otimes {\phi (x_{(2)})}^{[0]} \cosmash \omega (x_{(3)})_{(2)}\\
& \stackrel{(\ref{def:CP2})}{=} & \phi (x_{(1)}) \cosmash \omega (x_{(2)}) S(\omega (x_{(4)}))\omega (x_{(5)})_{(1)} \otimes \phi (x_{(3)}) \cosmash \omega (x_{(5)})_{(2)}\\
& \stackrel{(\ref{le:PC5})}{=} & \phi (x_{(1)}) \cosmash \omega (x_{(2)}) S(\omega (x_{(4)}))\omega (x_{(5)}) \otimes \phi (x_{(3)}) \cosmash \omega (x_{(6)})\\
& \stackrel{(\ref{def:CP3})}{=} & \phi (x_{(1)}) \cosmash \omega (x_{(2)}) S(\omega (x_{(3)}))\omega (x_{(4)}) \otimes \phi (x_{(5)}) \cosmash \omega (x_{(6)})\\
& \stackrel{(\ref{ISI})}{=} & \phi (x_{(1)}) \cosmash \omega (x_{(2)}) \otimes \phi (x_{(3)}) \cosmash \omega (x_{(4)})\\
& = & \Phi (x_{(1)}) \otimes \Phi (x_{(2)}) 
 =  (\Phi \otimes \Phi )\circ \Delta_D (x) .
\end{eqnarray*}
Therefore, the map $\Phi$ is a morphism of coalgebras.

For the uniqueness, consider a coalgebra map $\Psi : D \rightarrow \underline{C\cosmash H}$ such that $\phi =\phi_0 \circ \Psi$ and $\omega =\omega_0 \circ \Psi$. Denoting by $\pi : C\otimes H \rightarrow \underline{C\cosmash H}$ the canonical projection, we have for an element $x\in D$.
\[
\pi \circ (\phi_0 \otimes \omega_0 ) \circ ( \Psi \otimes \Psi )\circ \Delta_D (x) =\phi (x_{(1)}) \cosmash \omega (x_{(2)}) =\Phi (x) .
\]
On the other hand, denoting $\Psi (x) =\sum_i c_i \cosmash h_i$, we have
\begin{eqnarray*}
\pi \circ (\phi_0 \otimes \omega_0 ) \circ ( \Psi \otimes \Psi )\circ \Delta_D (x)
& = & \pi \circ (\phi_0 \otimes \omega_0 ) \circ \underline{\Delta}( \sum_i c_i \cosmash h_i ) \\
& = & \sum_i \phi_0 ({c_i}_{(1)}  \cosmash {{c_i}_{(2)}}^{[-1]} {h_i}_{(1)} ) \cosmash \omega_0 ( {{c_i}_{(2)}}^{[0]} \cosmash {h_i}_{(2)} )\\
& = & \sum_i {c_i}_{(1)} \epsilon_H ({{c_i}_{(2)}}^{[-1]}) \epsilon_H ( {h_i}_{(1)}) \cosmash \nabla ({{c_i}_{(2)}}^{[0]}) {h_i}_{(2)}\\
& = & \sum_i {c_i}_{(1)} \cosmash \nabla ({c_i}_{(2)}) h_i 
 =  \sum_i c_i \cosmash h_i = \Psi (x).
\end{eqnarray*}
Therefore, $\Psi (x) =\Phi (x)$ for every $x\in D$. This finishes the proof.
\end{proof}

Throughout this paper, we will need the versions of these constructions and results for partial right comodule coalgebras and right cosmash coproducts. Let $C$ be a partial right $H$-comodule coalgebra, with right coaction denoted by $\rho (c)=c^{[0]}\otimes c^{[1]}$, then we can construct the right cosmash coproduct $\underline{H\rcosmash C} \subseteq H\otimes C$ spanned by elements of the form
\[
h\rcosmash x =h\widetilde{\nabla} (x_{(1)}) \otimes x_{(2)}= h\epsilon ({x_{(1)}}^{[0]}){x_{(1)}}^{[1]} \otimes x_{(2)} , 
\]
in which $\widetilde{\nabla}: C\rightarrow H$ is defined by $\widetilde{\nabla}(x) =\epsilon(x^{[0]}) x^{[1]}$. This right cosmash coproduct is a coalgebra with comultiplication
\[
\underline{\Delta} (h\rcosmash x)=h_{(1)}\rcosmash {x_{(1)}}^{[0]} \otimes h_{(2)} {x_{(1)}}^{[1]} \rcosmash x_{(2)} ,
\]
and counit $\underline{\epsilon} (h\rcosmash x )=\epsilon_H (h)\epsilon_C (x)$. There is a morphism of coalgebras $\widetilde{\phi}_0 :\underline{H\rcosmash C} \rightarrow C$, given by ${\widetilde{\phi}_0 (h\rcosmash x)=\epsilon_H (h)x}$, and a partial corepresentation $\widetilde{\omega}_0 :\underline{H\rcosmash C} \rightarrow H$, given by $\widetilde{\omega}_0 (h\rcosmash x)=h\widetilde{\nabla}(x)$.
These two maps form a right contravariant pair, that is, for every $(h\rcosmash  x)\in \underline{H\rcosmash C}$ we have
\[
\widetilde{\phi}_0 (h\rcosmash x)^{[0]} \otimes  \widetilde{\phi}_0 (h\rcosmash x)^{[1]} = \widetilde{\phi}_0 ((h\rcosmash x)_{(2)}) \otimes S (\widetilde{\omega}_0 ((h\rcosmash x)_{(1)}) ) \widetilde{\omega}_0 ((h\rcosmash x)_{(3)})  
\]
and 
\[
\widetilde{\phi}_0 ((h\rcosmash x)_{(1)}) \otimes \widetilde{\omega}_0 ((h\rcosmash x)_{(2)}) S (\widetilde{\omega}_0 ((h\rcosmash x)_{(3)}) )  =\widetilde{\phi}_0 ((h\rcosmash x)_{(3)}) \otimes  \widetilde{\omega}_0 ((h\rcosmash x)_{(1)}) S (\widetilde{\omega}_0 ((h\rcosmash x)_{(2)}) ).
\]
Moreover, for any coalgebra $D$ and any right contravariant pair $(\phi :D\rightarrow C , \omega :D\rightarrow H )$, there is a unique morphism of coalgebras $\widetilde{\Phi} :D\rightarrow \underline{H\rcosmash C}$ factoring both $\phi$ and $ \omega$.

\subsection{Dualities}\selabel{dual}

\begin{proposition}
Let $(H,K,\bk{-,-})$ be a pairing of Hopf algebras, and $(C,A,\bk{-,-})$ a pairing between a coalgebra $C$ and an algebra $A$.
Consider a morphism of pairings $(\omega,\pi):(C,A,\bk{-,-})\to (H,K,\bk{-,-})$
\begin{enumerate}
\item If the pairing $(C,A,\bk{-,-})$ is right non-degenerate and $\omega$ is a partial corepresentation, then $\pi$ is a partial representation.
\item If $k$ is a field, the pairing $(H,K,\bk{-,-})$ is left non-degenerate and $\pi$ is a partial representation, then $\omega$ is a partial corepresentation.
\end{enumerate}
\end{proposition}

\begin{proof}
As we have observed before, the notions of partial representation and partial corepresentation are categorical duals. Hence the result follows directly from the observations made in \leref{pairings} and Proposition \ref{dualpairing}, following the same reasoning as in Remark \ref{re:dualpairing} in the global case. In part (2), observe that since the axioms for $\omega$ being a partial corepresentation need to be verified in the tensor product $H\ot H$, the condition that $k$ is a field is needed in order to obtain that a tensor product of non-degenerate pairings is again non-degenerate (see Proposition \ref{dualpairing}). 
\end{proof}

Considering the pairings $(C,C^*)$, $(A^\circ,A)$ and $(H,H^\circ)$, we immediately obtain the following.

\begin{corollary}
\begin{enumerate}
\item Let $H$ be a Hopf algebra, $C$ a coalgebra. Let $\omega:C\rightarrow H$ be a linear map and $\omega^*:H^\circ\to C^*$ its dual. 
\begin{enumerate}[(i)]
\item If $\omega :C\rightarrow H$ is a partial corepresentation, then $\omega^*$ is a partial representation.
\item If $k$ is a field, $H^\circ$ is dense in $H^*$ and $\omega^*$ is a partial representation, then $\omega$ is a partial corepresentation.
\end{enumerate}
\item Let $H$ be a Hopf algebra and $A$ be a algebra. Let $\pi: H\to A$ be a linear map and suppose that the image of the dual map $\pi^*:A^\circ\to H^*$ lies in $H^\circ$, then
\begin{enumerate}[(i)]
\item if $k$ is a field, $H^\circ$ is dense in $H^*$ and $\pi$ is a partial representation, then $\pi^*:A^\circ\to H^\circ$ is a partial corepresentation.
\item if $A^\circ$ is dense in $A^*$ and $\pi^*:A^\circ\to H^\circ$ is a partial corepresentation, then $\pi$ is a partial representation.
\end{enumerate}
\end{enumerate}
\end{corollary}

\begin{thm} \label{characterizationpcomod1} Let $H$ be a 
Hopf algebra over a field $k$ such that $H^\circ$ is dense in $H^*$, $M$ a $k$-vector space and $\rho :M\rightarrow M\otimes H$ a linear map . Then the following statements are equivalent.
\begin{enumerate}[(i)]
\item $(M, \rho)$ is a right partial $H$-comodule;
\item $(M,\lambda)$ is a left partial $H^\circ$-module, where $\lambda:H^\circ\ot M\to M$ is defined by $\lambda(h^*\ot m)=(I\otimes h^* )\rho(m)$;
\item the linear map $\pi :H^\circ \rightarrow \mbox{End}_k (M)$ given by $\pi (h^*)(m)=(I\otimes h^* )\rho (m)$ defines a partial representation of $H^\circ$.
\end{enumerate}
If moreover $H$ is finite dimensional, then the existence of a map $\lambda:H^*\ot M\to M$, which endows $M$ with a structure of a partial $H^*$-module, implies the existence of a map $\rho:M\to M\ot H$ such that $(M,\rho)$ is a partial right $H$-comodule and $\lambda$ is related to $\rho$ as in (ii) above. Consequently, over a finite dimensional Hopf algebra $H$, the category of partial $H^*$-modules and the category of partial $H$-comodules are isomorphic.
\end{thm}

\begin{proof} 
We check the equivalence between (i) and (iii), the equivalence between (ii) and (iii) follows from the definitions (see \cite{ABV}).

For any $m\in M$, denote $\rho (m)=m^{[0]}\otimes m^{[1]}$, and then the map $\pi$ can be written as $\pi (h^* )(m)=m^{[0]}h^* (m^{[1]})$. First, we have
\[
\pi (1_{H^\circ })(m) =m^{[0]} 1_{H^\circ} (m^{[1]}) =m^{[0]} \epsilon_H (m^{[1]}) .
\]
Then $\pi (1_{H^\circ })=\mbox{Id}_M$ if, and only if, the map $\rho$ satisfies (\ref{def:PCM1})

For any $m\in M$ and $h^*, k^* \in H^\circ$ we have
\begin{eqnarray*}
 (I\otimes h^* \otimes k^* )(m^{[0][0]}\otimes {m^{[0][1]}}_{(1)} \otimes {m^{[0][1]}}_{(2)} S(m^{[1]}))
& = & m^{[0][0]} h^* ({m^{[0][1]}}_{(1)}) k^* ({m^{[0][1]}}_{(2)} S(m^{[1]}))\\
& = & m^{[0][0]} h^* ({m^{[0][1]}}_{(1)}) k^*_{(1)} ({m^{[0][1]}}_{(2)}) k^*_{(2)}( S(m^{[1]}))\\
& = & m^{[0][0]} (h^* k^*_{(1)}) ({m^{[0][1]}}) S(k^*_{(2)})( m^{[1]})\\
& = & \pi (h^* k^*_{(1)})(m^{[0]} S(k^*_{(2)})( m^{[1]}))\\
& = & \pi (h^* k^*_{(1)}) \pi (S(k^*_{(2)}))(m).
\end{eqnarray*}
On the other hand, 
\begin{eqnarray*}
 (I\otimes h^* \otimes k^* )(m^{[0][0][0]}\otimes m^{[0][0][1]} \otimes m^{[0][1]} S(m^{[1]}))
& = & m^{[0][0][0]} h^* (m^{[0][0][1]}) k^* (m^{[0][1]} S(m^{[1]}))\\
& = & m^{[0][0][0]} h^* (m^{[0][0][1]}) k^*_{(1)} (m^{[0][1]} ) k^*_{(2)} (S(m^{[1]}))\\
& = & \pi (h^*) (m^{[0][0]} k^*_{(1)} (m^{[0][1]} ) S(k^*_{(2)}) (m^{[1]}))\\
& = & \pi (h^*) \pi (k^*_{(1)}) (m^{[0]} S(k^*_{(2)}) (m^{[1]}))\\
& = & \pi (h^*) \pi (k^*_{(1)}) \pi (S(k^*_{(2)}))(m).
\end{eqnarray*}
By the non-degeneracy of the pairing $(H,H^\circ)$, it then follows that the linear map $\rho$ satisfies (\ref{def:PCM2}) if and only if the linear map $\pi$ satisfies (\ref{def:PR2}). The other axiom is checked in the same way, leading to the conclusion that $(M,\rho)$ is a partial right $H$-comodule, if and only if the map $\pi :H^\circ \rightarrow \mbox{End}_k (M)$ is a partial representation of $H^\circ$.

For the last statement, given the linear map $\lambda$, we define $\rho(m)=\sum_i \lambda(h^*_i\ot m)\ot h_i$, where $\{(h_i,h^*_i)\}\subset H\times H^*$ is a finite dual base for $H$.
\end{proof}

\begin{example}
In \cite[Example 4.13]{ABV}, it was shown that Sweedler's four-dimensional Hopf algebra $H_4$ acts partially on the finitely generated algebra $k[x,z]/(2x^2-x,2xz-z)$, which contains the full polynomial algebra $k[z]$. Applying Theorem \ref{characterizationpcomod1} and using the fact that $H_4$ is self-dual, we find that $k[z]$ is a partial $H_4$-comodule. This comodule is exactly the irregular partial comodule from Example \ref{irregular}. Moreover, since this partial comodule is irregular it also follows from the results in this section that this partial comodule is not induced by a partial corepresentation.
\end{example}

\section{Universal coalgebras}\selabel{universal}

For the sake of simplicity we suppose from now on that $k$ is a field. However, many constructions still hold if $k$ is a general commutative ring and the considered $k$-modules are supposed to be (locally) projective.

\subsection{The universal corep-coalgebra}

Let $H$ be a Hopf algebra over a field $k$ and $(C(H),p)$ the cofree coalgebra over the vector space $H$.
Consider the set $\Lambda$  of all subcoalgebras $X$ of $C(H)$ such that $p|_{_X}:X \rightarrow H$ is a partial corepresentation. Then we define
\[
H^{par} = \sum_{X\in \Lambda} X,
\]
which is by construction the biggest sub-coalgebra of $C(H)$, such that $p|_{_{H^{par}}}$ is a partial corepresentation. We call $H^{par}$ the {\em universal corep-coalgebra}. By construction $H^{par}$ is cogenerated (as a coalgebra) by $H$.

\begin{lemma}\lelabel{parcorepcoalgmap}
With notation as above, the following assertions hold.
\begin{enumerate}[(i)]
\item
Given a coalgebra $C$ and a partial corepresentation $\omega :C \rightarrow H$, there exists a unique coalgebra map $\bar{\omega}: X \rightarrow H^{par}$ such that $\omega = p \circ \bar{\omega}$.
\item
$H$ is a subcoalgebra and a $k$-linear direct summand of $H^{par}$.
\end{enumerate}
\end{lemma}
	
\begin{proof} 
\ul{(i)}. By the universal property of the cofree coalgebra, there exists a unique coalgebra map $\bar{\omega}: C \rightarrow  C(H)$ such that $\omega = p \circ \bar{\omega}$. Since $\bar{\omega}$ is a coalgebra map, $\mbox{Im}(\bar{\omega})$ is a subcoalgebra of $C(H)$ and clearly $p|_{_{\mbox{Im}(\bar{\omega})}}$ is a partial corepresentation. Therefore, $\mbox{Im}(\bar{\omega})$ is an element of $\Lambda$ and we can take the corestriction $\bar{\omega}:C \to H^{par}$, which is the unique coalgebra morphism such that $\omega = p \circ \bar{\omega}$.\\
\ul{(ii)}. Since the identity map in $H$ is a partial corepresentation, there exists a coalgebra  map $\imath :H \rightarrow H^{par}$ such that $\mbox{Id}_H =p\circ \imath$. 
\end{proof}

\begin{thm} \thlabel{partcomodarecomod1}
Let $H$ be a Hopf algebra over a field $k$. 
\begin{enumerate}[(i)]
\item For any $H^{par}$-comodule $(M,\hat\rho)$ the pair $(M,\rho)$, where $\rho=(I\ot p)\circ \hat\rho$, is a partial $H$-comodule. This construction yields a fully faithful functor $\Gamma:\Mod^{H^{par}}\to \Mod^H_{par}$. 
\item The functor $\Gamma$ restricts to an isomorphism of categories $\Gamma_f:\Mod^{H^{par}}_f\to \Mod^H_{par,f}$, between the categories of firmly projective $H^{par}$-comodules and firmly projective partial $H$-comodules.
\item In particular, $\Gamma$ restricts to an isomorphism of categories $\Gamma_{fd}:\Mod^{H^{par}}_{fd}\to \Mod^H_{par,fd}$ between the categories of finite dimensional $H^{par}$-comodules and finite dimensional partial $H$-comodules. Hence, $H^{par}$ is the coalgebra that can be reconstructed from the fibre functor $U:\Mod^H_{par,fd}\to \Vect$ by means of the Tannakian formalism. 
\end{enumerate}
\end{thm}

\begin{proof} 
\ul{(i)}.
Consider a right $H^{par}$-comodule $(M, \hat{\rho})$, where we denote $\hat{\rho} (m)=m^{[0]} \otimes m^{[1]} \in M\otimes H^{par}$ for all $m\in M$. Let us verify that  $\rho :M\rightarrow M\otimes H, \rho (m)=m^{[0]}\otimes p(m^{[1]})$ endows $M$ with the structure of a partial $H$-comodule.
For any $m\in M$, we find that
$$
m^{[0]} \epsilon_H (p(m^{[1]})) =m^{[0]} \epsilon (m^{[1]})=m$$
hence axiom (\ref{def:PCM1}) holds.
To check axiom (\ref{def:PCM2}), we compute
\begin{eqnarray*}
m^{[0][0]} \otimes p( m^{[0][1]})_{(1)}\otimes  p( m^{[0][1]})_{(2)} S(p(m^{[1]})) 
& = & m^{[0]} \otimes p( {m^{[1]}}_{[1]})_{(1)}\otimes  p( {m^{[1]}}_{(1)})_{(2)} S(p({m^{[1]}}_{(2)})) \\
& \stackrel{(\ref{def:PC2})}{=} & m^{[0]} \otimes p( {m^{[1]}}_{(1)}) \otimes  p( {m^{[1]}}_{(2)}) S(p({m^{[1]}}_{(3)})) \\
& = & m^{[0][0][0]} \otimes p( m^{[0][0][1]})\otimes p( m^{[0][1]}) S( p(m^{[1]})) 
\end{eqnarray*}
where we used the coassociativity of the global comodule $(M,\hat\rho)$ in the first and last equalities.
In the same way, one can check axiom (\ref{def:PCM3}) 
so that we can conclude that $(M, \rho)$ is indeed a partial right $H$-comodule. 

If $f:M\to N$ is a morphism of $H^{par}$-comodules, then the left hand side of the diagram below commutes (and the right hand side commutes by naturality of the tensor product).
\[
\xymatrix{
M \ar[d]_f \ar[rr]^-{\hat\rho} && M\ot H^{par} \ar[rr]^-{id\ot p} \ar[d]_{f\ot id} && M\ot H \ar[d]_{f\ot id}\\
N \ar[rr]^-{\hat\rho} && N\ot H^{par} \ar[rr]^-{id\ot p}  && M\ot H
}
\]
This shows that $\Gamma$ is functorial by means of $\Gamma(f)=f$ and so clearly faithful. To prove that $\Gamma$ is also full,
let $(M,\hat\rho_M)$ and $(N,\hat\rho_N)$ be two right $H^{par}$-comodules and consider a $k$-linear map $f:M\to N$ that is $H$-colinear for the induced partial $H$-comodules $\Gamma(M)$ and $\Gamma(N)$. 
Recall that $H^{par}$ is cogenerated by $H$, hence the morphisms $p_n=p^{\ot n}\circ \Delta^n_{H^{par}}:H^{par}\to H^{\ot n}$ are jointly monic.  Therefore $f$ is $H^{par}$-colinear if and only if $p_n\circ (f\ot id)\circ \hat\rho_M = p_n\circ \hat\rho_N\circ f$ for all $n$. As one can see from the following diagram, using coassociativity of the $H_{par}$-coactions and the definition of the induced partial $H$-coactions, we find that the left hand side of the equation equals $(f\ot id)\circ \rho_M^{n+1}$ and the right hand side equals $\rho_N^{n+1}\circ f$. Hence both are equal as $f$ is a morphism of partial $H$-comodules.
\[
\xymatrix{
M \ar[d]_f 
\ar@/^{2pc}/[rrrrrr]^-{\rho^{n+1}_M}
\ar[rr]^-{\hat\rho_M} && M\ot H^{par} \ar[d]_{f\ot id} \ar@<-.5ex>[rr]_-{id\ot\Delta^n} 
\ar@<.5ex>[rr]^-{\hat\rho^n_M\ot 1} 
&& M\ot (H^{par})^{\ot n} 
\ar[d]_{f\ot id} \ar[rr]^{id\ot p^{\ot n}} \ar[d]_{f\ot id} && M\ot H^{\ot n} \ar[d]_{f\ot id}\\
N \ar[rr]^-{\hat\rho_N} \ar@/_{2pc}/[rrrrrr]_-{\rho^{n+1}_N} && N\ot H^{par} \ar@<.5ex>[rr]^-{id\ot\Delta^n} \ar@<-.5ex>[rr]_-{\hat\rho^n_N\ot 1} && N\ot (H^{par})^{\ot n}\ar[rr]^{id\ot p^{\ot n}} && N\ot H^{\ot n}
}
\]
\ul{(ii)}. 
Since we already know from part (i) that the functor $\Gamma$ is fully faithful, it suffice to prove that its restriction to firmly projective $H^{par}$-comodules is surjective on objects.
Therefore, consider an $R$-firmly projective partial right $H$-comodule $(M, \rho)$, where $R$ is a firm subring of $\End^H(M)$. By \thref{comodcorep}, $\rho$ induces a partial corepresentation $\omega:M^*\ot_R M\to H$, given by $\omega(f\ot m)=f(m^{[0]})m^{[1]}$. In turn, $\omega$ gives rise to a coalgebra map $\bar\omega:M^*\ot_R M\to H^{par}$ by \leref{parcorepcoalgmap}. Finally, $\omega$ endows $M$ with the structure of a $H^{par}$-comodule by defining $\hat\rho(m)=\sum_{i} e_r\ot \bar\omega(f_r\ot m^r)$ (see \thref{comodcorep}). 
\\ \ul{(iii)}. Is clear from (ii) and the Tannaka duality.
\end{proof}

\begin{remark}
Since $H^{par}$ is a usual coalgebra, its category of comodules satisfies the fundamental theorem, hence the image of the functor $\Gamma$ lies in the full subcategory of regular partial $H$-comodules. 
Hence it follows that if a Hopf algebra $H$ allows irregular partial modules (such as the Sweedler Hopf algebra $H_4$, see Example \ref{irregular}), then the category of partial $H$-modules cannot be equivalent to the category of comodules over $H^{par}$. 
\end{remark}

\begin{corollary}
Let $H$ be a Hopf algebra over a field $k$, such that $H^\circ$ is dense in $H^*$. Then
$$H^{par}\cong ((H^\circ)_{par})^\circ.$$
\end{corollary}

\begin{proof}
From Theorem \ref{characterizationpcomod1} we know that a (finite dimensional) partial $H$-comodule is the same as a (finite dimensional) partial $H^\circ$-module, so also a finite dimensional $(H^\circ)_{par}$-module. Hence (see \thref{partcomodarecomod1}(iii)) both $H^{par}$ and $((H^\circ)_{par})^\circ$ arise from the same fibre functor by means of the Tannaka-Krein duality, so they are isomorphic.
\end{proof}

From the previous corollary, we immediately find that there exists a pairing between $H^{par}$ and $H^\circ_{par}$.
We will now give an explicit description of this pairing, and show that it is non-degenerate.

Recall from Proposition \ref{pairingcofree}, that there is a non-degenerate dual pairing between the cofree coalgebra $C(H)$ and the tensor algebra $T(H^*)$. Consider the Sweedler dual $H^\circ\subset H^*$, we clearly have that $T(H^\circ)\subset T(H^*)$ and hence we obtain by restriction a pairing
\begin{eqnarray*}
&\bk{-,-}: C(H)\ot T(H^\circ) \to k,
&\langle x, h_{i_1}^* \otimes \cdots \otimes h_{i_n}^* \rangle =h_{i_1}^*(p(x_{(1)})) \cdots  h_{i_n}^*(p(x_{(n)}))
\end{eqnarray*}
where $p:C(H)\to H$ is the projection of the cofree coalgebra.
As already explained at the end of \seref{cofree}, if $H^\circ$ is dense in $H^*$, this pairing is still non-degenerate.

\begin{prop} \label{pairingHpar} 
Let $H$ be a Hopf algebra such that $H^\circ$ is dense in $H^*$.
The subcoalgebra $H^{par}\subseteq C(H)$ is the annihilator space $I^{\perp}$ of the ideal $I\trianglelefteq TH^\circ$ which defines the relations of $H^\circ_{par}$ (see Definition \ref{hpar}). As a consequence, there is a nondegenerate dual pairing
\begin{equation}\label{defpairingHpar}
\begin{array}{rccl} \left( \; , \; \right) : & H^{par} \otimes H^\circ_{par} & \rightarrow & k \\ 
& x \ot [k^{1*}]\ldots [k^{n*}] &\mapsto  &
k^{1*}(p(x_{(1)})) \ldots k^{n*}(p(x_{(n)}))
\end{array} .
\end{equation}
\end{prop}

\begin{proof} First, let us verify that the subcoalgebra $I^{\perp}$ of the cofree coalgebra $C(H)$ coincides with the subcoalgebra $H^{par} \subseteq C(H)$. 
Let $X$ be any subcoalgebra of $C(H)$.
Then we find for any $y\in X$, 
\[
\langle y, 1_{H^\circ} -1_{TH^\circ } \rangle = \langle y,1_{H^\circ } \rangle -\langle y,1_{TH^\circ }  \rangle =\epsilon_H (p(y)) -\epsilon_{C(H)} (y) =
(\epsilon_H\circ p -\epsilon_{C(H)})(y).
\]
Hence $p|_{_X}$ satisfies axiom (\ref{def:PC1}) if and only if $X\subset (1_{H^\circ} -1_{TH^\circ })^\perp$.

Furthermore, for any $y\in X$ as before and any $h^* , k^* \in H^\circ$, we find
\begin{eqnarray*}
&& \hspace*{-2cm} \langle y, h^*\otimes k^*_{(1)} \otimes S_{H^\circ}(k^*_{(2)}) - h^* k^*_{(1)} \otimes S_{H^\circ}(k^*_{(2)})\rangle =\\
& = & h^* (p(y_{(1)})) k^*_{(1)} (p(y_{(2)})) S_{H^\circ}(k^*_{(2)}) (p(y_{(3)})) -h^* k^*_{(1)} (p(y_{(1)})) S_{H^\circ}(k^*_{(2)}) (p(y_{(2)}))\\
& = & h^* (p(y_{(1)})) k^*_{(1)} (p(y_{(2)})) k^*_{(2)} (S_H(p(y_{(3)}))) -h^* (p(y_{(1)})_{(1)}) k^*_{(1)} (p(y_{(1)})_{(2)}) k^*_{(2)} (S_H(p(y_{(2)})))\\
& = & h^* (p(y_{(1)})) k^* (p(y_{(2)}) S_H(p(y_{(3)}))) -h^* (p(y_{(1)})_{(1)}) k^*(p(y_{(1)})_{(2)} S_H(p(y_{(2)})))\\
& = & (h^* \otimes k^* )\left( p(y_{(1)})\otimes p(y_{(2)}) S_H(p(y_{(3)})) -p(y_{(1)})_{(1)} \otimes p(y_{(1)})_{(2)} S_H(p(y_{(2)}))\right)
\end{eqnarray*}
and therefore, $p|_{_X}$ satisfies axiom (\ref{def:PC2}) if and only if $X\subset (h^*\otimes k^*_{(1)} \otimes S_{H^\circ}(k^*_{(2)}) - h^* k^*_{(1)} \otimes S_{H^\circ}(k^*_{(2)}))^\perp$ for all $h^*, k^*\in H^\circ$.

In the same way, we see that $p|_{_X}$ satisfies axiom (\ref{def:PC3}) if and only if $X$ lies in the annihilator of the generators of $I$ of type (3) from Definition \ref{hpar}. Therefore, we can conclude that the biggest subcoalgebra $C(H)$ for which the restriction of $p$ is a corepresentation equals the biggest subcoalgebra of $C(H)$ which annihilates $I$. In other words, $H^{par}=I^\perp$.

By Theorem \ref{reductionofpairing}, we conclude that there is the reduced pairing 
\[
\begin{array}{rccl} \left( \; , \; \right) : & H^{par} \otimes H^\circ _{par} & \rightarrow & k \\ & x\otimes a+I & \mapsto & \langle x, a \rangle \end{array}
\]
which is well-defined and nondegenerate. One easily checks that this is exactly the pairing given in \eqref{defpairingHpar}.
\end{proof}

\begin{exmp} \label{examplekc2} Take $H=kC_2$ the group algebra of the cyclic group with two elements. Since this algebra is finite dimensional, $H^\circ=H^*$. Denote by $\{u_e,u_g\}$ the base of grouplike elements of $H$, and by $\{p_e,p_g\}$ the dual base of $H^*$.
In \cite[example 4.5]{ABV} it was shown that $(kC_2)^*_{par} \cong k[x]/\langle x(x-1)(2x-1) \rangle$,
where $x=[p_e]$ and $1-x=[p_g]$. As $(kC_2)^*_{par}$ is again finite dimensional (in fact, it is three dimensional), we can conclude that 
$(kC_2)^{par} \cong \left( (kC_2)^*_{par}  \right)^*$ is again three dimensional. 

Take the canonical base $\{1,x,x^2\}$ for $(kC_2)^*_{par}$, and denote by $\{u,v,w\}$ the corresponding dual base of $(kC_2)^{par}$. Then we find that the counit and comultiplication of $(kC_2)^{par}$ are given on this base by
\begin{eqnarray*}
\epsilon(u)=1 && \Delta(u)= u\ot u\\
\epsilon(v)=0 & & \Delta(v)= u\ot v + v\ot u -1/2(v\ot w+w\ot v) -3/4w\ot w\\
\epsilon(w)=0& & \Delta(w)=u\ot w+w\ot u+v\ot v+3/2(v\ot w+w\ot v)+7/4 w\ot w  
\end{eqnarray*}
The map $p:(kC_2)^{par}\to kC_2$ is then given by 
$$p(u)=u_g, p(v)=u_e-u_g, p(w)=0$$
Take now $y=u+v+w$ and $z=u+1/2v+1/4w$, then we can check that $y$ and $z$ are a grouplike elements and moreover $p(y)=u_e$ and $p(z)=1/2u_e+1/2u_g$. It follows that $\{u,y,z\}$ is a base of grouplike elements for the universal coalgebra $kC_2^{par}$. 

Recall that (global) comodules over a coalgebra with a base of grouplike elements coincide with graded vector spaces. This is why comodules over the coalgebra $kC_2$ can be understood as $\ZZ_2$-graded vector spaces. In the same way, it follows from the above reasoning that a (finite dimensional) partial comodule $(M,\rho)$ over the Hopf algebra $kC_2$ has not 2 but 3 components. Beside the homogeneous components of degree $0$ and $1$ as in the global case, there is also a `mixed' component of degree $1/2$, which corresponds to the subspace of elements $m\in M$ such that $\rho(m)=m\ot z$. The same observation was also made in \cite{ABV} using different arguments.
\end{exmp}

\subsection{The universal copar coalgebra}

As before, let $H$ be a Hopf algebra and $(C(H),p)$ the cofree coalgebra over $H$.
We define the following $k$-linear maps:
\begin{eqnarray*}
E:  H^{par}  \rightarrow  H &&  x  \mapsto  p(x_{(1)})S(p(x_{(2)}))\\
\widetilde{E}:  H^{par}  \rightarrow H && x  \mapsto  S(p(x_{(1)}))p(x_{(2)}) 
\end{eqnarray*}
From the universal property of the cofree coalgebra, we then obtain unique 
coalgebra morphisms $\overline{E}: H^{par}\rightarrow C(H)$  and $\overline{\widetilde{E}}:H^{par}\rightarrow C(H)$, such that $E=p\circ \overline{E}$ and $\widetilde{E}=p\circ \overline{\widetilde{E}}$. Denote by $C=\mbox{Im}(\overline{E})$ and $\widetilde{C}=\mbox{Im}(\overline{\widetilde{E}})$. 

\begin{rmk}From now on $C$ and $\widetilde{C}$ will denote only these two particular subcoalgebras of $C(H)$; the coalgebra $C$ will be called the base coalgebra of $H^{par}$. 
\end{rmk}

\begin{lemma} \label{lemadosE} Let $H$ be a Hopf algebra with invertible antipode, then the following identities hold for all $x\in H^{par}$:
\begin{enumerate}[(i)]
\item $\epsilon_H \circ E(x) =\epsilon_H \circ  \widetilde{E}(x)=\epsilon_{H^{par}}(x)$.
\item $E(x_{(1)})\otimes p(x_{(2)}) =E(x_{(2)})_{(2)}\otimes S^{-1}( E(x_{(2)})_{(1)}) p(x_{(1)})$.
\item $p(x_{(1)})\otimes E(x_{(2)})=E(x_{(1)})_{(1)} p(x_{(2)}) \otimes E(x_{(1)})_{(2)}$.
\item $E(x_{(1)})E(x_{(2)})=E(x)$.
\item $\widetilde{E}(x_{(1)})\otimes p(x_{(2)}) =\widetilde{E}(x_{(2)})_{(1)} \otimes p(x_{(1)})\widetilde{E}(x_{(2)})_{(2)} $.
\item $p(x_{(1)})\otimes \widetilde{E}(x_{(2)})=p(x_{(2)}) S^{-1}(\widetilde{E}(x_{(1)})_{(2)})\otimes \widetilde{E}(x_{(1)})_{(1)}$.
\item $\widetilde{E}(x_{(1)})\widetilde{E}(x_{(2)})=\widetilde{E}(x)$.
\item $\widetilde{E}(x_{(1)})\otimes E(x_{(2)})=\widetilde{E}(x_{(2)})\otimes E(x_{(1)})$.
\item $E(x_{(1)})_{(1)}E(x_{(2)}) \otimes E(x_{(1)})_{(2)} =E(x_{(1)})E(x_{(2)})_{(1)} \otimes E(x_{(2)})_{(2)}$
\item $\widetilde{E}(x_{(1)})\otimes \ol E(x_{(2)})=\widetilde{E}(x_{(2)})\otimes \ol E(x_{(1)})$.
\item $\ol{\widetilde{E}}(x_{(1)})\otimes E(x_{(2)})=\ol{\widetilde{E}}(x_{(2)})\otimes E(x_{(1)})$.
\item $\ol{\widetilde{E}}(x_{(1)})\otimes \ol E(x_{(2)})=\ol{\widetilde{E}}(x_{(2)})\otimes \ol E(x_{(1)})$.
\item $S(p(x_{(1)})) E(x_{(2)})_{(1)} \otimes E(x_{(2)})_{(2)} = S(p(x_{(2)})) \otimes E(x_{(1)})$.
\item
$\begin{array}[t]{rcl}
	S(p(x_{(1)})) E(x_{(2)})_{(1)} E(x_{(3)})_{(1)} \dots E(x_{(n)})_{(1)}
	&\otimes& E(x_{(2)})_{(2)} \otimes E(x_{(3)})_{(2)} \otimes \dots \otimes E(x_{(n)})_{(2)} =\\
	&=&S(p(x_{(n)})) \otimes E(x_{(1)}) \otimes E(x_{(2)}) \otimes \dots \otimes E(x_{(n-1)}).
\end{array}$
\end{enumerate}
\end{lemma}

\begin{proof} 
\ul{(i)}. We only check the identity for $E$, the identity for $\widetilde{E}$ follows in a similar way.
\[
\epsilon_H (E(x))=\epsilon_H (p(x_{(1)})S(p(x_{(2)})))=\epsilon_H (p(x_{(1)}))\epsilon_H (p(x_{(2)})) \epsilon_H (p(x))=\epsilon_{H^{par}} (x) .
\]
\ul{(ii)}.
First note that since $p$ is a partial representation, we have
$p(x)=p(x_{(3)})S^{-1} (p(x_{(2)}))p(x_{(1)})$.
Then, 
\begin{eqnarray*}
&&\hspace{-2cm} E(x_{(1)})\otimes p(x_{(2)})  =  p(x_{(1)})S(p(x_{(2)})) \otimes p(x_{(3)}) =\\
& \stackrel{(\ref{def:PC3})}{=} & p(x_{(1)})S(p(x_{(2)})_{(1)}) \otimes p(x_{(2)})_{(2)} \\
& = & p(x_{(1)}) S( p(x_{(4)})_{(1)} S^{-1} (p(x_{(3)}))_{(1)} p(x_{(2)})_{(1)} ) \otimes p(x_{(4)})_{(2)} S^{-1} (p(x_{(3)}))_{(2)} p(x_{(2)})_{(2)}\\
& = & p(x_{(1)}) S( p(x_{(4)})_{(1)} S^{-1} (p(x_{(3)})_{(2)}) p(x_{(2)})_{(1)} ) \otimes p(x_{(4)})_{(2)} S^{-1} (p(x_{(3)})_{(1)}) p(x_{(2)})_{(2)}\\
& = & p(x_{(1)}) S(p(x_{(2)})_{(1)}) p(x_{(3)})_{(2)} S( p(x_{(4)})_{(1)}) \otimes S^{-1}(S(p(x_{(2)})_{(2)}) p(x_{(3)})_{(1)} S( p(x_{(4)})_{(2)}) )\\
& \stackrel{(\ref{le:PC4})}{=} & p(x_{(1)})_{(1)} S(p(x_{(1)})_{(2)}) p(x_{(2)})_{(2)} S( p(x_{(3)})_{(1)}) \otimes S^{-1}(S(p(x_{(1)})_{(3)}) p(x_{(2)})_{(1)} S( p(x_{(3)})_{(2)}) )\\
& = & p(x_{(2)})_{(2)} S( p(x_{(3)}))_{(2)} \otimes S^{-1}(S(p(x_{(1)})) p(x_{(2)})_{(1)} S( p(x_{(3)}))_{(1)} )\\
& = & p(x_{(2)})_{(2)} S( p(x_{(3)}))_{(2)} \otimes S^{-1} (p(x_{(2)})_{(1)} S( p(x_{(3)}))_{(1)} )p(x_{(1)})\\
& = & E(x_{(2)})_{(2)} \otimes S^{-1} (E(x_{(2)})_{(1)})p(x_{(1)}).
\end{eqnarray*}
\ul{(iii)}. Again we simply compute 
\begin{eqnarray*}
p(x_{(1)})\otimes E(x_{(2)})  &=&  p(x_{(1)})\otimes p(x_{(2)})S(p(x_{(3)})) \\
& \stackrel{(\ref{def:PC2})}{=} & p(x_{(1)})_{(1)} \otimes p(x_{(1)})_{(2)} S(p(x_{(2)})) \\
& \stackrel{(\ref{ISI})}{=} &  p(x_{(1)})_{(1)} S(p(x_{(2)}))_{(1)} p(x_{(3)})_{(1)}\otimes p(x_{(1)})_{(2)} S(p(x_{(2)}))_{(2)} p(x_{(3)})_{(2)} S(p(x_{(4)})) \\
& \stackrel{(\ref{le:PC5})}{=} &   p(x_{(1)})_{(1)} S(p(x_{(2)}))_{(1)} p(x_{(3)})_{(1)}\otimes p(x_{(1)})_{(2)} S(p(x_{(2)}))_{(2)} p(x_{(3)})_{(2)} S(p(x_{(3)})_{(3)})\\
& = & p(x_{(1)})_{(1)} S(p(x_{(2)}))_{(1)} p(x_{(3)}) \otimes p(x_{(1)})_{(2)} S(p(x_{(2)}))_{(2)}\\
& = & E(x_{(1)})_{(1)} p(x_{(2)}) \otimes E(x_{(1)})_{(2)} .
\end{eqnarray*}
\ul{(iv)}.
$E(x_{(1)})E(x_{(2)}) = p(x_{(1)})S(p(x_{(2)}))p(x_{(3)})S(p(x_{(4)})) \stackrel{(\ref{ISI})}{=} p(x_{(1)})S(p(x_{(2)})) =E(x)$.\\
Items \ul{(v)}, \ul{(vi)} and \ul{(vii)} are analogous, respectively, to items \ul{(ii)}, \ul{(iii)} and \ul{(iv)}.\\
\ul{(viii)}. Applying $S\otimes \mbox{Id}$ on identity (iii) we find
\[
S(p(x_{(1)}))\otimes E(x_{(2)})=S(E(x_{(1)})_{(1)} p(x_{(2)})) \otimes E(x_{(1)})_{(2)} .
\]
Using this, we obtain
\begin{eqnarray*}
\widetilde{E}(x_{(1)})\otimes E(x_{(2)}) & = & S(p(x_{(1)}))p(x_{(2)}) \otimes E(x_{(3)}) 
 \stackrel{(c)}{=}  S(p(x_{(1)})) E(x_{(2)})_{(1)} p(x_{(3)}) \otimes E(x_{(2)})_{(2)}\\
& = & S( E(x_{(1)})_{(1)} p(x_{(2)})) E(x_{(1)})_{(2)} p(x_{(3)}) \otimes E(x_{(1)})_{(3)}\\
& = & S(p(x_{(2)}))p(x_{(3)}) \otimes \epsilon (E(x_{(1)})_{(1)}) E(x_{(1)})_{(2)}
 =  \widetilde{E}(x_{(2)})\otimes E(x_{(1)}).
\end{eqnarray*}
\ul{(ix)} On one hand we have
\begin{eqnarray*}
 E(x_{(1)})_{(1)}E(x_{(2)}) \otimes E(x_{(1)})_{(2)} 
& = & p(x_{(1)})_{(1)} S(p(x_{(2)}))_{(1)}  p(x_{(3)})S(p(x_{(4)})) \otimes p(x_{(1)})_{(2)} S(p(x_{(2)}))_{(2)}\\
& \stackrel{(\ref{def:PC3})}{=} &   p(x_{(1)})_{(1)} S(p(x_{(2)})_{(1)})_{(1)}  p(x_{(2)})_{(2)}S(p(x_{(3)})) \otimes p(x_{(1)})_{(2)} S(p(x_{(2)})_{(1)})_{(2)}\\
& = &  p(x_{(1)})_{(1)} S(p(x_{(2)})_{(2)})  p(x_{(2)})_{(3)} S(p(x_{(3)})) \otimes p(x_{(1)})_{(2)} S(p(x_{(2)})_{(1)})\\
& = &  p(x_{(1)})_{(1)} S(p(x_{(3)})) \otimes p(x_{(1)})_{(2)} S(p(x_{(2)}))\\
& \stackrel{(\ref{def:PC2})}{=} &  p(x_{(1)}) S(p(x_{(4)})) \otimes p(x_{(2)}) S(p(x_{(3)})).
\end{eqnarray*}
On the other hand,
\begin{eqnarray*}
 E(x_{(1)})E(x_{(2)})_{(1)} \otimes E(x_{(2)})_{(2)} 
& = & p(x_{(1)})S(p(x_{(2)})) p(x_{(3)})_{(1)}S(p(x_{(4)}))_{(1)} \otimes  p(x_{(3)})_{(2)}S(p(x_{(4)}))_{(2)}\\
& \stackrel{(\ref{le:PC4})}{=} & p(x_{(1)})_{(1)} S(p(x_{(1)})_{(2)}) p(x_{(2)})_{(1)}S(p(x_{(3)}))_{(1)} \otimes  p(x_{(2)})_{(2)}S(p(x_{(3)}))_{(2)}\\
& = & p(x_{(1)})_{(1)}S(p(x_{(2)}))_{(1)} \otimes  p(x_{(1)})_{(2)}S(p(x_{(2)}))_{(2)}\\
& \stackrel{(\ref{def:PC2})}{=} &  p(x_{(1)})S(p(x_{(3)}))_{(1)} \otimes  p(x_{(2)})S(p(x_{(3)}))_{(2)}\\
& = & p(x_{(1)})S(p(x_{(3)})_{(2)}) \otimes  p(x_{(2)})S(p(x_{(3)})_{(1)})\\
& \stackrel{(\ref{def:PC3})}{=} & p(x_{(1)}) S(p(x_{(4)})) \otimes p(x_{(2)}) S(p(x_{(3)})).
\end{eqnarray*}
\ul{(x)}. Recall that $C$ is cogenerated by $H$, which means that the maps $p_n:\xymatrix{C\ar[r]^-{\Delta^n} & C^{\ot n} \ar[r]^-{p^{\ot n}} & H^{\ot n}}$ are jointly injective. Hence (x) holds if and only if the identity holds after applying the map $id\ot p_n$ for any $n\in \NN$. Using the fact that $E$ is a coalgebra map, we find indeed	
\begin{eqnarray*}
\widetilde{E}(x_{(1)})\otimes p_n(\ol E(x_{(2)}))&=& \widetilde{E}(x_{(1)})\otimes E(x_{(2)}) \ot \cdots \ot E(x_{(n+1)})
= \widetilde{E}(x_{(n+1)})\otimes E(x_{(1)}) \ot \cdots \ot E(x_{(n)})\\
&=& \widetilde{E}(x_{(2)})\otimes p_n(\ol E(x_{(1)}))
\end{eqnarray*}
where we used (viii) $n$ times in the second equality.\\
Items (xi) and (xii) are proven in the same way.\\
(\ul{xiii})
Let $x \in H^{\rm par}$, so
\begin{eqnarray*}
S(p(x_{(1)})) E(x_{(2)})_{(1)} \otimes E(x_{(2)})_{(1)}
&									=	& S(p(x_{(1)})) S(S^{-1}(E(x_{(2)})_{(1)})) \otimes E(x_{(2)})_{(1)}\\
&									=	& S(S^{-1}(E(x_{(2)})_{(1)})p(x_{(1)})) \otimes E(x_{(2)})_{(1)}\\
&\stackrel{(\ref{lemadosE}~(ii))}{	=}	& S(p(x_{(2)})) \otimes E(x_{(1)})
\end{eqnarray*}
Item (xiv) can be proven applying (xiii) $n$-times.
This ends the proof.
\end{proof}

\begin{defi} Let $H$ be a Hopf algebra over a field $k$ and $C$ be a coalgebra. A linear map $f:C\rightarrow H$ is said to be {\em copar} relative to $H$ (or {\em $H$-copar} for short) if it satisfies
\begin{enumerate}[label=(PQ\arabic*),ref=PQ\arabic*]
\item\label{def:PQ1} $\epsilon_H \circ f =\epsilon_C$.
\item\label{def:PQ2} $f(c)=f(c_{(1)})f(c_{(2)})$, for every $c\in C$.
\item\label{def:PQ3} $f(c_{(1)})f(c_{(2)})_{(1)} \otimes f(c_{(2)})_{(2)} =f(c_{(1)})_{(1)} f(c_{(2)}) \otimes f(c_{(1)})_{(2)}$, for every $c\in C$.
\end{enumerate}

A {\em universal copar coalgebra} relative to $H$ is a pair $(C^{par} (H) , \beta)$ in which $C^{par} (H)$ is a coalgebra and ${\beta :C^{par}(H) \rightarrow H}$ is a copar map such that for any coalgebra $D$ and any copar map $f:D\rightarrow H$ there exists a unique morphism of coalgebras $\bar{f}:D\rightarrow C^{par}$  such that $f=\beta \circ \bar{f}$.
\end{defi}

From Lemma \ref{lemadosE} it immediately follows that the map $E$ is $H$-copar.
As any universal object, the universal copar coalgebra, if it exists, is unique up to isomorphism of coalgebras. Our next aim is to show that this universal copar coalgebra indeed exists and is non-trivial.

\begin{prop} Let $\mathcal{C}_H$ be the sum of subcoalgebras $X\subseteq C(H)$ such that $p|_{_X}$ is an H-copar map. Then the pair $(\mathcal{C}_H , p)$ is a universal copar coalgebra relative to $H$.
\end{prop}

\begin{proof} Consider a coalgebra $D$ and a copar map $\theta :D\rightarrow H$.  There is a unique morphism of coalgebras $\overline{\theta}:D \rightarrow C(H)$ such that 
$p\circ \overline{\theta}=\theta$. It is easy to see that the map $p$ restricted to the subcoalgebra $\overline{\theta} (D)\subseteq C(H)$ satisfies the conditions (\ref{def:PQ1}), (\ref{def:PQ2}) and (\ref{def:PQ3}), hence it is a copar map. Therefore $\overline{\theta} (D)\subseteq \mathcal{C}_H$, and the co-restriction $\overline{\theta}:D\rightarrow \mathcal{C}_H$ is the required coalgebra morphism factorizing the copar map $\theta$. Therefore the pair $(\mathcal{C}_H , p)$ is isomorphic to the universal copar coalgebra associated to $H$. 
\end{proof}

\begin{prop}\label{pairingApar} The subcoalgebra $\mathcal{C}_H \subseteq C(H)$ is the annihilator space $J^{\perp}$ of the ideal $J\trianglelefteq TH^\circ$ which defines the relations of $A_{par}(H^\circ )$ (see \seref{partrep}). As a consequence, there is a nondegenerate dual pairing 
	\begin{equation}\label{defpairingApar}
	\begin{array}{rccl} \langle \! \langle \; , \; \rangle \! \rangle : & \mathcal{C}_H \otimes A_{par} (H^\circ ) & \rightarrow & k \\ & x\otimes a+J & \mapsto & \langle x, a \rangle \end{array} .
	\end{equation}
\end{prop}

\begin{proof} The proof follows the same steps of Proposition \ref{pairingHpar}.
\end{proof}

Let us write this pairing more explicitly. 
Consider an element $x\in \mathcal{C}_H$ and a monomial $\varepsilon_{k^{1*}} \ldots \varepsilon_{k^{n*}} \in A_{par}(H^\circ  )$ then
\begin{eqnarray}
\label{pairingCA}
\langle \! \langle x , \varepsilon_{k^{1*}} \ldots \varepsilon_{k^{n*}} \rangle \! \rangle & = & \langle x, k^{1*}\otimes 
\cdots k^{n*}\otimes \rangle \nonumber \\ 
& = & k^{1*}(p(x_{(1)})) \ldots k^{n*}(p(x_{(n)})) .
\end{eqnarray}
\begin{rmk}
	It is important to be careful here with the expression of this pairing, because the map $p$ restricted to $\mathcal{C}_H \subseteq C(H)$ is a copar map, while the same map $p$ restricted to $H^{par} \subseteq C(H)$ is a partial co-representation.
\end{rmk}
One can relate this pairing with the pairing $\left( \; , \; \right):H^{par} \otimes H^{\circ}_{par} \rightarrow k$ by using the coalgebra map $\overline{E} :H^{par} \rightarrow \mathcal{C}_H$. Indeed, the map $E:H^{par} \rightarrow H$ is $H$-copar and $E=p\circ \overline{E}$, then we immediately obtain that $C=\mbox{Im}(\overline{E})$ is a subcoalgebra of $\mathcal{C}_H$. The map $\overline{E}$ is adjoint to the source map $s: A_{par}(H^{\circ}) \rightarrow H^{\circ}_{par}$. More explicitly, this means that for each $x\in H^{par}$ and $a=\varepsilon_{h^{1*}} \ldots  \varepsilon_{h^{n*}} \in A_{par}(H^{\circ})$, we have
\begin{equation}
\label{sourcedual}
\langle \! \langle \overline{E}(x) , a \rangle \! \rangle =\left( x, s (a) \right) =\left( x , [h^{1*}_{(1)}][S(h^{1*}_{(2)})] \ldots [h^{n*}_{(1)}][S(h^{n*}_{(2)})] \right) .
\end{equation}

In \cite{ABV} it was proved that the subalgebra $A_{par}(H) \subseteq H_{par}$ was itself the universal algebra factorizing the conditions dual to (\ref{def:PQ1}), (\ref{def:PQ2}) and (\ref{def:PQ3}); in the coalgebra case, the situation is slightly different. Indeed, it is easy to see that the subcoalgebra $C=\overline{E}(H^{par})$ is dense in the universal copar coalgebra $\mathcal{C}_H$ relative to the finite topology but, in general, one cannot prove that these two coalgebras coincide. As the source map is injective, one can conclude by Proposition \ref{pairingandmorphisms} that the subspace
\[
C^{\perp}=\{ a\in A_{par} (H^{\circ})\; | \; \langle \! \langle \overline{E}(x) , a \rangle \! \rangle =0 , \; \; \forall x\in H^{par} \}.
\]
is the null space. 
This density of $C=\overline{E}(H^{par})$ in the universal coalgebra $\mathcal{C}_H$ is essential in what follows, since it guarantees that the elements of $C$ separate points in $A_{par} (H^{\circ})$ and therefore 
 the pairing \eqref{pairingCA} remains  nondegenerate when restricted to $C\otimes A_{par} (H^{\circ})$. This fact will be used in the proofs of all the results in the following sections.

One needs also to mention the analogous properties of the coalgebra $\widetilde{C} =\overline{\widetilde{E}}(H^{par}) \subseteq C(H)$. The linear map $p|_{\widetilde{C}}$ is what we will call an anti-copar map. Generally speaking, an anti-copar map from a coalgebra $C$ to the Hopf algebra $H$ is a linear map $f:C\rightarrow H$, satisfying
\begin{enumerate}[label=(APQ\arabic*),ref=APQ\arabic*]
\item\label{def:APQ1} $\epsilon_H \circ f =\epsilon_C$.
\item\label{def:APQ2} $f(c)=f(c_{(1)})f(c_{(2)})$, for every $c\in C$.
\item\label{def:APQ3} $f(c_{(2)})_{(1)} \otimes f(c_{(1)}) f(c_{(2)})_{(2)} =f(c_{(1)})_{(1)}\otimes f(c_{(1)})_{(2)} f(c_{(2)}) $, for every $c\in C$.
\end{enumerate}
In fact,the coalgebra $\widetilde{C}$ is a dense subcoalgebra of the universal coalgebra $\widetilde{\mathcal{C}}_H \subseteq C(H)$, which factorizes anti-copar maps by coalgebra morphisms. There is also a nondegenerate dual pairing between the coalgebra $\widetilde{C}$ and the subalgebra $\widetilde{A}(H^\circ )\subseteq H^\circ_{par}$, generated by elements $\widetilde{\varepsilon}_{h^*} =[S(h^*_{(1)})][h^*_{(2)}]$, with $h^* \in H^\circ$. 

\begin{proposition}\label{CisotildeC}
The coalgebras $C=\ol E(H^{par})$ and $\widetilde{C}=\ol{\widetilde E}(H^{par})$ presented above are anti-isomorphic.
\end{proposition}

\begin{proof} 
Consider the linear map $S^{-1}\circ p : C(H)\to H$. Since $S^{-1}$ is an anti-Hopf algebra morphism, the restriction of this linear map to $(H^{par})^{cop}$ is a corepresentation and hence there exists a unique anti-coalgebra map $\sigma:H^{par}\to H^{par}$ such that $p\circ \sigma=S^{-1}\circ p$. In the same way, the restriction of $S^{-1}\circ p$ to $\widetilde C^{cop}$ is a copar map, and hence there exists a unique anti-coalgebra map $\hat\sigma:\widetilde{C} \rightarrow \Cc_H$, such that $p\circ \widehat{\sigma}=S^{-1} \circ p$ and therefore
$$
p\circ \widehat\sigma \circ \ol{\widetilde E} = S^{-1}\circ p\circ \ol{\widetilde E} = S^{-1}\circ \widetilde E
$$
On the other hand, for any $x\in H^{par}$ we find that
\begin{eqnarray*}
S^{-1}\circ \widetilde E(x)&=& S^{-1}(p(x_{(2)})) p(x_{(1)}) = S^{-1}(p(x_{(2)})) S\circ S^{-1}(p(x_{(1)}))\\
&=& p(\sigma(x_{(2)})) S(p(\sigma(x_{(1)}))) = p(\sigma(x)_{(1)}) S(p(\sigma(x)_{(2)})) \\
&=& E\circ p\circ \sigma (x) = p\circ \ol E\circ \sigma(x).
\end{eqnarray*}
Hence we conclude that 
$p\circ \widehat\sigma \circ \ol{\widetilde E}=S^{-1}\circ \widetilde E=p\circ \ol E\circ \sigma$
are the same partial anti-copar map. Therefore $\widehat\sigma \circ \ol{\widetilde E}= \ol E\circ \sigma:H^{par}\to \Cc_H$ and the image of $\widehat\sigma$ lies in $C$.

In the same way, the anti-copar map $S\circ p :C^{cop}\rightarrow H$ induces a coalgebra morphism $\widetilde{\sigma}:C^{cop}\rightarrow \widetilde\Cc_H$ such that $p\circ \widetilde{\sigma} =S\circ p$, and whose image lies in $\widetilde C$. It is easy to verify that the anti-morphisms of coalgebras $\widehat{\sigma}: \widetilde{C}\rightarrow C$ and $\widetilde{\sigma} :C\rightarrow \widetilde{C}$ are mutually inverse, proving our assertion.
\end{proof}

\begin{thm} \label{coaction} Let $H$ be a Hopf algebra over a field $k$
and $C$ be the base 
coalgebra of $H^{par}$. Then the $k$-linear map
\[
\begin{array}{rccl} \lambda : & C & \rightarrow & H\otimes C \\ \, & \overline{E}(x) & \mapsto & p(x_{(1)})S(p(x_{(3)}))\otimes \overline{E}(x_{(2)}) \end{array}
\]
defines in $C$ a structure of a partial left $H$-comodule coalgebra.
\end{thm}

\begin{proof} 
First, one needs to prove that the map $\lambda$ is well defined, that is, for $x\in \mbox{Ker} (\overline{E})$, then 
\[
p(x_{(1)})S(p(x_{(3)}))\otimes \overline{E}(x_{(2)})=0 .
\]
Since we know that $C$ is cogenerated by $H$, it is enough to check that this identity holds after applying the maps $id\ot p_n$ for all $n$ (see \seref{cofree}). 
\begin{eqnarray*}
&& \hspace*{-2cm} p(x_{(1)})S(p(x_{(3)}))\ot p_n(\ove (x_{(2)})) =\\
& = & p(x_{(1)})S(p(x_{(3)})) \ot (E(x_{(2)})) \ot p_{n-1}(\ove (x_{(3)}))\\
& \stackrel{\ref{lemadosE}(iii)}{=} & E(x_{(1)})_{(1)} p(x_{(2)}) S(p(x_{(4)})) \ot (E(x_{(1)})_{(2)}) \ot p_{n-1}( \ove (x_{(3)}))\\
&\vdots\\
&=& E(x_{(1)})_{(1)} E(x_{(2)})_{(1)}\cdots E(x_{(n)})_{(1)}E(x_{(n+1)}) \ot E(x_{(1)})_{(2)}\ot E(x_{(2)})_{(2)}\ot \cdots\ot E(x_{(n)})_{(2)}
\end{eqnarray*}
which is zero since $p_n(\ol E(x))=0$ for all $n$.
Therefore $\lambda :C\rightarrow H\otimes C$ is a well defined linear map.

We need now to verify the axioms (\ref{def:LPCC1}), (\ref{def:LPCC2}) and (\ref{def:LPCC3}) from Definition \ref{comodule-coalgebra}.
For this, we will need Lemma \ref{lemadosE}(x), which can be written in the following form:
\begin{equation}\label{auxiliar}
S(p (x_{(1)})) p(x_{(2)}) \otimes \overline{E}(x_{(3)})  =S(p (x_{(2)})) p(x_{(3)}) \otimes   \overline{E}(x_{(1)}) .
\end{equation}

(\ref{def:LPCC1}) Consider an element $\overline{E}(x)\in C$, for $x\in H^{par}$, then
\begin{eqnarray*}
 {\overline{E}(x_{(1)})}^{[-1]} {\overline{E}(x_{(2)})}^{[-1]} \otimes {\overline{E}(x_{(1)})}^{[0]} \otimes {\overline{E}(x_{(2)})}^{[0]} 
& =&  p(x_{(1)})S(p(x_{(3)}))  p(x_{(4)})S(p(x_{(6)})) \otimes \overline{E}(x_{(2)}) \otimes \overline{E}(x_{(5)})\\
& \stackrel{(\ref{auxiliar})}{=} & p(x_{(1)})S(p(x_{(2)}))  p(x_{(3)})S(p(x_{(6)})) \otimes \overline{E}(x_{(4)}) \otimes \overline{E}(x_{(5)})\\
& \stackrel{(\ref{ISI})}{=} & p(x_{(1)})S(p(x_{(4)}))\otimes \overline{E}(x_{(2)}) \otimes \overline{E}(x_{(3)})\\
& = & (I\otimes \Delta_C )(p(x_{(1)})S(p(x_{(3)}))\otimes \overline{E}(x_{(2)}) )\\
& = & (I\otimes \Delta_C )\circ \lambda (\overline{E}(x)).
\end{eqnarray*}

(\ref{def:LPCC2}) Take $\overline{E}(x)\in C$, for $x\in H^{par}$, then
\[
(\epsilon_H \otimes I)\circ \lambda (\overline{E}(x))= \epsilon_H (p(x_{(1)})S(p(x_{(3)}))) \overline{E}(x_{(2)})= \epsilon_H (x_{(1)})\epsilon_H (x_{(3)}) \overline{E}(x_{(2)})=\overline{E}(x).
\]

In order to prove (\ref{def:LPCC3}), first note that $\nabla (\overline{E}(x))=E(x)$. Indeed,
\begin{eqnarray*}
\nabla (\overline{E}(x)) & = & \overline{E}(x)^{[-1]} \epsilon_C (\overline{E}(x)^{[0]})
 =  p(x_{(1)})S(p(x_{(3)})) \epsilon_C (\overline{E}(x_{(2)})) \\
& = & p(x_{(1)})S(p(x_{(3)})) \underline{\epsilon}(x_{(2)})
 =  p(x_{(1)})S(p(x_{(2)})) =E(x).
\end{eqnarray*}

(\ref{def:LPCC3}) For $\overline{E}(x)\in C$, with $x\in H^{par}$, we have
\begin{eqnarray*}
& \,  & \nabla (\overline{E}(x_{(1)})) {\overline{E}(x_{(2)})^{[-1]}}_{(1)} \otimes {\overline{E}(x_{(2)})^{[-1]}}_{(2)} \otimes \overline{E}(x_{(2)})^{[0]} \\
& = & E(x_{(1)})p(x_{(2)})_{(1)} S(p(x_{(4)}))_{(1)} \otimes p(x_{(2)})_{(2)} S(p(x_{(4)}))_{(2)} \otimes \overline{E}(x_{(3)})\\
& = & E(x_{(1)})p(x_{(2)}) S(p(x_{(5)}))_{(1)} \otimes p(x_{(3)}) S(p(x_{(5)}))_{(2)} \otimes \overline{E}(x_{(4)})\\
& = & p(x_{(1)}) S(p(x_{(4)}))_{(1)} \otimes p(x_{(2)}) S(p(x_{(4)}))_{(2)} \otimes \overline{E}(x_{(3)})\\
& = & p(x_{(1)}) S(p(x_{(5)})_{(2)}) \otimes p(x_{(2)}) \widetilde{E}(x_{(3)})S(p(x_{(5)})_{(1)}) \otimes \overline{E}(x_{(4)})\\
& \stackrel{\ref{lemadosE}(x)}{=} & p(x_{(1)}) S(p(x_{(5)})_{(2)}) \otimes p(x_{(2)}) \widetilde{E}(x_{(4)})S(p(x_{(5)})_{(1)}) \otimes \overline{E}(x_{(3)})\\
& = & p(x_{(1)}) S(p(x_{(6)})) \otimes p(x_{(2)}) \widetilde{E}(x_{(4)})S(p(x_{(5)})) \otimes \overline{E}(x_{(3)})\\
& \stackrel{\ref{lemadosE}(x)}{=} & p(x_{(1)}) S(p(x_{(6)})) \otimes p(x_{(2)}) \widetilde{E}(x_{(3)})S(p(x_{(5)})) \otimes \overline{E}(x_{(4)})\\
& = & p(x_{(1)}) S(p(x_{(5)})) \otimes p(x_{(2)}) S(p(x_{(4)})) \otimes \overline{E}(x_{(3)})\\
& = & p(x_{(1)})S(p(x_{(3)})) \otimes \lambda (\overline{E}(x_{(2)}))
 =  (I \otimes \lambda ) \circ \lambda (\overline{E}(x)),
\end{eqnarray*}
and
\begin{eqnarray*}
& \,  &  {\overline{E}(x_{(1)})^{[-1]}}_{(1)} \nabla (\overline{E}(x_{(2)})) \otimes {\overline{E}(x_{(1)})^{[-1]}}_{(2)} \otimes \overline{E}(x_{(1)})^{[0]} =\\
& = & p(x_{(1)})_{(1)} S(p(x_{(3)}))_{(1)} E(x_{(4)}) \otimes p(x_{(1)})_{(2)} S(p(x_{(3)}))_{(2)} \otimes \overline{E}(x_{(2)}) \\
& = & p(x_{(1)})_{(1)} S(p(x_{(3)})_{(2)}) E(x_{(4)}) \otimes p(x_{(1)})_{(2)} S(p(x_{(3)})_{(1)}) \otimes \overline{E}(x_{(2)}) \\
& = & p(x_{(1)})_{(1)} S(p(x_{(4)})) E(x_{(5)}) \otimes p(x_{(1)})_{(2)} S(p(x_{(3)})) \otimes \overline{E}(x_{(2)}) \\
& = & p(x_{(1)})_{(1)} S(p(x_{(4)})) \otimes p(x_{(1)})_{(2)} S(p(x_{(3)})) \otimes \overline{E}(x_{(2)}) \\
& = & p(x_{(1)})_{(1)} S(p(x_{(5)})) \otimes p(x_{(1)})_{(2)} \widetilde{E} (x_{(3)})S(p(x_{(4)})) \otimes \overline{E}(x_{(2)}) \\
& \stackrel{\ref{lemadosE}(x)}{=} & p(x_{(1)})_{(1)} S(p(x_{(5)})) \otimes p(x_{(1)})_{(2)} \widetilde{E} (x_{(2)})S(p(x_{(4)})) \otimes \overline{E}(x_{(3)}) \\
& = & p(x_{(1)}) S(p(x_{(6)})) \otimes p(x_{(2)}) \widetilde{E} (x_{(3)})S(p(x_{(5)})) \otimes \overline{E}(x_{(4)}) \\
& = & p(x_{(1)}) S(p(x_{(5)})) \otimes p(x_{(2)}) S(p(x_{(4)})) \otimes \overline{E}(x_{(3)}) \\
& = & p(x_{(1)})S(p(x_{(3)})) \otimes \lambda (\overline{E}(x_{(2)}))
 =  (I \otimes \lambda ) \circ \lambda (\overline{E}(x)).
\end{eqnarray*}
Therefore, $C$ is a partial left $H$-comodule coalgebra.
\end{proof}

\begin{rmk} 
We know from the theory of partial representations of Hopf algebras that the subalgebra ${A=A_{par}(H^\circ)}$ of the partial ``Hopf'' algebra $(H^\circ)_{par}$ generated by the elements $\epsilon_k=[k_{(1)}][S(k_{(2)})]$ for all $k\in H^\circ$, is a partial $H^\circ$-module algebra. This structure on $A$ is in duality with the structure on $C$ defined in the previous Theorem, in the sense of \cite[Theorem 6.8]{BV}, which means that the equality 
$$\bkk{c,h\cdot a}=\bk{c^{[-1]},h }\bkk{c^{[0]},a}$$
holds for all $c\in C$, $h\in H^\circ$ and $a\in A$, where $\bk{-,-}:H\ot H^\circ\to k, \bk{x,k}=k(x)$ is the obvious non-degenerate pairing between $H$ and $H^\circ$, induced by evaluation.
Let us write $c=\ol E(x)$ for some $x\in H^{par}$ and $a=\epsilon_{k^1}\cdots \epsilon_{k^n}$ for $k^1,\ldots,k^n\in H^\circ$. Then we find indeed
\begin{eqnarray*}
&&\hspace*{-1cm}\bk{p(x_{(1)})S(p(x_{(3)})),h}\bkk{\overline{E}(x_{(2)}, \varepsilon_{k^1}\cdots \varepsilon_{k^n}}=\\
&=& h_{(1)}(p(x_{(1)})) h_{(2)}(S(p(x_{(n+2)}))) k^1(\ol E(x_{(2)})) \cdots k^n(\ol E(x_{(n+1)}))\\
&=& h_{(1)}(p(x_{(1)})) k^1_{(1)}(p(x_{(2)}))S(k^1_{(2)})(p(x_{(3)})) \cdots k^n_{(1)}(p(x_{2n}))S(k^n_{(2)})(p(x_{2n+1}))S(h_{(2)})(p(x_{(2n+2)}))\\
&=& \left(x,[h_{(1)}][k^1_{(1)}][S(k^1_{(2)})][k^2_{(1)}][S(k^2_{(2)})]\cdots [k^n_{(1)}][S(k^n_{(2)})][S(h_{(2)})] \right)\\
&=& \left(x,[h_{(1)}k^1_{(1)}][S(k^1_{(2)})][k^2_{(1)}][S(k^2_{(2)})]\cdots [k^n_{(1)}][S(k^n_{(2)})][S(h_{(2)})] \right)\\
&=& \left(x,[h_{(1)}k^1_{(1)}][S(h_{(2)}k^1_{(2)})][h_{(3)}k^1_{(3)}][S(k^1_{(4)})][k^2_{(1)}][S(k^2_{(2)})]\cdots [k^n_{(1)}][S(k^n_{(2)})][S(h_{(2)})] \right)\\
&=& \left(x,[h_{(1)}k^1_{(1)}][S(h_{(2)}k^1_{(2)})][h_{(3)}k^1_{(3)}S(k^1_{(4)})][k^2_{(1)}][S(k^2_{(2)})]\cdots [k^n_{(1)}][S(k^n_{(2)})][S(h_{(2)})] \right)\\
&=& \left(x,[h_{(1)}k^1_{(1)}][S(h_{(2)}k^1_{(2)})][h_{(3)}][k^2_{(1)}][S(k^2_{(2)})]\cdots [k^n_{(1)}][S(k^n_{(2)})][S(h_{(2)})] \right)\\
&\vdots&\\
&=& \left(x,[h_{(1)}k^1_{(1)}][S(h_{(2)}k^1_{(2)})][h_{(3)}k^2_{(1)}][S(h_{(4)}k^2_{(2)})]\cdots [h_{(1n-1)}k^n_{(1)}][S(h_{(2n)}k^n_{(2)})][h_{(2n+1)}][S(h_{(2n+2)})] \right)\\
&=& \bkk{E(x),\varepsilon_{h_{(1)k^1}}\varepsilon_{h_{(2)}k^2}\cdots\varepsilon_{h_{(n)}k^n}\varepsilon_{h_{(n+1)}}}\\
&=& \bkk{E(x),h\cdot \varepsilon_{k^1}\varepsilon_{k^2}\cdots\varepsilon_{k^n}}
\end{eqnarray*}
\end{rmk}

\subsection{The cosmash structure of the universal corep coalgebra}

Throughout this section, we suppose that the antipode of $H$ is bijective.

\begin{lemma} \label{phi0} Let $H$ be a Hopf algebra, $H^{par}$ be the universal coalgebra previously defined and $\overline{E}:H^{par} \rightarrow C$ the coalgebra map which factorizes $E:H^{par} \rightarrow H$. If $\overline{\omega}_0 :\underline{C\cosmash H} \rightarrow H^{par}$ is the coalgebra map which factorizes the standard partial corepresentation $\omega_0$, then $\overline{E} \circ \overline{\omega}_0 =\phi_0$ as coalgebra morphisms between $\underline{C \cosmash H}$ and $C$.
\end{lemma}

\begin{proof} Recall that $C=\overline{E}(H^{par})=\mathcal{C}_H \subseteq C(H)$ is the universal copar coalgebra associated to $H$. 

Consider now the following linear maps:
\begin{enumerate}
\item $p\circ \phi_0 :\underline{C\cosmash H}\rightarrow H$.
\item $p\circ \overline{E} \circ \overline{\omega}_0 :\underline{C\cosmash H}\rightarrow H$.
\end{enumerate}

Taking an element $\sum_i x_i \cosmash h_i \in  \underline{C\cosmash H}$, we have
\begin{eqnarray*}
 p\circ \overline{E} \circ \overline{\omega}_0 (\sum_i x_i \cosmash h_i )  &=&  E\circ \overline{\omega}_0 (\sum_i x_i \cosmash h_i )
 =  \sum_i  p(\overline{\omega}_0  ( {x_i}_{(1)} \cosmash {{x_i}_{(2)}}^{[-1]} {h_i}_{(1)})) S(p(\overline{\omega}_0 ({{x_i}_{(2)}}^{[0]} \cosmash {h_i}_{(2)}))) \\
&&\hspace{-4cm} =  \sum_i \omega_0  ( {x_i}_{(1)} \cosmash {{x_i}_{(2)}}^{[-1]} {h_i}_{(1)})S(\omega_0  ({{x_i}_{(2)}}^{[0]} \cosmash {h_i}_{(2)}))
 =  \sum_i \nabla ({x_i}_{(1)}) {{x_i}_{(2)}}^{[-1]} {h_i}_{(1)} S({h_i}_{(2)}) S(\nabla ({{x_i}_{(2)}}^{[0]})) \\
&&\hspace{-4cm} =  \sum_i {x_i}^{[-1]}S(\nabla ({x_i}^{[0]} \epsilon (h_i )
 =  \sum_i {x_i}^{[-1]} S({x_i}^{[0][-1]}) \epsilon ({x_i}^{[0][0]})\epsilon (h_i )\\
&&\hspace{-4cm} = \sum_i \nabla ({x_i}_{(1)}) {{{x_i}_{(2)}}^{[-1]}}_{(1)} S( {{{x_i}_{(2)}}^{[-1]}}_{(2)}) \epsilon ({{x_i}_{(2)}}^{[0]})\epsilon (h_i )
 =  \sum_i \nabla ({x_i}_{(1)}) \epsilon ({{x_i}_{(2)}}^{[-1]}) \epsilon ({{x_i}_{(2)}}^{[0]})\epsilon (h_i )\\
&&\hspace{-4cm} = \sum_i \nabla ({x_i}_{(1)}) \epsilon ({{x_i}_{(2)}}) \epsilon (h_i )
 =  \sum_i \nabla ({x_i}) \epsilon (h_i ),
\end{eqnarray*}
and remembering that $C=\mbox{Im}(\overline{E})$, then for every $x_i \in C$ there exists an element $y_i \in H^{par}$ such that $x_i =\overline{E}(y_i )$, this leads to 
\begin{eqnarray*}
p\circ \overline{E} \circ \overline{\omega}_0 (\sum_i x_i \cosmash h_i ) & = & \sum_i \nabla ({x_i}) \epsilon (h_i ) 
 =  \sum_i \nabla (\overline{E}({y_i})) \epsilon (h_i ) \\
& = & \sum_i E(y_i )\epsilon (h_i ) =p( \sum_i \overline{E} (y_i )\epsilon (h_i ))  \\
& = & p\circ \phi_0 (\sum_i \overline{E} (y_i)\cosmash h_i )
 =   p\circ \phi_0 (\sum_i x_i \cosmash h_i ).
\end{eqnarray*}
Therefore, $p\circ \phi_0 = p\circ \overline{E} \circ \overline{\omega}_0$ as linear maps from $\underline{C\cosmash H}$ to $H$. Let us call this map $\theta$, which evaluated in an element $\xi= \sum_i \overline{E} (y_i ) \cosmash h_i \in  \underline{C\cosmash H}$ gives $\theta (\xi )=\sum _i E(y_i )\epsilon (h_i )$. It is easy to see that $\theta$ is a copar map. Indeed for axiom (\ref{def:PQ1}), 
\[
\epsilon_H (\theta (\sum_i \overline{E} (y_i ) \cosmash h_i )) =\sum_i \epsilon_H (E(y_i ))\epsilon_H (h_i ) \sum_i \epsilon (y_i )\epsilon (h_i ) =\underline{\epsilon} (\sum_i \overline{E} (y_i ) \cosmash h_i ) .
\]
For axiom (\ref{def:PQ2}),
\begin{eqnarray*}
\theta (\xi_{(1)}) \theta (\xi_{(2)}) & = & \sum_i \theta (  \overline{E} ({y_i}_{(1)} ) \cosmash \overline{E} ({y_i}_{(2)} )^{[-1]} {h_i}_{(1)}) \theta ( \overline{E} ({y_i}_{(2)} )^{[0]} \cosmash  {h_i}_{(2)}) \\
& = & \sum_i \theta (  \overline{E} ({y_i}_{(1)} ) \cosmash p({y_i}_{(2)})S(p({y_i}_{(4)})) {h_i}_{(1)}) \theta (\overline{E} ({y_i}_{(3)} ) \cosmash  {h_i}_{(2)}) \\
& = & \sum_i E({y_i}_{(1)} ) \epsilon (p({y_i}_{(2)})S(p({y_i}_{(4)})) {h_i}_{(1)}) E({y_i}_{(3)} ) \epsilon ({h_i}_{(2)}) \\
& = & \sum_i E({y_i}_{(1)} ) E({y_i}_{(3)} ) \epsilon ({y_i}_{(1)} )\epsilon ({y_i}_{(4)} ) \epsilon ({h_i})
 =  \sum_i E({y_i}_{(1)} ) E({y_i}_{(2)} ) \epsilon ({h_i})  \\
& = & \sum_i E({y_i}) \epsilon ({h_i}) 
 =  \theta (\sum_i \overline{E} (y_i ) \cosmash h_i ).
\end{eqnarray*}
And for axiom (\ref{def:PQ3}),
\begin{eqnarray*}
\theta (\xi_{(1)})_{(1)} \theta (\xi_{(2)}) \otimes \theta (\xi_{(1)})_{(2)} & = &  \sum_i E({y_i}_{(1)} )_{(1)} E({y_i}_{(2)} ) \epsilon ({h_i}) \otimes  E({y_i}_{(1)} )_{(2)} \\
& = & \sum_i  E({y_i}_{(1)} ) E({y_i}_{(2)} )_{(1)} \epsilon ({h_i}) \otimes  E({y_i}_{(2)} )_{(2)} \\
& = & \theta (\xi_{(1)}) \theta (\xi_{(2)})_{(1)} \otimes \theta (\xi_{(2)})_{(2)}.
\end{eqnarray*}
As the images of $\phi_0$ and $\overline{E} \circ \overline{\omega}_0$ are contained in $C$, then both coalgebra maps can be viewed as coalgebra maps from $\underline{C\cosmash H}$ to $C$. Moreover, $\theta$ is the result of composing $p$ with both maps and then, by the universal property of $C$, we have that  $\phi_0 =\overline{E} \circ \overline{\omega}_0$.
\end{proof}

\begin{thm}\label{isoHparCosmash}
	Let $H$ be a Hopf algebra and $C$ be the base coalgebra of $H^{par}$. Then $H^{par} \cong \underline{C\cosmash H}$, in which the structure of a partial left $H$-comodule coalgebra over $C$ is given by the Theorem \ref{coaction}.
\end{thm}

\begin{proof} First note that the pair $(\overline{E}, p)$ is a contravariant pair relative to $H^{par}$: the definition of the coaction $\lambda :C\rightarrow H\otimes C$ in Theorem \ref{coaction} is exactly the condition (CP2) of Definition \ref{contravariant-pair} and the identity \ref{auxiliar} is exactly the condition (CP3). Then there exists, by the universal property of the partial cosmash coproduct, a unique coalgebra morphism
\[
\begin{array}{rccl} \Phi : & H^{par} & \rightarrow & \underline{C\cosmash H} \\ \, & x & \mapsto & \overline{E}(x_{(1)}) \cosmash p(x_{(2)}) \end{array}
\]
such that $\omega_0 \circ \Phi =p$ and $\phi_0 \circ \Phi =\overline{E}$.

On the other hand, the map 
\[
\begin{array}{rccl} \omega_0 : & \underline{C\cosmash H} & \rightarrow & H \\ \, & \overline{E}(x) \cosmash h & \mapsto & E(x)h \end{array}
\]
is a partial representation of $H$ relative to $\underline{C\cosmash H}$, then there exists a unique coalgebra morphism $\overline{\omega}_0 :\underline{C\cosmash H} \rightarrow  H^{par}$ such that $\omega_0 =p\circ \overline{\omega}_0$. 

Consider $x\in H^{par}$, note that 
\[
\omega_0 (\Phi (x))=\omega_0 (\overline{E}(x_{(1)}) \cosmash p(x_{(2)})) =E(x_{(1)}) p(x_{(2)})=p(x).
\]
Then, $p\circ \overline{\omega}_0 \circ \Phi =p$. On the other hand $p:H^{par} \rightarrow H$ is a partial corepresentation of $H$ relative to $H^{par}$, then there is a unique morphism $\overline{p}: H^{par} \rightarrow H^{par}$ such that $p=p\circ \overline{p}$. By the uniqueness and because $p=p\circ \mbox{Id}_{H^{par}}$ we have that $\overline{\omega}_0 \circ \Phi =\mbox{Id}_{H^{par}}$.

Consider now an element $y=\sum_i \overline{E}(x_i)\cosmash h_i \in \underline{C\cosmash H}$, and denote $x=\overline{\omega}_0 (y)$, then 
\[
x_{(1)} \otimes x_{(2)} =\sum_i \overline{\omega}_0 (\overline{E}({x_i}_{(1)}) \cosmash \overline{E}({x_i}_{(2)})^{[-1]} {h_i}_{(1)} ) \otimes  \overline{\omega}_0 (\overline{E}({x_i}_{(2)})^{[0]} \cosmash{h_i}_{(2)}) .
\]
Therefore, 
\begin{eqnarray*}
\Phi \circ \overline{\omega}_0 (y) & = & \sum_i \overline{E}( \overline{\omega}_0 (\overline{E}({x_i}_{(1)}) \cosmash \overline{E}({x_i}_{(2)})^{[-1]} {h_i}_{(1)} )) \cosmash  p(\overline{\omega}_0 (\overline{E}({x_i}_{(2)})^{[0]} \cosmash{h_i}_{(2)})) \\
& = & \sum_i \phi_0 (\overline{E}({x_i}_{(1)}) \cosmash p({x_i}_{(2)}) p({x_i}_{(4)}){h_i}_{(1)} ) \cosmash \omega_0 (\overline{E}({x_i}_{(3)}) \cosmash{h_i}_{(2)})\\
& = & \sum_i \overline{E}({x_i}_{(1)}) \epsilon_H (p({x_i}_{(2)}) S(p({x_i}_{(4)})){h_i}_{(1)}) \cosmash E({x_i}_{(3)}) {h_i}_{(2)} \\
& = & \sum_i \overline{E}({x_i}_{(1)}) \epsilon_H ({x_i}_{(2)}) \epsilon_H ({x_i}_{(4)}) \epsilon_H ({h_i}_{(1)}) \cosmash E({x_i}_{(3)}) {h_i}_{(2)} \\
& = & \sum_i \overline{E}({x_i}_{(1)}) \cosmash E({x_i}_{(2)}) {h_i}
 =  \sum_i \overline{E}({x_i}_{(1)}) \cosmash \nabla (\overline{E}({x_i}_{(2)})) {h_i}\\
& = & \sum_i \overline{E}(x_i)\cosmash h_i =y,
\end{eqnarray*}
in which the second equality comes from Lemma~\ref{phi0}. This concludes the proof.
\end{proof}

Again, it is important to mention how these constructions and these results fit within the context in which one takes the base coalgebra $\widetilde{C}$ instead of $C$. First there is a partial right coaction $\rho :\widetilde{C}\rightarrow \widetilde{C}\otimes H$, given by
\[
\rho (\overline{\widetilde{E}}(x)) = \overline{\widetilde{E}}(x_{(2)}) \otimes S(p(x_{(1)}))p(x_{(3)} ),
\]
making the coalgebra $\widetilde{C}$ into a partial right $H$-comodule coalgebra. This allows one to construct the right partial cosmash coproduct $\underline{H\rcosmash \widetilde{C}}$, and then we can prove that this partial cosmash coproduct is isomorphic to the coalgebra $H^{par}$ via the coalgebra isomorphisms $\widetilde{\Phi} :H^{par}\rightarrow \underline{H\rcosmash \widetilde{C}}$ , given by $\widetilde{\Phi}(x)=p(x_{(1)})\rcosmash \overline{\widetilde{E}}(x_{(2)})$, and its inverse, $\overline{\widetilde{\omega}}_0 : \underline{H\rcosmash \widetilde{C}}\rightarrow H^{par}$, given by the universal property of $H^{par}$, such that  $p\circ \overline{\widetilde{\omega}}_0 (h\rcosmash \overline{\widetilde{E}}(x))=h\widetilde{E}(x)$. 

\section{The universal coalgebra $H^{par}$ as a Hopf-coalgebroid}\selabel{universals}

In \cite{ABV} (see Section \ref{partrep}), it was shown that the universal algebra $H_{par}$ 
has the structure of a Hopf algebroid. 
In this section, we will show that $H^{par}$ has the structure of a Hopf-coalgebroid that is dual to $(H^\circ)_{par}$.

\subsection{Hopf-coalgebroids}

The notion of a (left) bicoalgebroid was introduced in \cite{BM} and studied further in \cite{Bal}. We review this notion here and also introduce Hopf-coalgebroid as a dual notion of a Hopf algebroid \cite{Bohm}. Throughout this section, we suppose that $k$ is a field. However, as explained in \cite{Bal}, the definitions also make sense supposing that all considered $k$-modules are (locally) projective.

Let $C$ and $\Hh$ be two $k$-coalgebras, $\alpha:\Hh\to C$ a coalgebra morphism and $\beta:\Hh\to C$ an anti-coalgebra morphism.
Then $\Hh$ can be endowed with a $C$-bicomodule structure defining left and right coactions by
\begin{equation} \label{coactionsH}
\begin{array}{rlrl} \lambda : & \mathcal{H} & \rightarrow & C\otimes \mathcal{H}\\
\, & x & \mapsto & \alpha (x_{(1)}) \otimes x_{(2)} ,
\end{array}\qquad
\begin{array}{rlrl} \rho : & \mathcal{H} & \rightarrow &  \mathcal{H}\otimes C \\
\, & x & \mapsto &   x_{(2)} \otimes \beta (x_{(1)}).
\end{array}
\end{equation}
Now consider the cotensor product 
\[
\mathcal{H} \square^C \mathcal{H} = \mbox{Ker} ( \rho \otimes \mbox{Id}_{\mathcal{H}} -\mbox{Id}_{\mathcal{H}}\otimes \lambda). \]
In general, $\Hh\square^C\Hh$ is not a coalgebra, therefore in \cite{Bal} the `cocenter' of a bicomodule was introduced, as a dual version of the Takeuchi product. In this note, we will use the term (left) {\em B\'alint-product} for this object, which is defined as the coequalizer
	\[
	\xymatrix{\mathcal{H} \square^C \mathcal{H} \otimes C^* \ar@<.5ex>[rr]^-{\Phi_1} \ar@<-.5ex>[rr]_-{\Phi_2} && \mathcal{H} \square^C \mathcal{H} \ar[rr]^-{\pi_L} && \mathcal{H} \boxtimes^C_l \mathcal{H}} , 
	\]
	in which $\Phi_1 (\sum_i x^i \otimes y^i \otimes \varphi) =\sum_i x^i_{(1)} \otimes y^i \varphi (\beta (x^i_{(2)}))$ and $\Phi_2 (\sum_i x^i \otimes y^i \otimes \varphi) =\sum_i x^i \otimes y^i_{(1)} \varphi (\alpha (y^i_{(2)}))$. 
 Remark that using the dual base property, for any class $\sum_i x^i \boxtimes y^i\in \Hh\boxtimes^C_l \Hh$ we have an identity
	$$\sum_i x^i \boxtimes y^i_{(1)}\ot  \alpha(y^i_{(2)})=\sum_i x^i_{(1)} \boxtimes y^i\ot \beta(x^i_{(2)})$$
It can be proved that the B\'alint product is indeed a coalgebra with comultiplication $\Delta(\sum_i x^i \boxtimes y^i)=\sum_i x^i_{(1)} \boxtimes y^i_{(1)}\ot x^i_{(2)} \boxtimes y^i_{(2)}$ and counit $\epsilon(\sum_i x^i \boxtimes y^i)=\sum_i \epsilon_\Hh(x^i)\epsilon_\Hh(y^i)$.
\begin{defi} \label{definitioncoalgebroid} 
Let $C$ be a $k$-coalgebra. A left bicoalgebroid over $C$ is a sextuple $(\mathcal{H}, C, \alpha, \beta, \ol\mu_L, \eta_L )$ where
	\begin{enumerate}[label=(BC\arabic*),ref=BC\arabic*]
		\item\label{def:BC1} $\mathcal{H}$ is a coalgebra, $\alpha :\mathcal{H}\rightarrow C$ is a coalgebra morphism, $\beta :\mathcal{H} \rightarrow C$ is an anti-coalgebra morphism, such that 
\[
\alpha (x_{(1)})\otimes \beta (x_{(2)}) =\alpha (x_{(2)}) \otimes \beta (x_{(1)}),
\]
endowing $\Hh$ with a $C$-bicomodule structure as in \eqref{coactionsH}.
\item\label{def:BC2} The map $\ol\mu_L:\Hh\boxtimes^C_l\Hh\to \Hh$ is a coalgebra morphism and the restriction $\mu_L=\ol\mu_L\circ\pi_L:\Hh\square^C\Hh\to \Hh$ is associative, i.e. $\mu_L \circ (\mu_L \square^C \mbox{Id}_{\mathcal{H}}) =\mu_L \circ (\mbox{Id}_{\mathcal{H}} \square^C \mu_L)$.
\item
\label{def:BC3} The map $\eta_L :C \rightarrow \mathcal{H}$ is a $C$-bicomodule map and it is the unit map relative to the multiplication $\mu_L$ (resp. $\mu_R$), that is
		\begin{equation}\label{unitHcoid}
		\mu_L \circ (\eta_L \square^C \mbox{Id}_{\mathcal{H}})\circ \lambda =\mbox{Id}_{\mathcal{H}} =\mu_L \circ  (\mbox{Id}_{\mathcal{H}} \square^C \eta_L )\circ \rho .
		\end{equation}
		Moreover, the unit map also satisfies $\epsilon_{\mathcal{H}} \circ \eta_L =\epsilon_C$ and
		\[
		\Delta (\eta_L (c))=\eta_L (c)_{(1)} \otimes \eta_L (\alpha (\eta_L (c)_{(2)}))=\eta_L (c)_{(1)} \otimes \eta_L (\beta (\eta_L (c)_{(2)})) 
		\]
		for every $c\in C$.
\end{enumerate}
\end{defi}

In a similar way, a right bicoalgebroid over a coalgebra $\widetilde{C}$ is a sextuple $(\mathcal{H}, \widetilde{C}, \widetilde{\alpha}, \widetilde{\beta}, \mu_R , \eta_R )$. Here $\widetilde{\alpha} :\mathcal{H}\rightarrow \widetilde{C}$ is a coalgebra morphism and $\widetilde{\beta} :\mathcal{H} \rightarrow \widetilde{C}$ is an anti-coalgebra morphism such that $\widetilde{\alpha} (x_{(1)})\otimes \widetilde{\beta} (x_{(2)}) =\widetilde{\alpha} (x_{(2)}) \otimes \widetilde{\beta} (x_{(1)})$. One considers $\Hh$ as a $\widetilde{C}$-bicomodule structure with left and right $\widetilde{C}$-coactions given by
$\widetilde\lambda(x)=\widetilde{\beta} (x_{(2)}) \otimes x_{(1)}$ and $\widetilde\rho(x)=x_{(1)} \otimes \widetilde{\alpha} (x_{(2)})$. 
The right Balint coproduct is defined by the coequalizer 
	\[
		\xymatrix{\mathcal{H} \square^{\widetilde{C}} \mathcal{H} \otimes \widetilde{C}^* \ar@<.5ex>[rr]^-{\Psi_1} \ar@<-.5ex>[rr]_-{\Psi_2} && \mathcal{H} \square^{\widetilde{C}} \mathcal{H} \ar[rr]^-{\pi_R} && \mathcal{H} \boxtimes^{\widetilde{C}}_r \mathcal{H}
	},
	\]
		for $\Psi_1 (\sum_i x^i \otimes y^i \otimes \varphi) =\sum_i x^i_{(2)} \otimes y^i \varphi (\widetilde{\alpha} (x^i_{(1)}))$ and $\Psi_2 (\sum_i x^i \otimes y^i \otimes \varphi) =\sum_i x^i \otimes y^i_{(2)} \varphi (\beta (y^i_{(1)}))$.
		The map $\ol\mu_R:\Hh\boxtimes^{\widetilde C}_r\Hh\to \Hh$ is a coalgebra morphism and $\Hh$ is an associative algebra in the monoidal category of $\widetilde C$-bicomodules with multiplication $\mu_R=\ol\mu_R\circ\pi_R=\Hh\square^{\widetilde C}\Hh\to \Hh$ and unit $\eta_R:C\to \Hh$. Moreover, we have identities $\epsilon_{\mathcal{H}} \circ \eta_R =\epsilon_{\widetilde{C}}$) and
		$
		\Delta (\eta_R (c))=  \eta_R (\widetilde{\alpha} (\eta_R (c)_{(1)})) \otimes \eta_R (c)_{(2)} = \eta_R (\widetilde{\beta} (\eta_R (c)_{(1)})) \otimes \eta_R (c)_{(2)}$,
		for every $c\in \widetilde{C}$.

\begin{definition}\label{defHcoid}
	Let $\mathcal{H}$ be a left bicoalgebroid over $C$ and a right bicoalgebroid over $\widetilde{C}$, where $\widetilde{C}\cong C^{cop}$ are anti-isomorphic coalgebras. We say that $\Hh$ is a Hopf-coalgebroid if there exists an anti-coalgebra morphism $\mathcal{S}: \mathcal{H}\rightarrow \mathcal{H}$ and the following compatibility conditions hold: 
	\begin{enumerate}[(a)]
		\item Compatibility between left and right comodule structures, 
		\begin{eqnarray*}
		\widetilde{\beta}\circ \eta_L \circ \alpha =\widetilde{\beta} , \qquad & \, & \qquad \widetilde{\alpha}\circ \eta_L \circ \beta =\widetilde{\alpha}, \\
		{\beta}\circ \eta_R \circ \widetilde{\alpha} ={\beta} , \qquad & \, & \qquad {\alpha}\circ \eta_R \circ \widetilde{\beta} ={\alpha} .
		\end{eqnarray*}
		\item Mixed associativity between left and right multiplication, 
		\[
		\mu_R \circ ( \mu_L \square^{\widetilde{C}} \mbox{Id}_{\mathcal{H}} )=\mu_L \circ ( \mbox{Id}_{\mathcal{H}} \square^C \mu_R ), \quad  \mbox{ and }  \quad \mu_R \circ ( \mbox{Id}_{\mathcal{H}} \square^{\widetilde{C}} \mu_L  )=\mu_L \circ (  \mu_R  \square^C 
		\mbox{Id}_{\mathcal{H}} ).
		\]
		\item Compatibility between the antipode and the structural maps, that is, for each $x\in \mathcal{H}$, 
		\[
		\beta (\mathcal{S} (x)_{(1)}) \otimes \mathcal{S} (x)_{(2)} \otimes \widetilde{\beta} (\mathcal{S} (x)_{(3)}) =\alpha (x_{(3)})\otimes \mathcal{S}(x_{(2)}) \otimes \widetilde{\alpha}(x_{(1)}) .
		\]
		\item Antipode axiom, that is 
		\[
		\mu_L \circ (\mathcal{S} \otimes \mbox{Id}_{\mathcal{H}}) \circ \Delta = \eta_R \circ \widetilde{\alpha} , \quad \mbox{ and } \quad \mu_R \circ (\mbox{Id}_{\mathcal{H}} \otimes \mathcal{S}) \circ \Delta = \eta_L \circ {\alpha}.
		\]
	\end{enumerate}
\end{definition}

\begin{rmk}
\begin{enumerate}
\item The definitions of a bicoalgebroid and of a Hopf coalgebroid can be formulated entirely in terms of the multiplication $\mu_L:\Hh\square^C\Hh\to \Hh$ rather than the map $\ol\mu_L$. In this case, one needs to impose the condition
\[
		\sum_i \mu_L (x_{i} \otimes y_{i(1)}) \otimes \alpha (y_{i(2)}) =\sum_i \mu_L (x_{i(1)} \otimes y_{i}) \otimes \beta (x_{i(2)})
		\]
in order to assure that $\mu_L$ factors to the B\'alint product.
\item 
Remark that in the unit condition of the left coalgebroid \eqref{unitHcoid} it is implicitly assumed that 
the images of $(\eta_L \square^C \mbox{Id}_{\mathcal{H}}) \circ \lambda$ and $(\mbox{Id}_{\mathcal{H}} \square^{C} \eta_L )\circ \rho$ 
 are in $\mathcal{H} \square^{C} \mathcal{H}$. 
But this is automatically satisfied, because $\lambda (\mathcal{H} ) \subseteq C\square^C \mathcal{H}$ and $\rho (\mathcal{H}) \subseteq \mathcal{H} \square^{C} C$ 
and $\eta_L$ 
is a morphism of $C$-bicomodules. 
\item In a similar way, note that in the antipode axiom (item (d) in Definition \ref{defHcoid}) it is implicitly expected that for any $x\in \mathcal{H}$ we have
		\[
		\mathcal{S} (x_{(1)}) \otimes x_{(2)} \in \mathcal{H} \square^{C} \mathcal{H} ,
		\]
		and
		\[
		x_{(1)} \otimes \mathcal{S} (x_{(2)})\in \mathcal{H} \square^{\widetilde{C}} \mathcal{H} .
		\]
		One can easily prove these facts. For the first one, we have
		\begin{eqnarray*}
		\mathcal{S} (x_{(1)})_{(2)} \otimes \beta ( \mathcal{S} (x_{(1)})_{(1)} ) \otimes x_{(2)} & = & (\epsilon_{\widetilde{C}} \otimes \mbox{Id} \otimes \mbox{Id}) \left( \widetilde{\beta} ( \mathcal{S} (x_{(1)})_{(3)} ) \otimes \mathcal{S} (x_{(1)})_{(2)} \otimes \beta ( \mathcal{S} (x_{(1)})_{(1)} )\right)  \otimes x_{(2)} \\
		& \stackrel{(c)}{=} & (\epsilon_{\widetilde{C}} \otimes \mbox{Id} \otimes \mbox{Id}) \left( \widetilde{\alpha} (x_{(1)}) \otimes \mathcal{S} (x_{(2)}) \otimes \alpha (x_{(3)}) \right) \otimes x_{(4)}  \\
		& = & \mathcal{S} (x_{(1)}) \otimes \alpha (x_{(2)}) \otimes x_{(3)} .
		\end{eqnarray*}
		Therefore, $\mathcal{S} (x_{(1)}) \otimes x_{(2)} \in \mathcal{H} \square^{C} \mathcal{H}$. The second result is proved in the same way.
\end{enumerate}
\end{rmk}

	\begin{exmp} Any Hopf algebra $H$ over $k$ can be viewed as a Hopf-coalgebroid by taking $C=\widetilde C=k$ the trivial $k$-coalgebra, $\alpha =\beta =\widetilde{\alpha} =\widetilde{\beta} =\epsilon$, $\mu_L=\mu_R$ the multiplication of $H$, and $\eta_L=\eta_R$ the unit map of $H$.
	\end{exmp}
	
	\begin{exmp} Given a groupoid $\mathcal{G}$, denote by $k\mathcal{G}$ the $k$-coalgebra with the morphisms of $\Gg$ as a base of grouplike elements and $C=k\mathcal{G}^{(0)}$ the coalgebra with the objects of $\Gg$ as a base of grouplike elements.
The morphisms $\alpha$ and $\widetilde{\beta}$ are given by the linearization of the target map $t$, while $\widetilde{\alpha}$ and $\beta$ are given by the linearization of the source map $s$. One can easily verify that $k\mathcal{G} \square^C k\mathcal{G} = k\mathcal{G} \square^{\widetilde{C}} k\mathcal{G} =k(\mathcal{G} {}_s \times_t \mathcal{G})$. The left and the right multiplication maps are the same and are induced by the multiplication in the groupoid and the left and right unit maps are given by the inclusion of the coalgebra $k\mathcal{G}^{(0)}$ into the coalgebra $k\mathcal{G}$. Finally, the antipode map is given by the inversion on the groupoid, that is $\mathcal{S}(\delta_g)=\delta_{g^{-1}}$. The verification of the axioms is straightforward. Remark that in contrast to the construction of a Hopf algebroid structure on a groupoid algebra, in the coalgebroid case no finiteness on the number of objects of $\Gg$ is needed.
	\end{exmp}
	
	\begin{exmp}
	It is well-known that weak Hopf algebras lead to examples of Hopf algebroids.  
Already in \cite{BM} it was observed that examples of bicoalgebroids can be constructed from weak bialgebras. The same construction allows to view weak Hopf algebras as Hopf coalgebroids. One can notice that the base coalgebra of the Hopf coalgebroid associated to a weak Hopf algebra $H$, and the base algebra of the Hopf algebroid structure associated to the same weak Hopf algebra $H$ are dual to each other, which is in correspondence with the fact that this base (co)algebra is separable Frobenius. The previous example is a special case of this construction in case the groupoid has a finite number of objects.
	\end{exmp}
	
	\begin{exmp} 
	Given a Hopf algebra $H$ and a coalgebra $C$, the coalgebra $\mathcal{H}=C\otimes H \otimes C^{cop}$ is a Hopf coalgebroid, with the usual coalgebra structure on the tensor product and with $\widetilde{C}=C^{cop}$. In this case, for any $c\otimes h\otimes d \in \mathcal{H}$ we have
		\[
		\alpha (c\otimes h\otimes d)=\widetilde{\beta} (c\otimes h\otimes d)= c\epsilon (h)\epsilon (d) ,
		\]
		and 
		\[
		\beta (c\otimes h\otimes d)=\widetilde{\alpha} (c\otimes h\otimes d)= d\epsilon (h)\epsilon (c) .
		\] 
		The left and right multiplication (which coincide) can defined globally on $\mathcal{H}\otimes \mathcal{H}$ by means of the following formula:
		\[
		\mu ((c\otimes h \otimes d)\otimes (c'\otimes h'\otimes d'))= c\epsilon (c')\otimes hh' \otimes d'\epsilon (d) .
		\]
		The left unit is given by $\eta_L (c)=c_{(1)} \otimes 1_H \otimes c_{(2)}$, and the right unit is given by $\eta_R (c)=c_{(2)} \otimes 1_H \otimes c_{(1)}$, for every $c\in C$. The antipode map is written as $\mathcal{S}(c\otimes h \otimes d)=d\otimes S(h) \otimes c$. We leave to the reader the verification of the axioms of Hopf-coalgebroid structure.
	\end{exmp}
	
\subsection{The Hopf-coalgebroid structure on the universal coalgebra $H^{par}$}

Let $H$ be a Hopf algebra over a field $k$ with bijective antipode. 
Recall from the proof of Proposition \ref{CisotildeC} the existence of a coalgebra morphism
$\sigma :(H^{par})^{cop} \rightarrow H^{par}$ making the following diagram commutative
\[
\xymatrix{
(H^{par})^{cop} \ar[rr]^-\sigma \ar[d]_{p} && H^{par} \ar[d]^p\\
H^{cop} \ar[rr]^-{S^{-1}} && H
}
\]

Then consider the coalgebra map $\alpha =\overline{E} :H^{par} \rightarrow C$, and the anti-coalgebra map $\beta =\overline{E}\circ \sigma :H^{par} \rightarrow C$.  As explained in the previous section, these maps endow $H^{par}$ with a left and a right $C$-comodule structures, given by the maps
		\[
		\begin{array}{rrclrrcl} \lambda :& H^{par} & \rightarrow & C\otimes H^{par} \qquad\qquad & \rho :& H^{par} & \rightarrow & H^{par} \otimes C\\
		\, & x & \mapsto & \ove (x_{(1)}) \otimes x_{(2)} & & x & \mapsto & x_{(2)}  \otimes  \ove \circ \sigma (x_{(1)})
		\end{array} .
		\]
On the other hand, we have the coalgebra map $\widetilde{\alpha} =\overline{\widetilde{E}} :H^{par} \rightarrow \widetilde{C}$, and the anti-coalgebra map $\widetilde{\beta} =\overline{\widetilde{E}}\circ \sigma :H^{par} \rightarrow \widetilde{C}$. These induce a left and right $\widetilde{C}$-comodule structure on $H^{par}$ by means of the maps
		\[
		\begin{array}{rrclrrcl} \widetilde{\lambda} :& H^{par} & \rightarrow & \widetilde{C}\otimes H^{par} \qquad 
		\qquad &  \widetilde{\rho} :& H^{par} & \rightarrow & H^{par} \otimes \widetilde{C}\\
		\, & x & \mapsto & \overline{\widetilde{E}}\circ \sigma (x_{(2)}) \otimes x_{(1)} & & x & \mapsto & x_{(1)}  \otimes  \overline{\widetilde{E}} (x_{(2)})
		\end{array} .
		\]

\begin{lemma}
With notation and coactions as above, the maps
\[
\begin{array}{rccl} \widehat{\mu}_L : & H^{par} \boxtimes^{C}_l H^{par} & \rightarrow & H \\
\, & \sum_i x^i \boxtimes y^i & \mapsto & \sum_i p(x^i) p(y^i)
\end{array}
\]
and 
\[
\begin{array}{rccl} \widehat{\mu}_R : & H^{par} \boxtimes^{\widetilde{C}}_r H^{par} & \rightarrow & H \\
\, & \sum_i x^i \boxtimes y^i & \mapsto & \sum_i p(x^i) p(y^i)
\end{array}
\]
are partial corepresentations.		
\end{lemma}

\begin{proof}
We only prove the statement for $\widehat\mu_L$, the proof for $\widehat\mu_R$ follows by left-right duality.

Let us first verify that the map $\widehat\mu_L$ is well-defined. Therefore, we need to verify the equality 
\[
\sum_i p(x^i)p(y^i_{(1)}) \otimes \overline{E}(y^i_{(2)}) = \sum_i p(x^i_{(1)})p(y^i) \otimes \overline{E} \circ \sigma (x^i_{(2)}) 
\]
in $H\ot C$. Using the fact that the maps $p_n:C\to H^{\ot n}, p_n(x)=p(x_{(1)})\ot \cdots \ot p(x_{(n)})$ are jointly injective (see \seref{cofree}), it suffices to prove that this identity holds after composition with $Id\ot p_n$ for all $n$.
\begin{eqnarray*}
& &\hspace*{-2cm} \sum_i p(x^i)p(y^i_{(1)}) \ot p_n(\overline{E}(y^i_{(2)}))=\\
& = & \sum_i p(x^i) p(y^i_{(1)}) \ot E(y^i_{(2)}) \ot \cdots \ot E(y^i_{(n+1)})\\
& \stackrel{\ref{lemadosE}(iii)}= & \sum_i p(x^i) E(y^i_{(1)})_{(1)} p (y^i_{(2)})\ot E(y^i_{(1)})_{(2)} \ot E(y^i_{(3)}) \ot \cdots \ot E(y_{(n+1)}) \\
& \stackrel{\square}{=} & \sum _i p(x^i_{(2)}) S^{-1}(\widetilde{E}(x^i_{(1)}))_{(1)} p (y^i_{(1)}) \ot
 S^{-1}(\widetilde{E}(x^i_{(1)}))_{(2)} \ot E(y^i_{(2)})\ot \cdots \ot E(y^i_{(n)}) \\
& = & \sum_i p(x^i_{(2)}) S^{-1}(\widetilde{E}(x^i_{(1)})_{(2)}) p (y^i_{(1)}) \ot S^{-1}(\widetilde{E}(x^i_{(1)})_{(1)}) \ot E(y^i_{(2)}) \ot \cdots \ot E(y^i_{(n)})\\
& \stackrel{\eqref{lemadosE}(vi)}= & \sum_i p(x^i_{(1)}) p (y^i_{(1)}) \ot S^{-1}(\widetilde{E}(x^i_{(2)})) \ot E(y^i_{(2)}) \ot \cdots \ot E(y^i_{(n)}) \\
& = & \cdots \\
& = & \sum_i p(x^i_{(1)}) p (y^i) \ot S^{-1}(\widetilde{E}(x^i_{(n+1)}))  \ot S^{-1}(\widetilde E(x^i_{(n)})) \ot \cdots 
\ot S^{-1}(\widetilde E(x^i_{(2)})) \\
& = & \sum_i p(x^i_{(1)}) p (y^i) \ot p_n(\overline{E} \circ \sigma (x^i_{(2)})) 
\end{eqnarray*}
in which the symbol $\square$ on the top of the third equality means that the element $\sum_i x^i \otimes y^i$ belongs to the cotensor product $H^{par} \square^C H^{par}$, that is, the 
following identity holds:  
\[
\sum_i x^i_{(2)} \otimes \overline{E}\circ \sigma (x^i_{(1)}) \otimes y^i =\sum_i x^i\otimes \overline{E} (y^i_{(1)}) \otimes y^i_{(2)}.
\]
In the last equality, we used the identity $p\circ \ol E\circ \sigma=S^{-1} \circ \widetilde E$ (see the proof of Proposition \ref{CisotildeC})
		
Now, we have to prove that this map is a partial corepresentation of $H$ relative to the coalgebra $H^{par} \boxtimes^{C}_l H^{par}$. It is easy to see that axiom  (\ref{def:PC1}) is satisfied, so let us verify axiom (\ref{def:PC2}). Taking $\sum_i x^i \boxtimes y^i \in H^{par} \boxtimes^{C}_l H^{par}$,
\begin{eqnarray*}
&& \hspace*{-2cm} \sum_i \widehat{\mu}_L ((x^i \boxtimes y^i)_{(1)})_{(1)} \otimes \widehat{\mu}_L ((x^i \otimes y^i)_{(1)})_{(2)} S(\widehat{\mu}_L ((x^i \otimes y^i)_{(2)})) =\\
& = & \sum_i p(x^i_{(1)})_{(1)} p(y^i_{(1)})_{(1)} \otimes p(x^i_{(1)})_{(2)} p(y^i_{(1)})_{(2)} S(p(y^i_{(2)})) S(p(x^i_{(2)}))\\
& \stackrel{\eqref{def:PC2}}= & \sum_i p(x^i_{(1)})_{(1)} p(y^i_{(1)}) \otimes p(x^i_{(1)})_{(2)} p(y^i_{(2)}) S(p(y^i_{(3)})) S(p(x^i_{(2)}))\\
& = & \sum_i p(x^i_{(1)})_{(1)} E(y^i_{(1)})p(y^i_{(2)}) \otimes p(x^i_{(1)})_{(2)} p(y^i_{(3)}) S(p(y^i_{(4)})) S(p(x^i_{(2)}))\\
& \stackrel{\square}= & \sum_i p(x^i_{(2)})_{(1)} S^{-1}(\widetilde{E}(x^i_{(1)})) p(y^i_{(1)}) \otimes p(x^i_{(2)})_{(2)} p(y^i_{(2)}) S(p(y^i_{(3)})) S(p(x^i_{(3)}))\\
& = & \sum_i p(x^i_{(3)})_{(1)} S^{-1}(p(x^i_{(2)})) p(x^i_{(1)}) p(y^i_{(1)}) \otimes p(x^i_{(3)})_{(2)} p(y^i_{(2)}) S(p(y^i_{(3)})) S(p(x^i_{(4)}))\\
& \stackrel{\eqref{def:PC3}}= & \sum_i S^{-1}(p(x^i_{(2)})S(p(x^i_{(3)}))) p(x^i_{(1)}) p(y^i_{(1)}) \otimes p(x^i_{(4)}) p(y^i_{(2)}) S(p(y^i_{(3)})) S(p(x^i_{(5)}))\\
& = & \sum_i p(x^i_{(3)}) S^{-1}(p(x^i_{(2)})) p(x^i_{(1)}) p(y^i_{(1)}) \otimes p(x^i_{(4)}) p(y^i_{(2)}) S(p(y^i_{(3)})) S(p(x^i_{(5)}))\\
& = & \sum_i p(x^i_{(2)}) S^{-1}(\widetilde{E}(x^i_{(1)})) p(y^i_{(1)}) \otimes p(x^i_{(3)}) p(y^i_{(2)}) S(p(y^i_{(3)})) S(p(x^i_{(4)}))\\
& \stackrel{\square}= & \sum_i p(x^i_{(1)}) E(y^i_{(1)}) p(y^i_{(2)}) \otimes p(x^i_{(2)}) p(y^i_{(3)}) S(p(y^i_{(4)})) S(p(x^i_{(3)}))\\
& = & \sum_i p(x^i_{(1)}) p(y^i_{(1)}) \otimes p(x^i_{(2)}) p(y^i_{(2)}) S(p(x^i_{(3)})p(y^i_{(3)}))\\
& = & \sum_i \widehat{\mu}_L ((x^i \boxtimes y^i)_{(1)}) \otimes \widehat{\mu}_L ((x^i \boxtimes y^i)_{(2)}) S(\widehat{\mu}_L 
((x^i \boxtimes y^i)_{(3)})) .
\end{eqnarray*}
Again, the symbol $\square$ on the top of equalities 4 and 9 means that $\sum_i x^i \otimes y^i \in H^{par} \square^C H^{par}$. 
Axiom  (\ref{def:PC3}) can be checked by a similar computation. 
\end{proof}
		
From the previous lemma, it follows that there exist well-defined coalgebra maps
$$\overline{\mu}_L :H^{par} \boxtimes^{C}_l H^{par}  \rightarrow  H^{par} , 
\qquad \qquad 
\overline{\mu}_R : H^{par} \boxtimes^{\widetilde{C}}_r H^{par}  \rightarrow  H^{par},$$
such that $p\circ \overline{\mu}_L =\widehat{\mu}_L$ and $p\circ \ol\mu_R=\widehat\mu_R$.

Now, let us define the left unit map $\eta_L :C\rightarrow H^{par}$ given by
		\[
		\eta_L (\ove (x)) =\ovomega (\ove (x) \cosmash 1_H ) ,
		\]
		in which $\ovomega : C\cosmash H \rightarrow H^{par}$ is the coalgebra isomorphism defined in Theorem~\ref{isoHparCosmash}. Similarly, the right unit, $\eta_R :\widetilde{C} \rightarrow H^{par}$, is given by
		\[
		\eta_R (\overline{\widetilde{E}} (x)) =\overline{\widetilde{\omega}}_0 (1_H \rcosmash \overline{\widetilde{E}} (x)) .
		\]
		The coalgebra isomorphism map $\overline{\widetilde{\omega}}_0 :H\rcosmash \widetilde{C} \rightarrow H^{par}$ was defined at the end of \seref{universals}.
		
Finally, let us define the antipode on $H^{par}$. The linear map $S\circ p:\left( H^{par} \right)^{cop} \rightarrow H$, being a composition of a coalgebra map and a partial corepresentation, is again a partial corepresentation. Hence it induces a morphism of coalgebras $\mathcal{S} : \left( H^{par} \right)^{cop}\rightarrow H^{par}$ such that $p\circ \mathcal{S} =S\circ p$.
		
With the above notations, we have the following theorem:
\begin{thm} Consider a Hopf algebra $H$ with invertible antipode, then there is a Hopf coalgebroid structure $(H^{par}, C,\alpha ,\beta , \mu_L ,\eta_L , \widetilde{C} , \widetilde{\alpha}, \widetilde{\beta} , \mu_R ,\eta_R , \mathcal{S})$ on the universal coalgebra $H^{par}$.

Consequently, the category of $\Mod^{H^{par}}$ of $H^{par}$-comodules (which is isomorphic to the category of regular partial $H$-comodules) is closed monoidal and allows a closed monoidal forgetful functor to $({^C\Mod^C},\square^C,C)$. 
\end{thm}

It can be verified directly that the maps defined above indeed satisfy all axioms of a Hopf coalgebroid. However this proof is very long and technical, therefore we provide here an alternative approach, which also explains better the duality with the Hopf algebroid structure on $H_{par}$, but which is however only valid in the case that $H^\circ$ is dense in $H^*$.

	In order to provide a proof of the above claims, we will need the following lemma.
	
	\begin{lemma} Let $H$ be a Hopf algebra, for each $x\in H^{par}$ we have
		\begin{eqnarray}\label{deltancosmashum}
		\Delta^{n-1} (\ove (x) \cosmash 1_H ) & = & (\ove (x_{(1)}) \cosmash p(x_{(2)}) S(p(x_{(3n-2)}))) \otimes \cdots \otimes (\ove (x_{(2n-3)}) \cosmash p(x_{(2n-2)}) S(p(x_{(2n)}))) \otimes \nonumber\\
		& \, & \qquad \otimes (\ove (x_{(2n-1)}) \cosmash 1_H)
		\end{eqnarray}
		and	
		\begin{eqnarray}\label{deltanrcosmashum}
		\Delta^{n-1} (1_H \rcosmash \overline{\widetilde{E}} (x) ) & = & (1_H \rcosmash \overline{\widetilde{E}}(x_{(n)})) \otimes (S(p(x_{(n-1)})) p(x_{(n+1)}) \rcosmash \overline{\widetilde{E}}(x_{(n+2)})) \otimes \cdots \otimes \nonumber \\
		& \, & \qquad \otimes (S(p(x_{(1)})) p(x_{(3n-3)}) \rcosmash \overline{\widetilde{E}}(x_{(3n-2)})).
		\end{eqnarray}
	\end{lemma}
	
	\begin{proof} The proof goes by induction over $n$. For $n=1$ we have
		\[
		\Delta (\ove (x) \cosmash 1_H ) =(\ove (x_{(1)}) \cosmash p(x_{2}) S(p(x_{(4)}))) \otimes (\ove (x_{(3)})) .
		\]
		Now, suppose (\ref{deltancosmashum}) valid for $n$ , then 
		\begin{eqnarray*}
		&&\Delta^{n} (\ove (x) \cosmash 1_H ) =\\
		& = & (\ove (x_{(1)}) \cosmash p(x_{(2)}) S(p(x_{(3n-2)}))) \otimes \cdots \otimes (\ove (x_{(2n-3)}) \cosmash p(x_{(2n-2)}) S(p(x_{(2n)}))) \otimes \Delta (\ove (x_{(2n-1)}) \cosmash 1_H) \\
		& = & (\ove (x_{(1)}) \cosmash p(x_{(2)}) S(p(x_{(3n+1)}))) \otimes \cdots \otimes (\ove (x_{(2n-3)}) \cosmash p(x_{(2n-2)}) S(p(x_{(2n+3)}))) \otimes\\
		&& \qquad \otimes (\ove (x_{(2n-1)}) \cosmash p(x_{(2n)}) S(p(x_{(2n+2)}))) \otimes (\ove (x_{(2n+1)}) \cosmash 1_H) \\
		& = & (\ove (x_{(1)}) \cosmash p(x_{(2)}) S(p(x_{(3(n+1)-2}))) \otimes \cdots  \otimes (\ove (x_{(2(n+1)-3)}) \cosmash p(x_{(2(n+1)-2)}) S(p(x_{(2(n+1)}))) \otimes \\
		&& \qquad \otimes (\ove (x_{(2(n+1)-1)}) \cosmash 1_H) .
		\end{eqnarray*}
		Then (\ref{deltancosmashum}) is valid for all $n\geq 1$. 
		
		Analogously, we have the identity (\ref{deltanrcosmashum}).
	\end{proof}

\begin{proposition}
Let $H$ be a Hopf algebra with bijective antipode such that $H^\circ$ is dense in $H^*$. Then all the structure maps on $H^{par}$ defined above are adjoint to the corresponding structure maps on $(H^\circ)_{par}$ with respect to the dual pairings between $H^{par}$ and $(H^\circ)_{par}$ and between $C=\ol E(H^{par})$ and $A_{par}(H^\circ)$. 

Consequently, $H^{par}$ is indeed a Hopf coalgebroid.
\end{proposition}

\begin{proof}
(1). By (\ref{sourcedual}) the map $\alpha =\overline{E} :H^{par} \rightarrow C$ is adjoint to the left source map $s_L :A_{par} (H^{\circ}) \rightarrow H^{\circ}_{par}$, that is
			\[
			\langle \! \langle \overline{E}(x) , a \rangle \! \rangle =\left( x, s_L (a) \right ) . 
			\]		
(2). The anti-coalgebra map $\beta =\overline{E} \circ \sigma :H^{par} \rightarrow C$ is adjoint to the left target map $t_L :A_{par} (H^{\circ}) \rightarrow H^{\circ}_{par}$. Indeed, taking $x\in H^{par}$ and $a=\varepsilon_{h_{1*}} \cdots \varepsilon_{h_{n*}}$, we have
			\begin{eqnarray*}
			\langle \! \langle \overline{E} \sigma (x) , a \rangle \! \rangle & = & \langle \! \langle \overline{E} \sigma (x) , \varepsilon_{h^{1*}} \cdots \varepsilon_{h^{n*}} \rangle \! \rangle 
			 =  {h^{1*}} (p(\overline{E} \sigma (x_{(n)}))) \cdots {h^{n*}} (p(\overline{E} \sigma (x_{(1)})))  \\
			& = & {h^{1*}} (S^{-1}(\widetilde{E}(x_{(n)}))) \cdots {h^{n*}} (S^{-1}(\widetilde{E}(x_{(1)})))  \\
			& = & {h^{1*}} (S^{-1} (p(x_{(2n)})) p(x_{(2n-1)})) \cdots {h^{n*}} (S^{-1} (p(x_{(2)})) p(x_{(1)}))  \\
			& = & \left( x, [h^{n*}_{(2)}][S^{-1}(h^{n*}_{(1)})] \cdots [h^{1*}_{(2)}][S^{-1}(h^{1*}_{(1)})] \right)  \\
			& = & \left( x, \widetilde{\varepsilon}_{h^{n*}} \cdots \widetilde{\varepsilon}_{h^{1*}} \right)  
			 =  \left( x, t_L (a) \right) .
			\end{eqnarray*}
(3). By a similar reasoning, we have that $\widetilde\alpha$ and $\widetilde\beta$ are the respective adjoints of the right source and target maps. More precisely,
			\[
			\bkk{\overline{\widetilde{E}} (x) , a} =\left( x, s_R (a) \right),
			\qquad
			\bkk{ \overline{\widetilde{E}} \circ \sigma (x) , a} =\left( x, t_R (a) \right) 
			\]
			for every $x\in H^{par}$ and for every $a=\widetilde{\varepsilon}_{h^{1*}} \cdots \widetilde{\varepsilon}_{h^{n*}} \in \widetilde{A}_{par} (H^{\circ})$. 			\\
(4). The left unit $\eta_L :C\rightarrow H^{par}$, defined as $\eta_L (\overline{E}(x))=\ovomega (\ove (x) \cosmash 1_H)$ is adjoint to the left counit $\epsilon_L :H^{\circ}_{par} \rightarrow A_{par}(H^{\circ})$, defined as $\epsilon_L ([h^{1*}]\cdots [h^{n*}])=\varepsilon_{h^{1*}_{(1)}} \varepsilon_{h^{1*}_{(2)} h^{2*}_{(1)}} \cdots \varepsilon_{h^{1*}_{(n)} h^{2*}_{(n-1)}\cdots h^{n*}}$. Indeed, on one hand, we have
			\newcommand{\llangle}{\left\langle\!\left\langle}
			\newcommand{\rrangle}{\right\rangle\!\right\rangle}
			\begin{eqnarray*}
			&&\hspace{-1cm}\left(\eta\left(\overline{E}(x)\right), \left[h^{1*}\right] \cdots \left[h^{n*}\right]\right) 
			 =  \left(\ol{\omega_0}\left(\overline{E}(x) \cosmash 1\right), \left[h^{1*}\right] \cdots \left[h^{n*}\right]\right)\\
			& = & \left(\ol{\omega_0}\left(\overline{E}(x) \cosmash 1\right)_{(1)}, \left[h^{1*}\right]\right) \cdots \left(\ol{\omega_0}\left(\overline{E}(x) \cosmash 1\right)_{(n)}, \left[h^{n*}\right]\right)\\
			& = & \left(\ol{\omega_0}\left(\left(\overline{E}(x) \cosmash 1\right)_{(1)}\right), \left[h^{1*}\right]\right) \cdots \left(\ol{\omega_0}\left(\left(\overline{E}(x) \cosmash 1\right)_{(n)}\right), \left[h^{n*}\right]\right)\\
			& = & h^{1*}\left({\omega_0}\left(\left(\overline{E}(x) \cosmash 1\right)_{(1)}\right)\right) \cdots 
			h^{n*}\left({\omega_0}\left(\left(\overline{E}(x) \cosmash 1\right)_{(n)}\right)\right)\\
			& \stackrel{\eqref{deltancosmashum}}{=} &
			h^{1*}(\omega_0 (\ove (x_{(1)}) \cosmash p(x_{(2)}) S(p(x_{(3n-2)}))) ) \cdots h^{(n-1)*}(\omega_0 (\ove (x_{(2n-3)}) \cosmash p(x_{(2n-2)}) S(p(x_{(2n)}))) ) \nonumber \\
			& \,  & \qquad  h^{n*} (\omega_0 (\ove (x_{(2n-1)}) \cosmash 1_H)) \\
			& = & h^{1*}\left(p(x_{(1)}) S(p(x_{(2n-1)}))\right) \cdots
			h^{(n-1)*}\left(p(x_{(n-1)}) S(p(x_{(n+1)}))\right)
			h^{n*}\left(E(x_{(n)})\right) .
			\end{eqnarray*}
			On the other hand
			\begin{eqnarray*}
			&&\hspace{-1cm}\llangle \overline{E}(x), \ul{\epsilon}\left( \left[h^{1*}\right] \cdots \left[h^{n*}\right] \right) \rrangle = 
			 \llangle \overline{E}(x), \varepsilon_{{h^{1*}}_{{(1)}}} \varepsilon_{{h^{1*}}_{(2)}{h^{2*}}_{(1)}} \cdots \varepsilon_{{h^{{1*}}}_{(n)}{h^{2*}}_{(n-1)} \cdots {h^{(n-1)*}}_{(2)}{h^{n*}}} \rrangle\\
			& = &
			\llangle \overline{E}(x)_{(1)}, \varepsilon_{{h^{1*}}_{{(1)}}}\rrangle
			\llangle \overline{E}(x)_{(2)}, \varepsilon_{{h^{1*}}_{(2)}{h^{2*}}_{(1)}}\rrangle \cdots
			\llangle \overline{E}(x)_{(n)}, \varepsilon_{{h^{{1*}}}_{(n)}{h^{2*}}_{(n-1)} \cdots {h^{(n-1)*}}_{(2)}{h^{n*}}} \rrangle\\
			& = &
			{{h^{ 1*}}_{{(1)}}} \left( E (x_{(1)}) \right)
			{h^{ 1*}}_{(2)}{h^{2*}}_{(1)} \left( E (x_{(2)}) \right) \cdots
			{h^{{1*}}}_{(n)}{h^{2*}}_{(n-1)} \cdots {h^{(n-1)*}}_{(2)}{h^{n*}} \left( E (x_{(n)}) \right)\\
			& = &
			h^{1*} \left( E (x_{(1)}) E (x_{(2)})_{(1)} \cdots E (x_{(n)})_{(1)} \right)
			{h^{2*}}_{(1)} \left( E (x_{(2)})_{(2)} \right) \cdots
			{h^{2*}}_{(n-1)} \cdots {h^{(n-1)*}}_{(2)}{h^{n*}} \left( E (x_{(n)})_{(2)} \right)\\
			& = &
			h^{1*} \left( p (x_{(1)}) S(p(x_{(2)})) E (x_{(3)})_{(1)} \cdots E (x_{(n+1)})_{(1)} \right)
			{h^{2*}}_{(1)} \left( E (x_{(3)})_{(2)} \right) \cdots \nonumber \\
			& \, & \qquad \cdots {h^{2*}}_{(n-1)} \cdots {h^{(n-1)*}}_{(2)}{h^{n*}} \left( E (x_{(n+1)})_{(2)} \right)\\
			& \stackrel{\ref{lemadosE}(xiv)}{=} &
			h^{1*} \left( p(x_{(1)}) S(p(x_{(n+1)})) \right)
			{h^{2*}}_{(1)} \left( E (x_{(2)}) \right)
			{h^{2*}}_{(2)} {h^{3*}}_{(1)} \left( E (x_{(3)}) \right) \cdots\\ 
			&& \qquad \cdots
			{h^{2*}}_{(n-1)} \cdots {h^{(n-1)*}}_{(2)}{h^{n*}} \left( E (x_{(n)}) \right)\\
			& = &
			h^{1*} \left( p(x_{(1)}) S(p(x_{(n+1)})) \right)
			{h^{2*}}\left( E (x_{(2)}) E (x_{(3)})_{(1)} \cdots E (x_{(n)})_{(1)} \right)
			{h^{3*}}_{(1)} \left( E (x_{(3)})_{(2)} \right) \cdots\\
			&& \qquad \cdots
			{h^{3*}}_{(n-2)} \cdots {h^{(n-1)*}}_{(2)}{h^{n*}} \left( E (x_{(n)})_{(2)} \right)\\
			& = &
			h^{1*} \left( p(x_{(1)}) S(p(x_{(n+1)})) \right)
			h^{2*} \left( p(x_{(2)}) S(p(x_{(3)})) E (x_{(4)})_{(1)} \cdots E (x_{(n+1)})_{(1)} \right)
			{h^{3*}}_{(1)} \left( E (x_{(4)})_{(2)} \right) \cdots\\
			&& \qquad \cdots
			{h^{3*}}_{(n-2)} \cdots {h^{(n-1)*}}_{(2)}{h^{n*}} \left( E (x_{(n+1)})_{(2)} \right)\\
			& \stackrel{\ref{lemadosE}(xiv)}{=} &
			h^{1*} \left( p(x_{(1)}) S(p(x_{(n+2)})) \right)
			h^{2*} \left( p(x_{(2)}) S(p(x_{(n+1)})) \right)
			{h^{*3}}_{(1)} \left( E (x_{(3)}) \right) \cdots\\
			&& \qquad \cdots
			{h^{3*}}_{(n-2)} \cdots {h^{(n-1)*}}_{(2)}{h^{n*}} \left( E (x_{(n)}) \right)\\
			& \vdots &\\
			& = &
			h^{1*}\left(p(x_{(1)}) S(p(x_{(2n-1)}))\right)
			h^{2*} \left( p(x_{(2)}) S(p(x_{(2n-2)})) \right) \cdots
			h^{(n-1)*}\left(p(x_{(n-1)}) S(p(x_{(n+1)}))\right)
			h^{n*}\left(E(x_{(n)})\right).
			\end{eqnarray*}
(5). By a similar argument, we show that the right unit $\eta_R :\widetilde{C} \rightarrow H^{par}$, given by 
			$\eta_R (\overline{\widetilde{E}} (x)) =\overline{\widetilde{\omega}}_0 (1_H \rcosmash \overline{\widetilde{E}} (x))$ is adjoint to the right counit $\epsilon_R :H^{\circ}_{par} \rightarrow \widetilde{A}_{par}(H^{\circ})$, defined as $\epsilon_L ([h^{1*}]\cdots [h^{n*}])=\widetilde{\varepsilon}_{h^{1*} h^{2*}_{(1)} \cdots h^{n*}_{(1)}} \widetilde{\varepsilon}_{h^{2*}_{(2)} \cdots  h^{n*}_{(1)}} \cdots \widetilde{\varepsilon}_{h^{n*}_{(n)}}$.\\
(6). For the multiplication map, in Proposition \ref{dualpairing} we have that the category $\LDP_k$ is monoidal and the forgeful functors are strictly monoidal. Moreover, in the same proposition it is shown that the forgetful functor $R:\LDP_k\to \Vect_k^{op}$ creates kernels. Hence, the pairing $(H^{par}\square^C H^{par} , H^{\circ}_{par} \otimes_A H^{\circ}_{par} , \left( \! \left( \; ,\; \right) \!\right) )$, in which 
			\[
			\left( \! \left( \sum_i x_i \otimes y_i , \zeta \otimes \xi \right) \!\right) =\sum_i \left( x_i , \xi \right) \left( y_i , \zeta \right) ,
			\]
			is left non-degenerate and it is the equalizer for the pair of morphisms $(\rho \square^C \mbox{Id} , \triangleleft \otimes_A \mbox{Id})$ and $(\mbox{Id} \square^C \lambda , \mbox{Id} \otimes_A \triangleright)$, both from $(H^{par}\otimes H^{par} , H^{\circ}_{par} \otimes H^{\circ}_{par} , \left( \! \left( \; ,\; \right)\! \right) )$ to $(H^{par}\otimes C \otimes  H^{par} , H^{\circ}_{par} \otimes A_{par}(H^{\circ}) \otimes H^{\circ}_{par} , \left( \! \left( \! \left( \; ,\; \right)\! \right) \! \right) )$, in which
			\[
			\left( \! \left( \! \left( x\otimes \overline{E}(y) \otimes z , \zeta \otimes a \otimes \xi \right)\! \right) \! \right) =\left( x, \xi \right) \langle \! \langle \overline{E}(y) , a \rangle \! \rangle \left( z , \zeta \right) .
			\]
			We will now show that the left multiplication $\mu_L :H^{par} \square^C H^{par} \rightarrow H^{par}$ is adjoint to the left comultiplication $\Delta_L : H^{\circ}_{par} \rightarrow H^{\circ}_{par} \otimes_A H^{\circ}_{par}$. Consider $\sum_i x_i \otimes y_i \in H^{par} \square^C H^{par}$ and $[h^{1*}] \cdots [h^{n*}] \in H^{\circ}_{par}$, then
			\begin{eqnarray}
			\left( \mu_L \left( \sum_i x^i \otimes y^i \right) , [h^{1*}] \cdots [h^{n*}] \right)  
			& = & \sum_i h^{1*} \left(p\mu_L \left( \sum_i x^i_{(1)} \otimes y^i_{(1)} \right) \right) \cdots h^{n*} \left(p\mu_L \left( \sum_i x^i_{(n)} \otimes y^i_{(n)} \right)\right) \nonumber \\
			& = & \sum_i h^{1*} (p(x^i_{(1)} ) p(y^i_{(1)} )) \cdots h^{n*} (p(x^i_{(n)} ) p(y^i_{(n)} )) \nonumber \\
			& = & \sum_i h^{1*}_{(1)} (p(x^i_{(1)} )) h^{1*}_{(2)}(p(y^i_{(1)} )) \cdots h^{n*}_{(1)} (p(x^i_{(n)} )) h^{n*}_{(2)}( p(y^i_{(n)} )) \nonumber \\
			& = & \sum_i \left( x^i , [h^{1*}_{(1)}] \cdots [h^{n*}_{(1)}] \right) \left( y^i , [h^{1*}_{(2)}] \cdots [h^{n*}_{(2)}] \right) \nonumber \\
			& = & \left( \! \left( \sum_i x^i \otimes y^i , \Delta_L ([h^{1*}] \cdots [h^{n*}]) \right) \! \right) .\nonumber  
			\end{eqnarray}
In the same way, one shows that the right multiplication $\mu_R$ on $H^{par}$ is adjoint to the right comultiplication $\Delta_R$ on $(H^\circ)_{par}$.	
(7). We will now show that the antipodes of $H^{par}$ and $(H^\circ)_{par}$ are also adjoint. 
Consider an element $x\in H^{par}$ and $[h^{1*}]\cdots [h^{n*}]\in H^{\circ}_{par}$, then
		\begin{eqnarray}
		\left( \mathcal{S}(x),  [h^{1*}]\cdots [h^{n*}] \right) & = & h^{1*} (p(\mathcal{S}(x_{(n)}))) \cdots h^{n*} (p(\mathcal{S}(x_{(1)}))) \nonumber \\
		& = & h^{1*} (S(p(x_{(n)}))) \cdots h^{n*} (S(p(x_{(1)}))) \nonumber \\
		& = & S(h^{1*})(p(x_{(n)})) \cdots S(h^{n*})(p(x_{(1)})) \nonumber \\
		& = & \left( x, [S(h^{n*})] \cdots [S(h^{1*})] \right) \nonumber \\
		& = & \left( x , \mathcal{S}^{*}([h^{1*}]\cdots [h^{n*}])\right) ,
		\end{eqnarray}
		in which $\mathcal{S}^*$ is the antipode in the Hopf algebroid $H^{\circ}_{par}$.

From the above, it follows that we have the following morphisms in the category $\LDP_k$
\[\begin{array}{rrl}
(\overline{E} , s_L):&(H^{par}, H^{\circ}_{par}, \left( \; , \; \right) ) \rightarrow & (C, A_{par} (H^{\circ}), \langle \! \langle \; , \; \rangle \! \rangle ) ;\\
(\overline{E} \circ \sigma , t_L):&(H^{par}, H^{\circ}_{par}, \left( \; , \; \right) ) \rightarrow& (C, A_{par} (H^{\circ}), \langle \! \langle \; , \; \rangle \! \rangle );\\
(\overline{\widetilde{E}} , s_R) :&(H^{par}, H^{\circ}_{par}, \left( \; , \; \right) ) \rightarrow& (\widetilde{C} , \widetilde{A}_{par} (H^{\circ}) , \bkk{\; , \; } ) ;\\
(\overline{\widetilde{E}} \circ \sigma , t_R) :&(H^{par}, H^{\circ}_{par}, \left( \; , \; \right) ) \rightarrow& (\widetilde{C} , \widetilde{A}_{par} (H^{\circ}) , \bkk{\; , \; } ) ;\\
(\eta_L ,\epsilon_L ):& (C, A_{par}(H^{\circ} ), \langle \! \langle \; , \; \rangle \! \rangle )\rightarrow& (H^{par}, H^{\circ}_{par}, \left( \; , \; \right) ) ;\\
(\eta_R ,\epsilon_R ):& (\widetilde{C}, \widetilde{A}_{par}(H^{\circ} ), [ \! [ \; , \; ] \! ] ) \rightarrow& (H^{par}, H^{\circ}_{par}, \left( \; , \; \right) ) ;\\
(\mu_L,\Delta_L):& (H^{par}\square^C H^{par} , H^{\circ}_{par} \otimes_A H^{\circ}_{par} , \left( \! \left( \; ,\; \right) \!\right) ) \to & (H^{par}, H^{\circ}_{par}, \left( \; , \; \right) ) ;\\
(\mu_R,\Delta_R):& (H^{par}\square^C H^{par} , H^{\circ}_{par} \otimes_A H^{\circ}_{par} , \left( \! \left( \; ,\; \right) \!\right) ) \to & (H^{par}, H^{\circ}_{par}, \left( \; , \; \right) ) ;\\
(\mathcal{S} , \mathcal{S}^*) :&((H^{par})^{cop} , (H^{\circ}_{par})^{op} , \left( \; , \; \right) ) \rightarrow& (H^{par} , H^{\circ}_{par} , \left( \; , \; \right) ).
\end{array}\]
Moreover, we know that the forgetful functor $R:\LDP_k\to \Vect_k$ sends the pair $(H^{par},(H^\circ)_{par})$ to the Hopf algebroid $(H^\circ)_{par}$, which can be viewed as an internal Hopf coalgebroid in the dual category $\Vect_k^{op}$. Since the
structure maps of the Hopf algebroid $(H^\circ)_{par}$ all arise by applying the functor $R$ to morphisms in $\LDP_k$ listed above, and since the axioms these structure maps have to satisfy can all be expressed in terms of commutative diagrams, it follows from the fact that the functor $R$ is faithful (see Proposition~\ref{dualpairing}) that $(H^{par},(H^\circ)_{par})$ is an internal Hopf coalgebroid in the category $LDP_k$ and therefore applying the functor $L:\LDP_k\to {\Vect_k}^{op}$ leads to the conclusion that $H^{par}$ is a Hopf coalgebroid in $\Vect_k$.
\end{proof}

\begin{remark}
		As an example of our approach, let us verify that $\alpha (x_{(1)})\otimes \beta (x_{(2)}) =\alpha (x_{(2)})\otimes \beta (x_{(1)})$ for every $x\in H^{par}$. Indeed, for any $b\otimes a \in A_{par}(H^{\circ}) \otimes A_{par}(H^{\circ})$
		\begin{eqnarray*}
		\langle \! \langle \! \langle 	\alpha (x_{(1)})\otimes \beta (x_{(2)}) , b\otimes a \rangle \! \rangle \! \rangle & = & 
		\langle \! \langle \! \langle 	\overline{E} (x_{(1)})\otimes \overline{E}\circ \sigma (x_{(2)}) , b\otimes a \rangle \! \rangle \! \rangle 
		 =  \langle \! \langle \overline{E} (x_{(1)}) , a \rangle \! \rangle \langle \! \langle \overline{E} \circ \sigma (x_{(2)}) , b \rangle \! \rangle \nonumber \\
		& = & \left( x_{(1)} , s_L (a) \right) \left( x_{(2)} , t_L (b) \right) 
		 =  \left( x, s_l (a) t_L (b) \right) \nonumber \\
		& = & \left( x, t_L (b) s_l (a) \right) 
		 =  \left( x_{(2)} , s_L (a) \right) \left( x_{(1)} , t_L (b) \right) \nonumber \\
		& = & \langle \! \langle \overline{E} (x_{(2)}) , a \rangle \! \rangle \langle \! \langle \overline{E} \circ \sigma (x_{(1)}) , b \rangle \! \rangle \nonumber 
		=
		\langle \! \langle \! \langle 	\overline{E} (x_{(2)})\otimes \overline{E}\circ \sigma (x_{(1)}) , b\otimes a \rangle \! \rangle \! \rangle\\ & = & \langle \! \langle \! \langle 	\alpha (x_{(2)})\otimes \beta (x_{(1)}) , b\otimes a \rangle \! \rangle \! \rangle ,\nonumber
		\end{eqnarray*}	
		from which the result follows by the (left) non-degeneracy of the pairing.
\end{remark}

	\begin{exmp} As we have seen in example \ref{examplekc2}, the coalgebra $(kC_2)^{par}$ is generated by three grouplike elements, $x_1$, $x_2$ and 
		$x_3$. It is easy to see that $E(x_1)=\widetilde{E}(x_1)=E(x_2)=\widetilde{E}(x_2)=\delta_e$ and that $E(x_3)=\widetilde{E}(x_3)=\frac{1}{2}(\delta_e +\delta_g)$. Then 
		\[
		\ove (x_1)=\ovesigma (x_1)=\overline{\widetilde{E}}(x_1) =\overline{\widetilde{E}}(x_1)=\ove (x_2)=\ovesigma (x_2)=\overline{\widetilde{E}}(x_2) =\overline{\widetilde{E}}(x_2) , 
		\]
		while for $x_3$ all the operators coincide, but not with the others. This implies that $x_1 \otimes x_1$, $x_1 \otimes x_2$, $x_2 \otimes x_1$ $x_2 \otimes x_2$ and $x_3 \otimes x_3$ are in the domain of both, the left and right multiplication and the multiplication table for the left and the right multiplication is
		\[
		\begin{array}{c|ccc} \cdot & x_1 & x_2 & x_3 \\
		\hline 
		x_1 & x_1 & x_2 & - \\
		x_2 & x_2 & x_1 & - \\
		x_3 & - & - & x_3
		\end{array}
		\]
		This multiplication table reflects exactly the structure of the groupoid $\Gamma (C_2)$, as defined in \cite{dok0}, making the identification 
		\begin{eqnarray*}
		x_1 & \mapsto &  \delta_{(e, \{ e ,g \})}\\
		x_2 & \mapsto &  \delta_{(g, \{ e ,g \})}\\
		x_3 & \mapsto &  \delta_{(e, \{ e \})} .
		\end{eqnarray*}
		We conjecture that for a finite group $G$, the universal Hopf coalgebroid $kG^{par}$ is exactly the Hopf coalgebroid associated to the groupoid $\Gamma (G)$.
	\end{exmp}

\subsection*{Acknowledgements}
M.M.S.A.\ is partially supported by CNPq (projects 306583/2016-0 and 309469/2019-8).
For this research, J.V.\ would like to thank the FWB (F\'ed\'eration Wallonie-Bruxelles) for financial support via the ARC-project ``Hopf algebras and the symmetries of non-commutative spaces'' as well as the FNRS for the MIS-project ``Antipode'' (Grant F.4502.18). 

All authors wish to thank the referee for the useful comments that improved the presentation of this paper and William Hautekiet for pointing out a mistake in Proposition~\ref{dualpairing}.

\end{document}